\documentclass[12pt,twoside]{amsart}

\usepackage[pagebackref]{hyperref}

\usepackage[english]{babel}
\usepackage[all]{xy}
\usepackage{amssymb,amsmath,bbm,url,xr,a4wide,cleveref,color}
\usepackage[mathscr]{euscript}
\usepackage{mathrsfs}
\usepackage[margin=1.2in]{geometry}
\usepackage{amsmath}
\usepackage{amsthm}
\usepackage{amssymb}
\usepackage{amscd}
\usepackage{mathrsfs}
\usepackage{graphicx}
\usepackage{rotating, graphicx}
\usepackage{multirow}
\usepackage{mathtools}
\usepackage{mathrsfs}
\usepackage{changepage} 
\usepackage{paralist}

\usepackage{kantlipsum}
\usepackage{enumitem}
\usepackage{times}
\usepackage{subfig}
\usepackage{color}
\calclayout

\usepackage{enumitem}
\makeatletter
\def\namedlabel#1#2{\begingroup
    #2%
    \def\@currentlabel{#2}%
    \phantomsection\label{#1}\endgroup
}
\makeatother

\title{Perverse homotopy heart and MW-modules}

\author{Fr\'ed\'eric D\'eglise, Niels Feld, and Fangzhou Jin}

\date{\number\day-\number\month-\number\year}

\newtheorem{thm}[subsubsection]{Theorem}
\newtheorem{prop}[subsubsection]{Proposition}
\newtheorem{lm}[subsubsection]{Lemma}
\newtheorem{cor}[subsubsection]{Corollary}
\newtheorem*{conj}{Conjecture}

\theoremstyle{remark}
\newtheorem{rem}[subsubsection]{Remark}

\newtheorem{ex}[subsubsection]{Example}

\theoremstyle{definition}
\newtheorem{df}[subsubsection]{Definition}
\newtheorem{num}[subsubsection]{}
\newtheorem{paragr}[subsubsection]{}

\numberwithin{equation}{subsubsection}

\newcommand{\T}{\mathscr T} 

\newcommand{\cO}{\mathcal O}
\newcommand{\cM}{\mathcal M} 

\DeclareMathOperator{\Der}{D}
\DeclareMathOperator{\DF}{DF}

\newcommand{\smod}[1]{#1\!-\!\operatorname{Mod}}
\newcommand{\MGL}{\mathbf{MGL}}
\newcommand{\HH}{\mathbf H}

\newcommand{\gr}{\mathrm{gr}}

\newcommand{\CechH}{{\check{\mathrm{H}}^0}}

\newcommand{\SH}{\mathrm{SH}}

\newcommand{\Mod}{\text{-}\mathrm{Mod}}
\newcommand{\ab}{\mathscr Ab}

\newcommand{\Sany}{\Sigma} 
\newcommand{\sft}{\mathscr S^{ft}} 
\newcommand{\seft}{\mathscr S^{eft}} 

\newcommand{\pts}{\mathscr F}

\DeclareMathOperator{\Ker}{Ker}

\DeclareMathOperator{\uK}{\underline K} 
\DeclareMathOperator{\uPic}{\underline{\Pic}} 
\newcommand{\virt}{\mathscr K} 
\newcommand{\lb}{\mathscr L} 
\newcommand{\virtlci}{\mathscr K^{lci}} 

\newcommand{\cotg}{\mathrm L}
\newcommand{\cotgb}{\tau} 
\newcommand{\detcotgb}{\omega} 

\newcommand{\cL}{\mathcal{L}} 

\newcommand{\DM}{\mathrm{DM}}

\DeclareMathOperator{\Sp}{Sp}

\DeclareMathOperator\GW{GW}

\usepackage{mathtools}
\DeclarePairedDelimiter\ev{\langle}{\rangle}

\DeclareMathOperator{\Hom}{Hom}

\DeclareMathOperator{\uHom}{\underline{Hom}} 
\DeclareMathOperator{\Map}{Map}
\DeclareMathOperator{\Aut}{Aut}
\DeclareMathOperator{\Tor}{Tor}
\DeclareMathOperator{\Pic}{Pic}
\DeclareMathOperator{\rk}{rk}
\DeclareMathOperator{\spec}{Spec}

\DeclareMathOperator{\CH}{CH}
\newcommand{\CHW}{\widetilde{\mathrm{CH}}{}}

\newcommand{\cohtp}{\Pi_{\AA^1}} 

\newcommand{\ilim} { \varinjlim }
\newcommand{\plim} { \varprojlim }

\newcommand{\pro}[1] {\operatorname{pro}-#1}

\DeclareMathOperator{\zdiv}{div}
\DeclareMathOperator{\tdiv}{\widetilde{div}}

\DeclareMathOperator{\Comp}{C}

\newcommand{\derR}{\mathbf{R}}

\newcommand{\FF} {\mathbf F}
\newcommand{\NN} {\mathbf N}
\newcommand{\ZZ} {\mathbf Z}
\newcommand{\QQ} {\mathbf Q}

\newcommand{\CC} {\mathbf C}
\renewcommand{\AA} {\mathbf A}
\newcommand{\PP} {\mathbf P}
\newcommand{\GG} {\mathbf{G}_m}

\newcommand{\GGx}[1] {\mathbf{G}_{m,#1}}
\newcommand{\C}{\mathcal C} 
\newcommand{\E}{\mathbb E}

\newcommand{\un}{\mathbbm 1} 

\newcommand{\A}{\mathscr A}

\newcommand{\zar}{{\mathrm{Zar}}}
\newcommand{\nis}{{\mathrm{Nis}}}
\newcommand{\et}{\mathrm{\acute{e}t}}

\DeclareMathOperator{\uG}{\underline \Gamma}
\DeclareMathOperator{\Sh}{Sh}
\DeclareMathOperator{\PSh}{PSh}

\DeclareMathOperator{\Cortilde}{\tilde{\operatorname{Cor}}}

\newcommand{\Lsp}  {{\mathrm{L}^{sp}_S}}

\newcommand{\DA}{\Der_{\AA^1}}
\newcommand{\Sm}{\mathit{Sm}}
\newcommand{\Smsm}{\mathit{Sm}^{sm}}

\newcommand{\Smsp}{\mathit{Sm}^{sp}}

\newcommand{\Gtw}[1]{\{#1\}} 
\newcommand{\tw}[1]{\langle#1\rangle} 
\DeclareMathOperator{\Th}{Th} 

\DeclareMathOperator{\Spec}{Spec}

\newcommand{\KGL}{\mathbf{KGL}}

\newcommand{\Flag}{\operatorname{Flag}}

\newcommand{\colim}{\operatorname{colim}}

\newcommand{\Id}{\operatorname{Id}}

\newcommand\res{\operatorname{res}}
\newcommand\cores{\operatorname{cores}}
\newcommand\KMW{\underline{\operatorname{K}}^{MW}}

\newcommand\kMW{\operatorname{K}^{MW}}
\newcommand{\CatMW}{\operatorname{MW-Mod}}

\usepackage{pdftexcmds}

\DeclareFontFamily{U}{cbgreek}{}
\DeclareFontShape{U}{cbgreek}{m}{n}{
	<-6>    grmn0500
	<6-7>   grmn0600
	<7-8>   grmn0700
	<8-9>   grmn0800
	<9-10>  grmn0900
	<10-12> grmn1000
	<12-17> grmn1200
	<17->   grmn1728
}{}
\DeclareFontShape{U}{cbgreek}{bx}{n}{
	<-6>    grxn0500
	<6-7>   grxn0600
	<7-8>   grxn0700
	<8-9>   grxn0800
	<9-10>  grxn0900
	<10-12> grxn1000
	<12-17> grxn1200
	<17->   grxn1728
}{}

\makeatletter
\newcommand{\normalorbold}{%
	\ifnum\pdf@strcmp{\math@version}{bold}=\z@ bx\else m\fi
}
\makeatother

\newcommand{\dH}{{}^{\delta}\mathbf{H}}

\newcommand{\dR}{{}^{\delta}\mathcal{R}}

\newcommand{\CatM}{\operatorname{M-Mod}}

\newcommand{\CoCatMW}{\operatorname{MW-Mod}^{coh}}
\newcommand{\CoCatM}{\operatorname{M-Mod}^{coh}}

\newcommand{\hM}{\mathcal{M}}
\newcommand{\cohM}{M}

\newcommand{\hC}{{}^{\delta}\widetilde{C}}

\newcommand\bulleto{\bullet\!\!\! \to}

\begin{document}

\begin{abstract}
We compute the perverse $\delta$-homotopy heart of the motivic stable homotopy category over a base scheme with a dimension function $\delta$, rationally or after inverting the exponential characteristic in the equicharacteristic case. In order to do that, we define the notion of homological Milnor-Witt cycle modules and construct a homotopy-invariant Rost-Schmid cycle complex. Moreover, we define the category of cohomological Milnor-Witt cycle modules and show a duality result in the smooth case.
\end{abstract}

\maketitle

\tableofcontents

\section*{Introduction}

\subsection*{Historical approach}

The category of motivic complexes was described 
conjecturally by Beilinson in the form of a program
almost thirty years ago. A first construction was proposed
by Voevodsky in his thesis using the $h$-topology, which already
contained in germs what he will defined later as the $\AA^1$-homotopy
theory. Though that category was defined over an arbitrary base scheme,
it was too coarse to fulfill the program of Beilinson. It was realized
slightly later that this theory corresponds to what is called now
\emph{\'etale motives}, or sometimes \emph{Lichtenbaum motives} as the
corresponding formalism was conjectured by Lichtenbaum - almost at the same 
as Beilinson did: see \cite{Ayoub1}, \cite{CD15}.

Voevodsky refined his construction to get the right integral
category of motivic complexes over a perfect field: the main indication
that the later category is the correct one is constituted by the fact
the morphisms from the (homological) motive of a smooth scheme $X$
to the Tate object $\ZZ(n)[2n]$ is the Chow group of $n$-codimensional
cycles in $X$.
The search for an extension of his definition to an arbitrary base,
satisfying the program of Beilinson, has been an boiling question
since then. A satisfactory answer to that question was obtained,
after the fundamental works of \cite{Ayoub1,Ayoub07b} and \cite{CD3}
by M.~Spitzweck in \cite{Spitzweck01}. More recently, a complete picture
merging \cite{CD3} and \cite{Spitzweck01} was obtained in the equi-characteristic
case in \cite{CD5}. In short, in the original construction
of Voevodsky one had only to replace the Nisnevich topology by 
the cdh-topology if one invert the residue characteristic $p$.

However, one fundamental construction is still lacking
on the category of motivic complexes, even if one restricts
to equi-characteristics schemes: the motivic $t$-structure,
supposed to be the $t$-structure which is realized to the perverse
$t$-structure after taking $l$-adic realization.
On the other hand, one has at our disposal a well defined
$t$-structure on $\DM(k)$, for a perfect field $k$,
whose existence was a corollary of the theory of Voevodsky.
It was later extended to arbitrary base by Ayoub in \cite{Ayoub1}.
Actually, Ayoub constructed two extensions: the one we will be
interested in is the so-called perverse homotopy $t$-structure.
It has the distinctive feature to be obtained by gluing: this
roughly means it is determined by restriction to the points
of the base scheme. This allows the study of that category
by restriction to the case of fields.

Some time ago, the heart of the homotopy $t$-structure over a perfect
base field $k$ was described in the thesis of the first author, as the
category of cycle modules over $k$, previously defined by Rost.
This is a far-fetched generalization of the Gersten resolution
of homotopy sheaf with transfers, proved by Voevodsky.
Thus, it is natural to ask if this theorem extends to
an arbitrary base scheme. This is what we will prove here,
up to inverting the exponential characteristic of the base field.

The motivation of this work is the following conjecture first due to Ayoub:
\begin{conj}
	Let $S$ be a base scheme of characteristic $0$.
	Then the heart of the perverse homotopy t-structure of \cite{Ayoub1} relative to $S$
	is equivalent to the category of Rost cycles modules over $S$.
\end{conj}
This conjecture was attacked in \cite{Deg17}. The reason to work in characteristic $0$
is to get the non-degeneration of the perverse homotopy $t$-structure.
Thanks to \cite{BD1}, we now have such a t-structure for schemes over a field $k$ of arbitrary exponential
characteristic $p$, provided one takes $\ZZ[1/p]$-coefficients.

The purpose of the present work is to prove this conjecture.
Actually, we can do much better. First we can work with characteristic $p>0$ fields up to inverting $p$ in the coefficients. Second,
if we work with rational coefficients, we can work with arbitrary
noetherian finite dimensional base schemes equipped with a dimension
function. This requires using the $\delta$-homotopical version
Ayoub's perverse homotopy $t$-structure, as developed in
\cite{BD1}.

In this paper, we follow \cite{Deg17bis, Feld1, Feld2, Fel21, Fel21b, Fel21c}: we want to be careful about Thom spaces.

\subsection*{Main results}

The notion of \emph{Milnor-Witt cycle modules} is introduced by the second-named author in \cite{Feld1} over a perfect field which, after slight changes, can be generalized to more general base schemes (see \cite{BHP22} for the case of a regular base scheme). In the spirit of the general formalism of \emph{bivariant theories} (see \cite{Deg16}), we call these objects \emph{cohomological Milnor-Witt cycle modules}, to distinguish them from the notion introduced in Section~\ref{sec:homMW}, considered as the \emph{homological} variants. 
\par 
In Section \ref{sec:homMW}, we define the notion of \textit{homological Milnor-Witt cycle module} $\hM$ and its associated Rost-Schmid cycle complex ${}^{\delta}C_*(X,\hM,*)$, and prove basic functoriality properties.

\par In Section \ref{sec:Cycle_to_homotopy_modules}, we construct the \textit{homotopy cycle complex} $\hC_*(X,\hM,*)$ associated to a homological Milnor-Witt cycle module, and prove that it satisfies a strong property of homotopy invariance:
\begin{thm}[Strong Homotopy invariance] (see Corollary \ref{HomotopyInvarianceDeformedComplex})
	Let $\hM$ be a homological MW-cycle module. Let $X$ be a scheme over $S$, $q$ an integer, $*$ a line bundle over $X$ and $\pi:V\to X$ a vector bundle. Then the pullback morphism
	\begin{center}
		$\pi^!:\hC^*(X,\hM_q,*) \to \hC^*(V,\hM_q,*)$
	\end{center}
	is an isomorphism.
	
\end{thm}

The strong homotopy invariance of $\hC$ allows us to define Gysin maps for projective lci maps. In particular, one can see that the complex $\hC$ is a presheaf of smooth $S$-schemes, Nisnevich-local and $\AA^1$-invariant. Moreover, it computes the homology groups of the Rost-Schmid complex:

\begin{thm}
	The exist maps 
	\begin{center}
		$\widetilde{\alpha}_X: {}^{\delta}C^*(X,\hM_q, *) \to \hC_*(X,\hM_q,*)$
	\end{center}
	
	\begin{center}
		$\widetilde{H}_X: \hC^*(X,\hM_q,*)
		\to \hC^*(X,\hM_q,*)$.
	\end{center}	
	
	\begin{center}
		$\widetilde{r}_X: \hC^*(X,\hM_q,*) \to {}^{\delta}C^*(X,\hM_q,*)$.
	\end{center}
	that form an $h$-data (in the sense of \cite[§9.1]{Rost96}), i.e. they satisfy the following properties:
	\begin{itemize}
		\item $\widetilde{H}_X \circ \widetilde{\alpha}_X = 0$,
		\item $\widetilde{r}_X \circ \widetilde{\alpha}_X=\Id$,
		\item $\delta(\widetilde{H}_X)=\Id-\widetilde{\alpha}_X *\circ \widetilde{r}_X$
	\end{itemize}
	where $\delta(\widetilde{H}_X)=\partial^{n-1}\circ \widetilde{H}_X + \widetilde{H}_X\circ \partial^{n+1}$.
	
\end{thm}

\par In Section \ref{sec:Rost_transform}, we define the notion of \textit{twisted Borel-Moore homology theory} and shows how to obtain a homological Milnor-Witt cycle module from a motivic spectrum.

\begin{thm}[see Theorem \ref{thm_Rost_transform}]
	Let $\E$ be a motivic spectrum. For any $p \in \ZZ$, there exists an object $\hat H_p^\delta(\E)$ which is a (homological) MW-cycle module over $S$ (and functorial in $\E$).
	In particular, we get a canonical functor, called the \emph{Rost transform}, from the stable motivic homotopy category to the category of homological Milnor-Witt cycle modules:
	\begin{align*}
		\dR: \SH(S) &\rightarrow \CatMW_S\\
		\E &\mapsto \hat \E=\hat H_0^\delta(\E).
	\end{align*}
	Finally, if $v$ is a virtual bundle over $X$, of rank $0$ and determinant $\cL$,
	one has a canonical isomorphism of graded complexes, natural in $\E$:
	\begin{equation}\label{eq:Gersten&RS-complexes}
		{}^{\delta}C_*(X,\E\tw v) \simeq 
		{}^{\delta}C_*(X,\hat  \E,\cL).
	\end{equation}
\end{thm}

Building from the results in Section \ref{sec:Cycle_to_homotopy_modules}, we can also construct of functor from the category of Milnor-Witt cycle modules to the derived homotopy category:

\begin{thm}[see Theorem \ref{thm:functor_KMW->DA}]
	Let $\hM$ be a homology MW-cycle modules.
	Then $\dH  \hM(X)$ is an $\AA^1$--local, Nisnevich fibrant, $\Omega$-fibrant element of $\DA(S)$. Moreover, we get a functor $$
	\dH:\CatMW_S \rightarrow \DA(S).
	$$

\end{thm}

\par In Section \ref{sec:Equivalence}, we prove the main theorem of the present paper: there is an adjunction between the category of homological Milnor-Witt cycle modules and the category of homotopy derived complexes which are non-negative for the t-structure of \cite{BD1}. Moreover, this adjunction restricts to an equivalence to the heart of this t-structure.

In this version, we give the full proof of the first part of the following theorem, and gives a detailed argument for the second part, based on the existence of a theory of sheaves with specialization that we sketch.
\begin{thm}[see Theorem \ref{thm_main_adjunction_theorem}]
	Consider the above notation.
	\begin{enumerate}
		\item There exists an additive functors:
		$$
		\dR:\DA(S,R) \rightarrow \CatMW_S
		$$
		which is a right inverse of $\dH$: $\dR \circ \dH \simeq \Id$.
		Thus the functor $\dH$ is faithful.
		\item If the category of $S$-schemes satisfies (Resol) and $\DA(-,R)$ is homotopically compatible,
		the previous adjunction induces an equivalence of (additive) categories:
		$$
		\dH:\CatMW_S \leftrightarrows \DA(S,R)^\heartsuit:\dR
		$$
		where the right hand-side is the heart for the perverse $\delta$-homotopy $t$-structure.
	\end{enumerate}
\end{thm}

\par In Section \ref{sec:MWmodcoh}, we define the notion of \textit{cohomological Milnor-Witt cycle module} and \textit{oriented homological Milnor-Witt cycle module } (also known as \textit{Milnor cycle module}).

Our main result on cohomological Milnor-Witt cycle modules will be a duality theorem relating these objects to their homological counterparts, which hold for possibly singular schemes, see Theorems~\ref{eq:cohdualori} and~\ref{thm:eqpin} below.

\begin{thm}[see Theorem \ref{thm:eqpin}]
	Let $(S,\lambda)$ be an excellent scheme with a pinning. Given a cohomological MW-cycle module $M$ over $S$, for every $S$-field $\operatorname{Spec}E$, we let
	\begin{align}
		D(M)(E,n)=M_n(E,\lambda_{E}).
	\end{align}
	Reciprocally given a homological MW-cycle module $\mathcal{M}$ over $S$, we let 
	\begin{align}
		\mathcal{D}(\mathcal{M})_n(E)=\mathcal{M}(E,\lambda_{E}^\vee,n).
	\end{align}
	Then these two constructions establish two functors that are equivalences of categories inverse to each other
	\begin{align}
		D:\CoCatMW_S\simeq\CatMW_S:\mathcal{D}.
	\end{align}
\end{thm}

\par Finally, in Section \ref{sec:Computations}, we discuss the Gersten conjecture and improve the result of \cite[Section 8]{Feld1}. Moreover, we compute the Chow-Witt group of zero cycles of a number ring:

\begin{thm}[see Theorem \ref{thm_Chow_Witt_number_ring}]
	Let $\cO_K$ be a number ring. Then, one has the following exact sequence
	$$
	\CH^1(\cO_K) 
	\rightarrow 
	\CHW^1(\cO_K) 
	\rightarrow 
	\CH^1(\cO_K)/2 
	\rightarrow 0
	$$
	where the first map is induced by the hyperbolic map $h$ and the second one is induced by modding out by $h$.
	In particular, $\CHW^1(X)$ is finite.
\end{thm}

\subsection*{Outline of the paper}
\par In Section \ref{sec:homMW}, we define the notion of \textit{homological Milnor-Witt cycle module} and its associated Rost-Schmid cycle complex, and prove basic functoriality properties.

\par In Section \ref{sec:Cycle_to_homotopy_modules}, we construct the \textit{homotopy cycle complex} associated to a homological Milnor-Witt cycle module, and prove that it satisfies a strong property of homotopy invariance, it computes the homology groups of the Rost-Schmid complex, and it has functorial Gysin morphisms for lci maps.

\par In Section \ref{sec:Rost_transform}, we define the notion of \textit{twisted Borel-Moore homology theory} and shows how to obtain a homological Milnor-Witt cycle module from a motivic spectrum. 

\par In Section \ref{sec:Equivalence}, we prove that there is an adjunction between the category of homological Milnor-Witt cycle modules and the category of homotopy derived complexes which are non-negative for the t-structure of \cite{BD1}. Moreover, this adjunction restricts to an equivalence to the heart of this t-structure.

\par In Section \ref{sec:MWmodcoh}, we define the notion of \textit{cohomological Milnor-Witt cycle module} and \textit{oriented homological Milnor-Witt cycle module } (also known as \textit{Milnor cycle module}). We prove that there exists a \textit{Poincaré duality} between homological and cohomological Milnor-Witt cycle modules, provided some conditions on the base scheme.

\par Finally, in Section \ref{sec:Computations}, we compute the Chow-Witt group of zero cycles of a number ring. Moreover, we discuss the Gersten conjecture and improve the result of \cite[Section 8]{Feld1}.

\subsubsection*{Future work}
In a subsequent paper, we will define and compute \textit{quadratic lengths}, and construct pullback maps associated to flat morphisms. This will be used to prove ramification formula and a stronger version of the base change theorem present in this article.
\par Moreover, the formalism of the present paper is used in \cite{Fel22a} to study birational invariants. More similar results should follow.

\section*{Notations and conventions}

In this paper, schemes are noetherian and finite dimensional.
We fix a base scheme $S$ and ring of coefficients $R$. If not stated otherwise, all schemes and morphisms of schemes are defined over $S$.
A \emph{point} (resp. \emph{trait}, \emph{singular trait}) of $S$ will be a morphism of schemes $\Spec(k) \rightarrow S$ essentially of finite type and such
that $k$ is a field (resp. local valuation ring, local ring of dimension $1$). Morphisms of such are morphisms of $S$-schemes.

\par List of categories used in this paper: 

\begin{itemize}
	\item $\ab$ the category of abelian groups
	\item $\smod R$ the category of $R$-modules
	\item $\Sp$ the $\infty$-category of spectra
	\item $\sft_S$ the category of $S$-schemes of finite type
	\item $\seft_S$ the category of $S$-schemes of essentially finite type
	\item $\Sm_S$ the category of smooth $S$-schemes of finite type
	\item $\pts_S$ the category of points of $S$

	\item $\virt$ the fibred category of virtual vector bundles (see \ref{df:virt})
	\item $\virtlci$ the lci-fibred category of virtual vector bundles
	\item $\SH_S$ the motivic stable homotopy category over $S$
	\item $\DA(S)$ the $\AA^1$-derived category over $S$
	\item $\Der^b_{\mathrm{coh}}(S)$ the {bounded derived category of quasi-coherent sheaves over $S$ with coherent cohomology}
	\item $\CatMW_S$ the category of homological Milnor-Witt cycle modules over $S$ (see \ref{MWCyleAbelian})
	\item $\CoCatMW_S$ the category of cohomological Milnor-Witt cycle modules over $S$ (see \ref{CoCatMW})
	\item $\Comp(\A)$ the category of complexes of an abelian category $\A$
	\item $\Der(\A)$ the derived category of complexes of an abelian category $\A$
	\item $F\A$ the category of filtered objects
	of an abelian category $\A$ (see \ref{notation_filtered_category})
	\item $\Comp^b(\Sh(X_t,R))_{Cous}$ the full subcategory of $\Comp^b(Sh(X_t,R))$ consisting of Cousin complexes (see \ref{def_Cousin_complex})
	\item $\Der(\Sh(X_t,R))_{CM}$ the full subcategory of $\Der(\Sh(X_t,R))$ consisting of Cohen-Macaulay complexes (see \ref{def_Cohen_Macaulay_complex})
	\item $\Smsm_S$ the category of smooth affine $S$-schemes, with morphisms the smooth maps, equipped with the induced Nisnevich topology (see \ref{def_cat_smsm})
	\item $\Smsp_S$ the category of smooth schemes with smooth maps and specializations (see \ref{def_cat_smsp}) 
	\item $\PSh^{sp}(S,R)$ (resp. $\Sh^{sp}(S,R)$) the category of presheaves (resp. Nisnevich sheaves) with specializations (see \ref{def_sheaves_with_specializations}).
	
	\item $\mathcal{G}_S$ the full subcategory of the category of $S$-schemes essentially of finite type, whose objects consist of $S$-fields and $S$-DVRs (see \ref{S_DVR}).
\end{itemize}

All schemes in this paper are assumed to
have a \emph{bounded dimension function} (see \cite[1.1]{BD1} for recall).
Pairs $(X,\delta)$ such that $X$ is a scheme and $\delta$ is a dimension function on $X$
are called \emph{dimensional schemes}. As usual, $\delta(S)$ is the maximum of the $\delta(\eta)$
for $\eta$ running over generic points of $S$.

Conventions: a morphism $f:X \rightarrow S$ (sometime denoted by $X/S$) is:
\begin{itemize}
	\item essentially of finite type
	if $f$ is the projective limit of a cofiltered system $(f_i)_{i \in I}$ of morphisms of finite type
	with affine and \'etale transition maps;  
		\item lci if it is smoothable and a local complete intersection (\emph{i.e.} admits a global factorization $f=p \circ i$,
		$p$ smooth and $i$ a regular closed immersion);
		\item essentially lci if it is a limit of lci morphisms with étale transition maps.

	\end{itemize}
	If $(S,\delta)$ is a dimensional scheme, and $X/S$ is essentially of finite type
	then $\delta$ can be extended canonically to $X$ (see e.g. \cite[1.1.7]{BD1}).
	We will simply denote by $\delta$ this extension so that $\delta(X)$ is defined.
	We put $X_{(\delta=p)}=\{x \in X\mid \delta(x)=p\}$, or simply $X_{(p)}$ when $\delta$ is clear.
	
	A point $x$ of $S$ is a map $x:\Spec(E) \rightarrow S$ essentially of finite type and such $E$ is a field.
	We also say that $E$ is a field over $S$.
	
	Given a morphism of schemes $f:Y \rightarrow X$, we let $\cotg_f$ be its cotangent complex,
	an object of $\Der^b_{\mathrm{coh}}(Y)$,
	and when the latter is perfect (e.g. if $f$ is essentially lci), we let $\cotgb_f$ be its associated
	virtual vector bundle over $Y$.
	
	We will use the axiomatic of motivic triangulated categories (\cite{CD3})
	and its extension to motivic $\infty$-categories (as in \cite{CisCoh}, \cite{DJK}).
	The main examples used in the paper are the universal (resp. rational universal)
	version the (rational) stable homotopy category $\SH$ (resp. $\SH_\QQ$) but several variants
	are of interest.
	
	We use the following notations for twists: $\Th_X(v)$ is the Thom space of $v$, $\Gtw 1=(1)[1]$ and
	$\tw 1=(1)[2]$.

	\par Let $E$ be a field. We denote by $\GW(E)$ the Grothendieck-Witt ring of non-degenerate symmetric bilinear forms on $E$. For any $a\in E^*$, we denote by $\ev{a}$ the class of the symmetric bilinear form on $E$ defined by $(X,Y)\mapsto aXY$ and, for any natural number $n$, we put $n_{\epsilon}=\sum_{i=1}^n \ev{-1}^{i-1}$. Recall that, if $n$ and $m$ are two natural numbers, then $(nm)_{\epsilon}=n_{\epsilon}m_{\epsilon}$.

	\begin{df} \label{S_DVR}
		\begin{enumerate}
			\item
			We call an \textbf{$S$-DVR} the spectrum of a discrete valuation ring which is essentially of finite type over $S$. 
			\item
			We denote by $\mathcal{G}_S$ the full subcategory of the category of $S$-schemes essentially of finite type, whose
			
			objects consist of $S$-fields and $S$-DVRs.
		\end{enumerate}
	\end{df}

	\par 
	Some conventions:
	\begin{itemize}
		\item $\cohM$ cohomological MW-modules
		\item $\hM$ homological MW-modules
		\item $\cL$ or simply $*$ line bundle
		\item $v,w$ virtual vector bundles
		\item $f,g$ scheme morphisms; $\phi, \psi$ ring morphisms
		\item $\E$ motivic spectrum
	\end{itemize}

	One considers cohomology and BM-homology theories:
	\begin{enumerate}
		\item \textit{\cite{DJK} conventions}:
		\begin{align*}
			\E^p(X,v)&=[\un_X,f^*\E \otimes \Th(v)[p]]_X, \\
			\E_p^{BM}(X/S,v)&=[\Th(v)[p],f^!\E ]_X.
		\end{align*}
		The last one is also called the bivariant theory with coefficients in $\E$,
		and the super-script "BM" is erased.
		\item \textit{Classical conventions, following algebraic topology and mixed motives}:
		\begin{align*}
			H^{p,q}(\Th(v),\E)&=[\Th(v),f^*\E(q)[p]]_X, \\
			H^{p,\Gtw q}(\Th(v),\E)&=[\Th(v),f^*\E\Gtw q[p]]_X, \\
			H^{p,\tw q}(\Th(v),\E)&=[\Th(v),f^*\E\tw q[p]]_X, \\
			H_{p,q}^{BM}(\Th(v)/S,\E)&=[\Th(v)(q)[p],f^!\E]_X, \\
			H_{p,\Gtw q}^{BM}(\Th(v)/S,\E)&=[\Th(v)\Gtw q[p],f^!\E]_X, \\
			H_{p,\tw q}^{BM}(\Th(v)/S,\E)&=[\Th(v)\tw q[p],f^!\E]_X.
		\end{align*}
		When $v=0$, we use $X$ in place of $\Th(0)$.
		
		\par 
		In particular, we have when $X$ is smooth:
		\begin{center}
			$H^{BM}_{-p, \{ -q \} }
			(\Th (+\cotg_f + v), \E)
			=
			H^{p, \{ q \} }
			(\Th ( v), \E)$
		\end{center}
		\par 
		This notation is closer to that of \cite{BD1}.
		Of course, these conventions are redundant but it seems that in our paper,
		the $\GG$-twist $\Gtw q$ is the most relevant.
	\end{enumerate}
	
	\subsubsection*{Acknowledgments}
	The authors deeply thanks Aravind Asok, Joseph Ayoub, Mikhail Bondarko, Baptiste Calmès, Jean Fasel, Marc Levine, Johannes Nagel, Paul Arne \O stv\ae r, Bertrand Toën for conversations, exchanges and ideas that led to the present paper.

	The work of the first two named authors are supported by the ANR HQDIAG project no ANR-21-CE40-0015. The second named author is supported by the ANR LabEx CIMI within the French State Programme “Investissement d’Avenir”.
	The third-named is supported by the National Key Research and Development Program of China Grant Nr.2021YFA1001400, the National Natural Science Foundation of China Grant Nr.12101455 and the Fundamental Research Funds for the Central Universities.

\section{Homological Milnor-Witt cycle modules}

\label{sec:homMW}

\subsection{Homological Milnor-Witt cycle premodules}

Fix $R$ a commutative unitary ring. The following definition is a generalization of \cite[Definition 3.1]{Feld1}.
\begin{df} \label{defMWmodules}
	A homological Milnor-Witt cycle premodule $\hM$ (also written: MW-cycle premodule) over the base scheme $S$ and with coefficients in the ring $R$ is an $R$-module $\hM_n(E)$ for any field $E$ over $S$ and any integer $n$ along with the following data \ref{itm:D1'},\dots, \ref{itm:D4'} and the following rules \ref{itm:R1a'},\dots, \ref{itm:R3e'}.
	\begin{description}
		\item [\namedlabel{itm:D2'}{(D2')}] Let $\phi:E\to F$ be a finite field extension over $S$ and $n$ an integer. There exists a morphism  $\phi^*:\hM_n(F) \to \hM_n(E)$.

		\item [\namedlabel{itm:D3'}{(D3')}] Let $E$ be a field over $S$ and $n,m$ two integers. For any element $x$  of $\kMW_n(E)$, there is a morphism 
		\begin{center}
			$\gamma_x : \hM_n(E)\to \hM_{m+n}(E)$
		\end{center}
		so that the functor $\hM_{?}(E):\ZZ\to R\Mod$ is a left module over the lax monoidal functor 
		
		\begin{center}
			${\kMW_?(E):\ZZ\to R\Mod}$.
		\end{center}
		
		\item [\namedlabel{itm:D4'}{(D4')}] Let $\cO$ be a $1$-dimensional local domain essentially of finite type over $S$ with fraction field $F$ and residue field $\kappa$.
		
		There exists a morphism of abelian groups
		\begin{align}
			\label{eq:D4'}
			\partial_\mathcal O : \hM_n(F) \to \hM_{n-1}(\kappa).
		\end{align}
		
		Let $E$ be a field over $S$, $n$ an integer and $\cL$ a 
		$1$-dimensional vector space over $E$, denote by
		\begin{equation}\label{eq:df_twists_preMW}
			\hM_n(E,\cL):=\hM_n(E)\otimes_{R[E^\times]}R[\cL^\times]
		\end{equation}
		where $R[\cL^\times]$ is the free $R$-module generated by the nonzero elements of $\cL$ and where the action of $R[E^\times]$ on $\hM_n(E)$ is given by $u\mapsto \langle u \rangle$ thanks to \ref{itm:D3'}.
		\item[\namedlabel{itm:D1'}{(D1')}] Let $\phi:E\to F$ be a field extension and $n$ an integer. There exists a morphism of abelian groups
		\begin{center}
			$\phi_!:\hM_n(E) \to \hM_n(F,\detcotgb_{F/E})$.
		\end{center}
		
		\item [\namedlabel{itm:R1a'}{(R1a')}] Let $\phi$ and $\psi$ be two composable morphisms of fields. We has
		\begin{center}
			$(\psi\circ \phi)_*=\psi_*\circ \phi_*$.
		\end{center}
		\item [\namedlabel{itm:R1b'}{(R1b')}]  Let $\phi$ and $\psi$ be two composable finite morphisms of fields. We has
		\begin{center}
			$(\psi\circ \phi)^*=\phi^*\circ \psi^*$.
		\end{center}
		\item [\namedlabel{itm:R1c'}{(R1c')}] Consider $\phi:E,\to F$ and $\psi:E\to L$ with $\phi$ finite and $\psi$ separable. Let $R$ be the ring $F\otimes_E L$. For each $p\in \Spec R$, let $\phi_p:L\to R/p$ and $\psi_p:F\to R/p$ be the morphisms induced by $\phi$ and $\psi$. One has
		\begin{center}
			$\psi_*\circ \phi^*=\displaystyle \sum_{p\in \Spec R} (\phi_p)^*\circ (\psi_p)_*$.
		\end{center}

		\item [\namedlabel{itm:R2'}{(R2')}] Let $\phi : E\to F$ be a field extension, let $x$ be in $\kMW_n (E)$ and $y$ be in $\kMW_{n'} (F)$.
		\item [\namedlabel{itm:R2a'}{(R2a')}] We have $\phi_* \circ \gamma_x= \gamma_{\phi_*(x)}\circ \phi_*$.
		\item [\namedlabel{itm:R2b'}{(R2b')}] Suppose $\phi$ finite. We have $\phi_*\circ \gamma_{\phi^*(x)}=\gamma_x \circ \phi_*$.
		\item [\namedlabel{itm:R2c'}{(R2c')}] Suppose $\phi$ finite. We have $\phi_*\circ \gamma_y \circ \phi^*= \gamma_{\phi_*(y)}$.
		\item [\namedlabel{itm:R3a'}{(R3a')}] 
		Let $E\to F$ be a field extension and $w$ be a valuation on $F$ which restricts to a non trivial valuation $v$ on $E$ with ramification $1$. Let $n$ be a virtual $\mathcal{O}_v$-module so that we have a morphism $\phi:(E, n)\to (F,n)$ which induces a morphism
		${\overline{\phi}:(\kappa(v),-\cotgb_{v}+n)\to
			(\kappa(w),n) }$. 
		We have 
		
		\begin{center}
			
			$\partial_w \circ \phi_*= \overline{\phi}_* \circ \partial_v$. 
		\end{center}
		\item [\namedlabel{itm:R3b'}{(R3b')}] Let $E\to F$ be a finite extension of fields over $k$, let $v$ be a valuation on $E$ and let $n$ be a $\mathcal{O}_v$-module. For each extension $w$ of $v$, we denote by 
		${\phi_w: (\kappa(v), n) \to ( \kappa(w), n)}$ the morphism 
		induced by ${\phi:(E,n)\to (F,n)}$. We have
		\begin{center}
			$\partial_v \circ \phi^*= \sum_w (\phi_w)^* \circ \partial_w$.
		\end{center}
		\item [\namedlabel{itm:R3c'}{(R3c')}] Let $\phi : (E,n)\to (F,n)$ be a morphism in $\mathfrak{F}_k$ and let $w$ be a valuation on $F$ which restricts to the trivial valuation on $E$. Then 
		\begin{center}
			$\partial_w \circ \phi_* =0$.
		\end{center}
		\item [\namedlabel{itm:R3d'}{(R3d')}] Let $\phi$ and $w$ be as in \ref{itm:R3c'}, and let $\overline{\phi}:(E,n)\to (\kappa(w),n)$ be the induced morphism. For any uniformizer $\pi$ of $v$, we have
		\begin{center}
			$\partial_w \circ \gamma_{[\pi]}\circ \phi_*= \overline{\phi}_*$.
		\end{center}
		
		\item [\namedlabel{itm:R3e'}{(R3e')}] Let $E$ be a field over $k$, $v$ be a valuation on $E$ and $u$ be a unit of $v$. Then
		\begin{center}
			$\partial_v \circ \gamma_{[u]}=\gamma_{-[\overline{u}]} \circ \partial_v$ and
			\\ $\partial_v \circ \gamma_\eta =\gamma_{-\eta} \circ \partial_v$.
		\end{center}

	\end{description}
\end{df}

\begin{rem}
	Note that axiom \ref{itm:R3a'} above is weaker than its counterpart in \cite[§2, Rule R3a]{Feld1} because
	we only consider the case without ramification. A stronger result with ramification will be proved later.
\end{rem}

\subsection{Homological Milnor-Witt cycle modules} \label{Modules}

\begin{paragr}\label{2.0.1}

Throughout this section, $\hM$ denotes a homological Milnor-Witt cycle premodule over $S$.
\par Let $X$ be a scheme over $S$.

\par Throughout this paper, we write for any integer $n$ and any line bundle $*$,
\begin{center}
	$\hM_n(x, *)=\hM_n(\kappa(x), *)$.
\end{center}
\par  If $X$ is irreducible, we write $\xi_X$ or $\xi$ for its generic point. 
\par Now suppose $X$ is an arbitrary scheme over $k$ and let $x,y$ be two points in $X$. We define a map
\begin{center}
	$\partial^x_y:\hM_{n}(x,*) \to \hM_{n-1}(y,*)$
\end{center}
as follows. Let $Z=\overline{ \{x\}}$. If $y\not \in Z^{(1)}$, then put $\partial^x_y=0$. If $y\in Z^{(1)}$, then the local ring of $Z$ at $y$ gives us a map
\begin{align}
	\label{eq:reshom}
	\partial^x_y: \hM_n(x, *) \to \hM_{n-1}(y, *),
\end{align}
according to \ref{itm:D4'}.
\end{paragr}

\begin{df}
\label{df:homcycmod}
A homological Milnor-Witt cycle module $\hM$ over $S$ is a homological Milnor-Witt cycle premodule $\hM$ which satisfies the following conditions \ref{itm:FD'} and \ref{itm:C'}.
\begin{description}
	\item [\namedlabel{itm:FD'}{(FD')}] {\sc Finite support of divisors.} Let $X$ be a scheme, $n$ an integer, $*$ a line bundle over $X$, and $\rho$ be an element of $\hM_n(\xi_X,*)$. Then $\partial_x(\rho)=0$ for all but finitely many $x\in X^{(1)}$.
	
	\item [\namedlabel{itm:C'}{(C')}] {\sc Closedness.} Let $X$ be integral and local of dimension 2, $n$ an integer, and $*$ a line bundle over $X$. Then
	\begin{center}
		$0=\displaystyle \sum_{x\in X_{(\delta=1)}} \partial^x_{x_0} \circ \partial^{\xi}_x: 
		\hM_n(\xi_X,*)\to \hM_{n-2}(x_0, *)$
	\end{center}
	where $\xi$ is the generic point and $x_0$ the closed point of $X$.
\end{description}
\end{df}
\begin{paragr}

Of course \ref{itm:C'} makes sense only under presence of \ref{itm:FD'} which guarantees finiteness in the sum.
More generally, note that if \ref{itm:FD'} holds, then for any scheme $X$, any $x\in X$ and any $\rho\in \hM_n(x, *)$ one has $\partial^x_y(\rho)=0$ for all but finitely many $y\in X$.
\end{paragr}

\begin{paragr} \label{DefDifferential}
If $X$ is irreducible and \ref{itm:FD'} holds for $X_{\mathtt{red}}$, we put for any integer $n$ and any line bundle $*$:
\begin{center}
	$d=(\partial^\xi_x)_{x\in X_{(\delta=1)}}:\hM_n(\xi, *) \to \displaystyle \bigoplus_{x\in X_{(\delta=1)}} \hM_n(x, *)$.
\end{center}
\end{paragr}

\begin{df} \label{DefMWmorphisms}
\begin{description}
	\item 
	A morphism $\omega:\hM\to \hM'$ of homological Milnor-Witt cycle modules over $k$ is a natural 
	transformation which commutes with the data \ref{itm:D1'},\dots, \ref{itm:D4'}.
\end{description}
\end{df}

\begin{rem}\label{MWCyleAbelian}
We denote by $\CatMW_S$ the category of Milnor-Witt cycle modules (where arrows are given by morphisms of MW-cycle modules). This is an abelian category.
\end{rem}

\begin{paragr}

In the following, let $F$ be a field over $S$ and ${\AA_F^1=\Spec F[t]}$ be the affine line over $\Spec F$ with function field $F(t)$. Moreover, we fix $n$ an integer and $*$ a line bundle over $S$.
\end{paragr}

\begin{prop} \label{Prop2.2}

Let $\hM$ be a homological Milnor-Witt cycle module over $S$. With the previous notations, the following properties hold.
\begin{description}
	\item [\namedlabel{itm:(H')}{(H')}]  {\sc  Homotopy
		property for $\AA^1$}. We have a short exact sequence

	\begin{center}
		$\xymatrix@C=10pt@R=20pt{
			0 \ar[r] &  \hM_n(F,*) \ar[r]^-{\res}   
			& \hM_n(F(t),*) \ar[r]^-d  &
			\bigoplus_{x\in {(\AA_F^1)}_{(\delta = 1)}} \hM_{n-1}(\kappa(x),*) \ar[r] & 0
		}
		$
	\end{center}
	where the map $d$ is defined in \ref{DefDifferential}.

	\item [\namedlabel{itm:(RC')}{(RC')}] {\sc Reciprocity for curves}. Let $X$ be a proper curve over $F$. Then
	\begin{center}
		
		$\xymatrix{
			\hM_n(\xi_X, *) \ar[r]^-d 
			& \displaystyle \bigoplus_{x\in X_{(\delta = 1)}} \hM_{n-1}(x, *) \ar[r]^-c 
			& \hM_{n-1}(F, *)
		}$
		
	\end{center}
	is a complex, that is $c\circ d=0$ (where $c={\sum_x \cores_{\kappa(x)/F}}$).
\end{description}
\end{prop}

\begin{paragr}
Axiom \ref{itm:FD'} enables one to write down the differential $d$ of the soon-to-be-defined complex $C_*(X,\hM,*)$, axiom \ref{itm:C'} guarantees that $d\circ d = 0$, property \ref{itm:(H')} yields the homotopy invariance of the Chow groups $A_*(X,\hM,*)$ and finally \ref{itm:(RC')} is needed to establish proper pushforward.

\end{paragr}

\subsection{The five basic maps} \label{FiveBasic}

The purpose of this section is to introduce the cycle complexes and each operation on them needed further on. Note that the five basic maps defined below are analogous to those of Rost (see \cite[§3]{Rost96}); they are the basic foundations for the construction of more refined maps such as Gysin morphisms (see Section \ref{GysinMorphisms}).
\begin{paragr}

Let $\hM$ and $\hM'$ be two homological Milnor-Witt cycle modules over $S$, let $X$ and $Y$ be two schemes and let $U\subset X$ and $V\subset Y$ be subsets. Fix an integer $n$ and a line bundle $*$. Given a morphism
\begin{center}
	$\alpha:\displaystyle \bigoplus_{x\in U} 
	\hM_n(x,*) \to \displaystyle \bigoplus_{y\in V} 
	\hM_n'(y, *)$,
\end{center}
we write $\alpha^x_y:\hM_n(x,*) \to \hM_n'(y, *)$ for the components of $\alpha$.
\end{paragr}

\begin{paragr}{\sc Change of coefficients}
\label{MWmod_change_of_coefficients}
Let $\alpha:\hM\to \hM'$ be a morphism between two homological Milnor-Witt cycle modules, let $X$ be a scheme over $S$ and $U\subset X$ a subset. We put
\begin{center}
	
	$\alpha_{\#}:\bigoplus_{x\in U} \hM_n(x,*) \to \bigoplus_{x\in U} \hM_n'(x,*)$
\end{center}
where $(\alpha_{\#})^x_x=\alpha_{\kappa(x)}$ and $(\alpha_{\#})^x_y=0$ for $x\neq y$.

\end{paragr}

\begin{paragr}
{\sc Milnor-Witt cycle complexes}.

Let $\hM$ be a homological Milnor-Witt cycle module, let $X$ be a scheme with a line bundle denoted by $*$ and $p$ be an integer. Recall that $X_{(\delta=p)}$ is the set of $(\delta = p)$-dimensional points of $X$. Define
\begin{center}
	
	${}^{\delta}C
	_p(X,\hM,*)=\displaystyle \bigoplus_{x\in X_{(\delta = p)}}
	\hM_n(x,*)$
\end{center}
and
\begin{center}
	$d=d_X:C^\delta_p(X,\hM,*)\to C^\delta_{p-1}(X,\hM,*)$
\end{center}
where $d^x_y=\partial^x_y$ as in \ref{2.0.1}. This definition makes sense by axiom \ref{itm:FD'}.

\par 
More precisely, for any integer $q$, we define:

$$
{}^{\delta}C_p(X,\cM_q,*)=
\bigoplus_{x \in X_{(\delta=p)}} \hM_{q+p}(\kappa(x),*).
$$
and
\begin{center}
	$d=d_X:
	{}^{\delta}C_p(X,\hM_q,*)\to 
	{}^{\delta}C_{p-1}(X,\hM_q,*)$
\end{center}
where $d^x_y=\partial^x_y$ as in \ref{2.0.1}. This definition makes sense by axiom \ref{itm:FD'}.
\par We have
\begin{center}
	${}^{\delta}C_p(X, \hM, *)
	=
	\bigoplus_{q \in \ZZ}
	{}^{\delta}C_p(X, \hM_q, *).
	$
\end{center}
\end{paragr}

\begin{prop} \label{dod=0}
With the previous notations, we have $d\circ d=0$.
\end{prop}
\begin{proof}
Same as in \cite[§3.3]{Rost96}. Axiom \ref{itm:C'} is needed.
\end{proof}

\begin{df} \label{DefinitionComplex}
The complex 
$({}^{\delta}C_p(X,\hM,*),d)_{p\geq 0}$ (resp. 
$({}^{\delta}C_p(X,\hM_q,*),d)_{p\geq 0}$) is called the {\em Milnor-Witt complex of cycles on $X$ with coefficients in $\hM$ (resp. $\hM_q$)}.
\end{df}

\begin{paragr}
For $f:X\to S$ smooth of dimension $d$ and $*$ an invertible $\cO_X$-module, put: $X^{(\delta = d-p)}=X_{(\delta = p)}$ and:
$$
{}^{\delta}C^p(X,\cM_q,*)=
{}^{\delta}C_{d-p}(X,\cM_{q-d},* \otimes \detcotgb_f).
$$
\end{paragr}

\begin{df}
\label{def:AgpM}
The {\em Chow-Witt group of $p$-dimensional cycles with coefficients in $\hM$} is defined as the $p$-th homology group of the complex ${}^{\delta}C_*(X,\hM,*)$ (resp. ${}^{\delta}C_*(X,\hM_q,*)$) and denoted by ${}^{\delta}A_p(X,\hM,*)$ (resp. ${}^{\delta}A_p(X,\hM_q,*)$).
\end{df}

In the following, we fix $\hM$ a homological MW-cycle module, $q$ an integer and $*$ a line bundle over an $S$-scheme $X$.

\begin{paragr}{\sc Pushforward}
Let $f:Y\to X$ be a $S$-morphism of schemes. Define
\begin{center}
	
	$f_*:
	{}^{\delta}C_p(Y,\hM_q,*)\to 
	{}^{\delta}C_p(X,\hM_q, *)$
\end{center}
as follows. If $x=f(y)$ and if $\kappa(y)$ is finite over $\kappa(x)$, then $(f_*)^y_x=\cores_{\kappa(y)/\kappa(x)}$. Otherwise, $(f_*)^y_x=0$.
\par With cohomological convention, if $f$ has a relative dimension $s$, we have
\begin{center}
	
	$f_*:{}^{\delta}C^p(Y,\hM_q,*\otimes \detcotgb_f)
	\to 
	{}^{\delta}C^{p+s}(X,\hM_{q-s}, *)$.
\end{center}

\end{paragr}

\begin{paragr}{\sc Pullback} \label{pullbackBasicMap}
Let $f:Y\to X$ be an {\em essentially smooth} morphism of schemes. Suppose $Y$ connected and denote by $d$ the relative dimension of $f$. Define
\begin{center}
	$f^!:
	{}^{\delta}C_p(Y\hM_q,*) \to {}^{\delta}C_{p+d}(Y,\hM_{q-d},*\otimes \detcotgb_f)$
\end{center}
as follows. If $f(y)=x$, then $(f^!)^x_y= \res_{\kappa(y)/\kappa(x)}$. Otherwise, $(f^!)^x_y=0$. If $Y$ is not connected, take the sum over each connected component.
\par With cohomological convention, we obtain:
\begin{center}
	$f^!:{}^{\delta}C^p(X,\hM_q,*) \to {}^{\delta}C^{p}(Y,\hM_{q},*)$.
\end{center}

\end{paragr}

\begin{rem}
The fact that the morphism $f$ is (essentially) smooth implies that there are no multiplicities to consider. The case where the morphism $f$ is flat will be handled later.
\end{rem}

\begin{paragr}{\sc Multiplication with units}
Let $a_1,\dots, a_n$ be global units in $\mathcal{O}_X^*$. Define
\begin{center}
	$[a_1,\dots, a_n]:
	{}^{\delta}C_p(X,\hM_q,*) \to 
	{}^{\delta}C_p(X,\hM_{q+n},*)$
\end{center}
as follows. Let $x$ be in $X_{(p)}$ and $\rho\in \hM(\kappa(x),*)$. We consider $[a_1(x),\dots, a_n(x)]$ as an element of ${\KMW (\kappa(x),*)}$.
If $x=y$, then put $[a_1,\dots , a_n]^x_y(\rho)=[a_1(x),\dots , a_n(x)]\cdot \rho) $. Otherwise, put $[a_1,\dots , a_n]^x_y(\rho)=0$.

\end{paragr}

\begin{paragr}{\sc Multiplication with $\eta$}
Define
\begin{center}
	
	$\eta:
	{}^{\delta}C_p(X,\hM_q,*)\to 
	{}^{\delta}C_p(X,\hM_{q-1},*)$
\end{center}
as follows. If $x=y$, 
then $\eta^x_y(\rho)
=\gamma_{\eta}(\rho)$. 
Otherwise, $\eta^x_y(\rho)=0$.

\end{paragr}

\begin{paragr}{\sc Boundary maps} \label{BoundaryMaps}
Let $X$ be a scheme of finite type over $k$, let $i:Z\to X$ be a closed immersion and let $j:U=X\setminus Z \to X$ be the inclusion of the open complement. We will refer to $(Z,i,X,j,U)$ as a boundary triple and define
\begin{center}

	$\partial=\partial^U_Z:
	{}^{\delta}C_p(U,\hM_q,*) \to 
	{}^{\delta}C_{p-1}(Z,\hM_q,*)$
\end{center}
by taking $\partial^x_y$ to be as the definition in \ref{2.0.1} with respect to $X$. The map $\partial^U_Z$ is called the boundary map associated to the boundary triple, or just the boundary map for the closed immersion $i:Z\to X$.

\end{paragr}

\begin{paragr}{\sc Generalized correspondences} \label{GeneralizedCorr}
We will use the notation
\begin{center}
	$\alpha  : [X,*] \bullet\!\!\! \to [Y,*]$

\end{center}
or simply 
\begin{center}
	$\alpha  : X \bullet\!\!\! \to Y$
	
\end{center}
to denote maps of complexes which are sums of composites of the five basics maps $f_*$, $ g^!$, $ [a]$, $\eta$, $ \partial$ for schemes over $k$.
Unlike Rost in \cite[§3]{Rost96}, we look at these morphisms up to quasi-isomorphisms so that a morphism $\alpha  : X \bullet\!\!\! \to Y$ may be a weak inverse of a well-defined morphism of complexes.

\end{paragr}

\subsection{Compatibilities} \label{Compatibilities}
In this section we establish the basic compatibilities for the maps considered in the last section. Fix $\hM$ a homological Milnor-Witt cycle module.

\begin{prop} \label{Prop4.1}

\begin{enumerate}
	\item Let $f:X\to Y$ and $f':Y\to Z$ be two morphisms of schemes. Then
	\begin{center}
		$(f'\circ f)_*=f'_* \circ f_*$.
		
	\end{center} 
	\item Let $g:Y\to X$ and $g':Z\to Y$ be two essentially smooth morphisms. Then:
	\begin{center}
		
		$(g\circ g')^!=g'^!\circ g^!$.
	\end{center}
	\item Consider a pullback diagram
	\begin{center}
		
		$\xymatrix{
			U \ar[r]^{g'} \ar[d]_{f'} & Z \ar[d]^f \\
			Y \ar[r]_{g} & X
		}$
	\end{center}
	with $f,f',g,g'$ as previously. Then
	\begin{center}
		
		$g^!\circ f_* = f'_* \circ g'^!$.
	\end{center}
\end{enumerate}
\end{prop}
\begin{proof}
\begin{enumerate}
	\item This is clear from the definition and by \ref{itm:R1b'}.

	\item The claim is trivial by \ref{itm:R1a'} (again, there are no multiplicities).
	\item This reduces to the rule \ref{itm:R1c'} (see \cite[Proposition 4.1]{Rost96}).
\end{enumerate}
\end{proof}

\begin{prop} \label{Lem4.2}

Let $f:Y\to X$ be a morphism of schemes. If $a$ is a unit on $X$, then
\begin{center}
	
	$f_*\circ [\tilde{f}^!(a)]=[a]\circ f_*$
\end{center} where $\tilde{f}^!:\mathcal{O}^*_X\to \mathcal{O}^*_Y$ is induced by $f$.

\end{prop}
\begin{proof}

This comes from \ref{itm:R2b'}.
              and the claim follows from R2c.
\end{proof}

\begin{prop} \label{Lem4.3}

Let $a$ be a unit on a scheme $X$.
\begin{enumerate}
	\item Let $g:Y\to X$ be an essentially smooth morphism. One has 
	\begin{center}
		
		$g^!\circ [a]=[\tilde{g}^!(a)]\circ g^!$.
	\end{center}
	\item Let $(Z,i,X,j,U)$ be a boundary triple. One has
	\begin{center}
		
		$\partial^U_Z \circ [\tilde{j}^!(a)]=-  [\tilde{i}^!(a)]\circ \partial^U_Z$.
	\end{center}
	Moreover,
	\begin{center}
		$\partial^U_Z \circ \eta=-\eta \circ \partial^U_Z$.
		
	\end{center}

\end{enumerate}
\end{prop}
\begin{proof}
The first result comes from \ref{itm:R2a'}, the second from \ref{itm:R2b'} and \ref{itm:R3e'}.
\end{proof}

\begin{prop} \label{Prop4.4}

Let $h:X\to X'$ be a morphism of schemes. Let $Z'\hookrightarrow X'$ be a closed immersion. Consider the induced diagram given by $U'=X'\setminus Z'$ and pullback:

\begin{center}
	$\xymatrix{
		Z \ar@{^{(}->}[r] \ar[d]^{f} & X \ar[d]^h & U \ar@{_{(}->}[l] \ar[d]^{g} \\
		Z' \ar@{^{(}->}[r] & X'  & U'. \ar@{_{(}->}[l]
	}$
\end{center}
\begin{enumerate}
	\item If $h$ is proper, then
	\begin{center}
		
		$f_*\circ \partial^U_Z = \partial^{U'}_{Z'} \circ g_*.$
	\end{center}
	\item If $h$ is essentially smooth, then
	\begin{center}
		
		$f^!\circ \partial^{U'}_{Z'} = \partial^U_Z \circ g^!.$
	\end{center}
\end{enumerate}

\end{prop}
\begin{proof}
This will follow from Proposition \ref{Prop4.6}.
\end{proof}

\begin{lm}

\label{RostLem4.5}
Let $g:Y\to X$ be a smooth morphism of schemes of finite type over a field of constant fiber dimension $1$, let $\sigma:X\to Y$ be a section of $g$ and let $t\in \mathcal{O}_Y$ be a global parameter defining the subscheme $\sigma(X)$. Moreover, let $\tilde{g}:U \to X$ be the restriction of $g$ where $U=Y\setminus \sigma(X)$ and let $\partial$ be the boundary map associated to $\sigma$. Then
\begin{center}
	$\partial \circ {[t]} \circ \tilde{g}^!=(\Id_X)_*$ and $\partial \circ \tilde{g}^!=0$,
\end{center}

\end{lm}
\begin{proof}
Same as \cite[Proposition 6.5]{Feld1}.
\end{proof}

\begin{prop}\label{Prop4.6}
\begin{enumerate}
	\item Let $f:X\to Y$ be a proper morphism of schemes. Then
	\begin{center}
		
		$d_Y\circ f_*= f_*\circ d_X$.
	\end{center}
	\item Let $g:Y\to X$ be an essentially smooth morphism. Then
	\begin{center}
		
		$g^!\circ d_X=d_Y \circ g^!$.
	\end{center}
	\item Let $a$ be a unit on $X$. Then
	\begin{center}
		
		$d_X \circ [a]=-[a]\circ d_X$.
	\end{center}
	Moreover,
	\begin{center}
		
		$d_X \circ \eta=-\eta \circ d_X$.
	\end{center}
	\item Let $(Z,i,X,j,U)$ be a boundary triple. Then
	\begin{center}
		
		$d_Z\circ \partial^U_Z=-\partial^U_Z\circ d_U$.
	\end{center}
\end{enumerate}

\end{prop}

\begin{proof} Same as \cite[Proposition 4.6]{Rost96}. The first assertion comes from Proposition \ref{Prop4.1}.1 and \ref{itm:R3b'}, the second  from \ref{itm:R3c'}, Proposition \ref{Prop4.1}.3, Proposition \ref{Prop4.6}.1 and \ref{itm:R3a'} (note that the proof is actually much easier in our case since there are no multiplicities to consider). The third assertion follows from the definitions and Proposition \ref{Lem4.3}.2, the fourth from the fact that $d\circ d=0$.

\end{proof}

\begin{paragr}

According to the previous section (see Proposition \ref{Prop4.6}), the morphisms $f_*$ for $f$ proper, $g^*$ for $g$ essentially smooth, multiplication by $[a_1,\dots , a_n]$ or $\eta$, $\partial^U_Y$ (anti)commute with the differentials, and thus define maps on the homology groups.

\end{paragr}
\begin{paragr}\label{LocalizationSequence}

Let $(Z,i,X,j,U)$ be a boundary triple. We can split the complex ${}^{\delta}C_*(X,\hM_q,*)$ as
\begin{center}
	
	${}^{\delta}C_*(X,\hM_q,*)={}^{\delta}C_*(Z,\hM_q,*)\oplus {}^{\delta}C_*(U,\hM_q,*)$
\end{center}so that there is a long exact sequence
\begin{center}
	
	$\xymatrix{
		\dots \ar[r]^-\partial &
		{}^{\delta}A_p(Z,\hM_q,*) \ar[r]^{i_*} 
		& {}^{\delta}A_p(X,\hM_q,*) \ar[r]^{j^!} 
		& {}^{\delta}A_p(U,\hM_q,*) \ar[r]^{\partial} & {}^{\delta}A_{p-1}(Z,\hM_q,*) \ar[r]^-{i_*} 
		& \dots.
	}$
\end{center}

\end{paragr}

\section{Main properties}
\label{sec:Cycle_to_homotopy_modules}

\subsection{Homotopy invariance} \label{HomotopyInvarianceSection}
As in \cite[Section 9]{Feld1}, we define a coniveau spectral sequence that will help us reduce the homotopy invariance property to the known case \ref{itm:(H')}.

\par Let $(X,\delta)$ be a dimensional scheme and $\pi:V\to X$ an essentially smooth morphism of schemes. Fix $\hM$ a homological MW-cycle module.
A $\delta$-flag on $X$ is a sequence $(Z_p)_{p\in \ZZ}$ of reduced closed subschemes of $X$ such that, for all $p$, we have $Z_p\subset  Z_{p+1}$ and $\delta(Z_p) \leq p$. The set of $\delta$-flags $\Flag_{\delta}(X)$ on $X$ is non-empty and cofiltered (i.e. two elements admit a lower bound).
\par Let $\mathfrak{Z}=(Z_p)_{p\in \ZZ}$ be a $\delta$-flag of $X$ and define $\pi^!\mathfrak{Z}=(\pi^!Z^p)_{p\in \ZZ}$ a $\delta$-	flag over $V$ by ${\pi^!Z^p=V\times_X Z^p}$. For $p,q\in \ZZ$, define
\begin{center}
	
	$D^{\mathfrak{Z}}_{p,q}={}^{\delta}A_{p+q}(\pi^*Z_p, \hM, *)$,
	\\ $E^{\mathfrak{Z}}_{p,q}={}^{\delta}A_{p+q}(\pi^*Z_p- \pi^*Z_{p-1},\hM,*)$.
\end{center}
We have a long exact sequence
\begin{center}
	
	$\xymatrix{
		\dots \ar[r] 
		& D^{\mathfrak{Z}}_{p-1,q+1} \ar[r]^-{i_{p,*}} 
		& D^{\mathfrak{Z}}_{p,q} \ar[r]^{j^!_p} 
		& E^{\mathfrak{Z}}_{p,q} \ar[r]^{\partial_p} 
		& D^{\mathfrak{Z}}_{p-1,q} \ar[r] & \dots
	}$
\end{center}

so that $(D_{p,q}^\mathfrak{Z},E_{p,q}^\mathfrak{Z})_{p,q\in \ZZ}$ is an exact couple where $j_p^*$ and $i_{p,*}$ are induced by the canonical immersions. By the general theory (see \cite[Chapter 3]{McCleary01}), this defines a spectral sequence that converges to $A_{p+q}(V,\hM,*)$ because the $E_{p,q}^1$-term is bounded (since the dimension of $V$ is finite).
\par For $p,q\in \ZZ$, denote by
\begin{center}
	$D_{p,q}^{1,\pi}=\displaystyle \lim_{\mathfrak{Z}\in \Flag_{\delta}(X)} D_{p,q} ^{\mathfrak{Z}}$,\\
	$E_{p,q}^{1,\pi}=\displaystyle \lim_{\mathfrak{Z}\in \Flag_{\delta}(X)} E^{\mathfrak{Z}}_{p,q}$
\end{center}
where the colimit is taken over the flags $\mathfrak{Z}$ of $X$ (see Proposition \ref{Prop4.4} for functoriality). Since filtered direct limits are exact in the derived category of abelian groups, the previous spectral sequences give the following theorem.
\begin{thm} \label{SpectralSequenceHmtpInvariance}
	
	We have the convergent spectral sequence
	\begin{center}
		
		$E_{p,q}^{1,\pi} \Rightarrow {}^{\delta}A_{p+q}(V,\hM,*)$.
	\end{center}
	
\end{thm}

We need to compute this spectral sequence. This is done in the following theorem.
\begin{thm} \label{SpectralSequenceComputation}
	
	For $p,q\in \ZZ$, we have a isomorphism
	\begin{center}
		
		$E_{p,q}^{1,\pi}\simeq \displaystyle \bigoplus_{x\in X_{(p)}}{}^{\delta}A_q(V_x,\hM,*)$.
	\end{center}
\end{thm}

\begin{proof} Denote by $\mathcal{I}_p$ the set of pairs $(Z,Z')$ where $Z$ is a reduced closed subscheme of $X$ of $\delta$-dimension $p$ and $Z'\subset Z$ is a closed subset containing the singular locus of $Z$. Notice that any such pair $(Z,Z')$ can be (functorially) extended into a $\delta$-flag of $X$. Moreover, for any $x$ in $X$, consider $\overline{ \{x\}}$ the reduced closure of $x$ in $X$ and $\mathfrak{F}(x)$ be the set of closed subschemes $Z'$ of $\overline{ \{x\}}$ containing its singular locus. The following equalities are all canonical isomorphisms by continuity (see Definition \ref{df:continuity}):
	\begin{center}
		\begin{tabular}{lll}
			
			$E_{p,q}^{1,\pi}$ &$ \simeq$ & $\displaystyle \lim_{{\mathfrak{Z}}\in \Flag_{\delta}(X)} {}^{\delta}A_q(V\times_X(Z_p-Z_{p-1}),\hM,*)$ \\
			&    $\simeq$  
			& $\displaystyle \lim_{(Z,Z')\in \mathcal{I}_p} {}^{\delta}A_q(V\times_X(Z-Z'),\hM,*)$ \\
			&   $\simeq$   &
			$\displaystyle \bigoplus_{x\in X_{(p)}} \displaystyle \lim_{Z'\in \mathfrak{F}(x)} {}^{\delta}A_q(V\times_X(\overline{ \{x\} }-Z'),\hM,*)$ \\
			&  $\simeq$    
			& $\displaystyle \bigoplus_{x\in X_{(p)}} {}^{\delta}A_q(V_x,\hM,*).$
		\end{tabular}
		
	\end{center}
	Note that the proof gives an explicit construction of the isomorphism so that we may call it \textit{canonical}.
	
\end{proof}

\begin{thm}[Homotopy Invariance] \label{HomotopyInvariance}
	Let $X$ be a scheme, $V$ a vector bundle of rank $n$ over $X$, $\pi:V\to X$ the canonical projection. Then, for every $q\in \ZZ$, the canonical morphism
	\begin{center}
		$\pi^!:{}^{\delta}A_q(X,\hM,*)\to {}^{\delta}A_{q+n}(V,\hM,*)$
		
	\end{center}
	is an isomorphism.
\end{thm}

\begin{proof}
	
	From a noetherian induction and the localization sequence \ref{LocalizationSequence}, we can reduce to the case where $V=\AA^n_X$ is the affine trivial vector bundle of rank $n$. Moreover, we may assume $n=1$ by induction. With the previous notations, Theorem \ref{SpectralSequenceHmtpInvariance} gives the spectral sequence
	\begin{center}
		
		$E_{p,q}^{1,\pi} \Rightarrow A^{\delta}_{p+q}(V,\hM,*)$
	\end{center}
	where $E_{p,q}^{1,\pi} $ is (abusively) defined as previously, but twisted accordingly.
	By Theorem \ref{SpectralSequenceComputation}, the page $E_{p,q}^{1,\pi}$ is isomorphic to $\bigoplus_{x\in X_{(p)}}{}^{\delta}A_q(V_x,\hM,*)$. According to the property \ref{itm:(H')}, this last expression is isomorphic (via the map $\pi$) to $\bigoplus_{x\in X_{(p)}}{}^{\delta}A_q(\Spec \kappa(x),\hM,*)$. Using again \ref{SpectralSequenceComputation}, this group is isomorphic to $E_{p,q}^{1,\Id_X}$, which converge to ${}^{\delta}A_{p+q}(X,\hM,*)$. By Proposition \ref{Prop4.4}, the map $\pi$ induces a morphism of exact couples
	\begin{center}
		$(D_{p,q}^{1,\Id_X},E_{p,q}^{1,\Id_X})\to (D_{p,q+n}^{1,\pi},E_{p,q+n}^{1,\pi})$
	\end{center} 
	hence we have compatible isomorphisms on the pages which induce the pullback
	\begin{center}
		${\pi^!:{}^{\delta}A_q(X,\hM,*)\to {}^{\delta}A_{q+n}(V,\hM,*)}$.
	\end{center}
\end{proof}

\subsection{Homotopy cycle complex}
\label{subsec_homotopy_cycle_complex}

Fix $\hM$ a homological MW-cycle module. If $X$ is a scheme and $\pi:V\to X$ a vector bundle over $X$, we have proved \ref{HomotopyInvarianceSection} that the canonical map
\begin{center}
	$\pi^!:{}^{\delta}C_p(X,\hM_q,*)\to {}^{\delta}C_{p}(V,\hM_{q},*)$
	
\end{center}
is a \underline{quasi}-isomorphism of complexes. 
\par Unfortunately, we would rather have a true isomorphism. In the following, we take some ideas from \cite{Ivorra14} and construct a new cycle complex associated to the Rost-Schmid complex which do have this property.

\begin{paragr}{\sc Notation}
	Denote by $\pi_{n,n-1}: \AA^n \to \AA^{n-1}$ the canonical projection induced by the inclusion of rings $\ZZ[t_1,\dots , t_{n-1}] \to \ZZ[t_1,\dots , t_{n}]$. For $X$ a scheme, we denote by $\pi_{X,n,n-1}$ the map $\Id_X \times \pi_{n,n-1}$ and by 
	\begin{center}
		$\pi_{X,n}: X \times_S \AA^n_S \to X$
	\end{center}
	the projection.
	
\end{paragr}
\begin{df}\label{df:functorial_complex}
	Let $X$ be a scheme, $*$ be a line bundle over $X$, $q$ an integer, and $\hM$ be a homological MW-cycle module. We denote by $\hC_*(X,\hM, \cL)$ the graded complex of abelian groups 
	\begin{center}
		$\hC^*(X,\hM_q, *)
		:=
		\colim_{n \in \NN} {}^{\delta}C^*(X\times \AA^n,\hM_q, * )$
	\end{center}
	with structural maps given by the smooth pullbacks $\pi^!_{X,n,n-1}$.
\end{df}

\begin{rem}
	Note that $\hC$ is naturally equipped with differential maps and is indeed a complex. Moreover, we have a graded map
	\begin{center}
		$\widetilde{\alpha}_X: {}^{\delta}C^*(X,\hM_q, *) \to \hC^*(X,\hM_q, *)$.
	\end{center}
\end{rem}

\begin{paragr}{\sc Basic maps}
	\label{Basic_maps_homotopy_RS_complex}
	
	The basics maps defined in \ref{FiveBasic} extend naturally to the complex $\hC^*(X,\hM,*)$.
	\par Let $X,Y$ be two schemes over $S$, and $f:Y\to X$ a smooth morphism of relative dimension $r$. By functoriality of pullbacks, we have a commutative diagram
	
	\begin{center}
		
		$
		\xymatrixcolsep{14mm}
		\xymatrix{
			{}^{\delta}C^*(X,\hM_q, *)
			\ar[r]
			\ar[d]^{f^!}
			&
			\dots
			\ar[r]
			&
			{}^{\delta}C^*(\AA_X^{n-1},\hM_q, *)
			\ar[r]^-{\pi^!_{X,n,n-1}}
			\ar[d]^{(f\times_S \Id_{\AA^{n-1}})^!}
			&
			{}^{\delta}C^*(\AA_X^{n},\hM_q, *)
			\ar[d]^{(f\times_S \Id_{\AA^{n}})^!}
			&
			\\
			{}^{\delta}C^*(Y,\hM_q, *)
			\ar[r]
			&
			\dots
			\ar[r]
			&
			{}^{\delta}C^*(\AA_Y^{n-1},\hM_q, *)
			\ar[r]_-{\pi^!_{Y,n,n-1}}
			&
			{}^{\delta}C^*(\AA_Y^{n},\hM_q, *)
			&
		}$
	\end{center}

	and thus, by taking colimits, we obtain a pullback
	\begin{center}
		$f^!:\hC^*(X,\hM_q,*) \to \hC^*(Y,\hM_q,*)$.
	\end{center}
	Similarly, if $f:Y\to X$ is a proper morphism of relative dimension $s$, we have a pushforward map:
	\begin{center}
		
		$f_*:\hC^p(Y,\hM_q,*\otimes \detcotgb_f)
		\to 
		\hC^{p+s}(X,\hM_{q-s}, *)$.
	\end{center}

	Let $a\in \cO^{\times}_X$ be a global unit. We have a commutative diagram 
	
	\begin{center}
		
		$\xymatrix{
			{}^{\delta}C^*(X,\hM_q, *)
			\ar[r]
			\ar[d]^{[a]}
			&
			\dots
			\ar[r]
			&
			{}^{\delta}C^*(X\times_S \AA^{n-1},\hM_{q+1}, *)
			\ar[r]^{\pi^!_{Y,n,n-1}}
			\ar[d]^{[\pi^*_{n-1}(a)]}
			&
			{}^{\delta}C^*(X\times_S \AA^{n},\hM_{q+1}, *)
			\ar[r]
			\ar[d]^{[\pi^*_{n}(a)]}
			&
			\dots
			\\
			{}^{\delta}C^*(X,\hM_{q+1}, *)
			\ar[r]
			&
			\dots
			\ar[r]
			&
			{}^{\delta}C^*(X\times_S \AA^{n-1},\hM_{q+1}, *)
			\ar[r]_{\pi^!_{X,n,n-1}}
			&
			{}^{\delta}C^*(X\times_S \AA^{n},\hM_{q+1}, *)
			\ar[r]
			&
			\dots
		}$
	\end{center}
	
	and thus an induced morphism
	\begin{center}
		$[a]:\hC^p(X,\hM_q,*) \to \hC^p(X,\hM_{q+1},*)$.
	\end{center}
	Finally, let $i:Z\to X$ be a closed immersion and let $j:U=X\setminus Z \to X$ be the inclusion of the open complement. We have a \textit{boundary map}
	\begin{center}

		$\partial=\partial^U_Z:\hC^p(U,\hM_q,*) \to \hC^{p+1}(Z,\hM_q,*)$.
	\end{center}
	\begin{cor}
		All compatibilities between the basic maps proved in \ref{Compatibilities} hold true for $\hC^*(X,\hM_q, *)$.
		
	\end{cor}
\end{paragr}

\begin{paragr}
	Note that, by definition, the canonical pullback map 
	\begin{center}
		$\hC^*(X,\hM_q,*) \to \hC^*(X\times_S \AA^n_S,\hM_q,*)$
	\end{center}
	is an isomorphism. We are going to prove that this homotopy invariance property hold true in general.

\end{paragr}

\begin{paragr}
	\label{HomotopyComplexIsNisnevich}
	Fix $X$ a scheme over $S$, $q$ an integer, and $*$ a line bundle over $X$. We denote by
	\begin{center}
		$\HH: U \mapsto \hC^*(U,\hM_q,*)$
	\end{center}  the presheaf on the small Nisnevich site $X_{\text{Nis}}$. We prove that this is in fact a Nisnevich sheaf. Indeed, consider a Nisnevich square, i.e. a cartesian diagram
	\begin{center}
		
		$\xymatrix{
			U_V
			\ar[r]
			\ar[d]
			&
			V
			\ar[d]^{\phi}
			\\
			U
			\ar[r]_{j}
			&
			X
		}$
	\end{center}
	with $j$ an open immersion and $\phi$ an étale map such that the induced morphism $V\setminus U_V \to X\setminus U$ is an isomorphism. We have decompositions as a direct sum of abelian groups
	\begin{center}
		${}^{\delta}C_p(X,\hM_q,*)=
		C^{\delta}_p(U,\hM_q,*)
		\oplus
		C^{\delta}_p(Z,\hM_q,*)$
	\end{center}
	\begin{center}
		$C^{\delta}_p(V,\hM_q,*)=
		C^{\delta}_p(U_V,\hM_q,*)
		\oplus
		C^{\delta}_p(Z_V,\hM_q,*)$
	\end{center}
	where $Z$ (resp. $Z_V$) is the closed complement of $U$ in $X$ ($U_V$ in $V$) with its reduced scheme structure. Since the squares

	\begin{center}
		
		$\xymatrix{
			U_V \times_S \AA^n
			\ar[r]
			\ar[d]
			&
			V\times_S \AA^n
			\ar[d]^{\phi}
			\\
			U\times_S \AA^n
			\ar[r]_{j}
			&
			X\times_S \AA^n
		}$
	\end{center}
	are also Nisnevich squares, we have a similar decompositions
	
	\begin{center}
		$C^{\delta}_{p+n}(X\times_S \AA^n,\hM_q,*)=
		C^{\delta}_{p+n}(U\times_S \AA^n,\hM_q,*)
		\oplus
		C^{\delta}_{p+n}(Z\times_S \AA^n,\hM_q,*)$,
	\end{center}
	\begin{center}
		$C^{\delta}_{p+n}(V\times_S \AA^n,\hM_q,*)=
		C^{\delta}_{p+n}(U_V\times_S \AA^n,\hM_q,*)
		\oplus
		C^{\delta}_{p+n}(Z_V\times_S \AA^n,\hM_q,*)$.
	\end{center}
	Taking the colimit over $n$, we get decompositions as a direct sum of abelian groups
	\begin{center}
		$\hC^p(X,\hM_q,*)=
		\hC^p(U,\hM_q,*)
		\oplus
		\hC^p(Z,\hM_q,*)$
	\end{center}
	\begin{center}
		$\hC^p(V,\hM_q,*)=
		\hC^p(U_V,\hM_q,*)
		\oplus
		\hC^p(Z_V,\hM_q,*)$
	\end{center}
	and it follows that the presheaf $\HH$ on $X_{\text{Nis}}$ is a Nisnevich sheaf.
	
\end{paragr}

\begin{cor}[Strong Homotopy invariance]
	\label{HomotopyInvarianceDeformedComplex}
	Let $X$ be a scheme over $S$, $q$ an integer, $*$ a line bundle over $X$ and $\pi:V\to X$ a vector bundle. Then the pullback morphism
	\begin{center}
		$\pi^!:\hC^*(X,\hM_q,*) \to \hC^*(V,\hM_q,*)$
	\end{center}
	is an isomorphism.
	
\end{cor}
\begin{proof}
	One can deduce from the previous paragraph that the presheaf on $X_{\text{Nis}}$:
	\begin{center}
		$U\mapsto \hC^*(V_U,\hM_q,*)$
	\end{center}
	is also a Nisnevich sheaf. Thus we are reduced to proving the result locally in the case of a trivial vector bundle, which is true by definition. 
\end{proof}

\begin{paragr}
	We now prove that there is a strong homotopy equivalence between the Rost-Schmid complex $C$ and its associated homotopy complex $\hC$.
	The trick is to use the formulas of \cite[§9.1]{Rost96}, conveniently adapted, to define an explicit homotopy of complexes (these ideas also appear in \cite[§3]{Ivorra14} and \cite[Theorem 5.38]{MorelLNM}).
	
\end{paragr}

\begin{paragr}
	In the following, we recall that we use the notation of Rost
	\begin{center}
		$X \bulleto Y $
	\end{center}
	to denote maps of complexes
	\begin{center}
		${}^{\delta}C^*(X,\hM_q,*) \to {}^{\delta}C^*(Y,\hM_q,*)$
	\end{center}
	which are sums of composites of the usual maps $f_*, g^*, [a], \eta, \partial$ defined previously. 
	\par For instance, if $X$ is a scheme over $S$, we define the composites
	\begin{center}
		$\xymatrixcolsep{3pc}\xymatrix{
			r_{X,1}:X\times \AA^1
			\ar@{{*}->}[r]^-{j^!}
			&
			X\times(\AA^1-\{0\})
			\ar@{{*}->}[r]^{[-1/t]}
			&
			X\times(\AA^1-\{0\})
			\ar@{{*}->}[r]^-{\partial_{\infty}}
			&
			X,
		}$
	\end{center}
	\begin{center}
		$\xymatrixcolsep{3pc}\xymatrix{
			H_{X,1}:X\times \AA^1
			\ar@{{*}->}[r]^-{p_2^!}
			&
			X\times(\AA^1\times \AA^1-\Delta)
			\ar@{{*}->}[r]^{[s-t]}
			&
			X\times(\AA^1\times \AA^1-\Delta)
			\ar@{{*}->}[r]^-{p_{1*}}
			&
			X\times \AA^1.
		}$
	\end{center}
	Here $t$ is the coordinate of $\AA^1=\Spec \ZZ[t]$ and $s,t$ are the coordinates of $\AA^1\times \AA^1=\Spec \ZZ[s]\times \Spec \ZZ[t]$. Moreover $\Delta=\{s-t=0\}$ is the diagonal, $j$ is the standard inclusion, $p_1$ and $p_2$ are given by the standard projections and $\partial_{\infty}$ is induced by $X=X\times \infty\subset X\times \PP^1-\{0\})$ with open complement $X\times (\AA^1-\{0\})$. 
	
	\par By induction, we define
	\begin{center}
		$r_{X,n}= 
		r_{X,n-1}\circ r_{X\times_S \AA^{n-1}}$
	\end{center}
	and
	\begin{center}
		$H_{X,n}=
		H_{X\times_S \AA^{n-1}}
		+
		\pi^!_{X,n,n-1} \circ H_{X,n-1} \circ  r_{X\times_S \AA^{n-1}}$.
	\end{center}
	Note that we have
	
	\[\begin{array}{rcl}
		H_{X,n} \circ \pi^!_{X,n,n-1}
		&
		=
		&
		H_{X\times_S \AA^{n-1}}
		\circ
		\pi^!_{X,n,n-1}
		+
		\pi^!_{X,n,n-1}
		\circ
		H_{X,n-1}
		\circ
		r_{X\times_S \AA^{n-1}}
		\circ
		\pi^!_{X,n,n-1}
		\\
		{}
		&
		=
		&
		0
		+
		\pi^!_{X,n,n-1}
		\circ
		H_{X,n-1}
		\\
		{}
		&
		=
		&
		\pi^!_{X,n,n-1}
		\circ
		H_{X,n-1}	
	\end{array}\]
	and thus we obtain a graded map of degree $-1$
	\begin{center}
		$\widetilde{H}_X =
		\colim_n H_{X,n}: \hC(X,\hM_q,*)
		\to \hC(X,\hM_q,*)$.
	\end{center}	
	Since $r_{X,n}\circ \pi^!_{X,n,n-1}=r_{X,n-1}$, we also have a map
	\begin{center}
		$\widetilde{r}_X: \hC^*(X,\hM_q,*) \to C_{\delta}^*(X,\hM_q,*)$.
	\end{center}

\end{paragr}
\begin{thm}
	\label{thm_h_data}
	Keeping the previous notations, the maps 
	\begin{center}
		$\widetilde{\alpha}_X: {}^{\delta}C^*(X,\hM_q, *) \to \hC_*(X,\hM_q,*)$
	\end{center}
	
	\begin{center}
		$\widetilde{H}_X: \hC^*(X,\hM_q,*)
		\to \hC^*(X,\hM_q,*)$.
	\end{center}	
	
	\begin{center}
		$\widetilde{r}_X: \hC^*(X,\hM_q,*) \to {}^{\delta}C^*(X,\hM_q,*)$.
	\end{center}
	form an $h$-data (in the sense of \cite[§9.1]{Rost96}), i.e. they satisfy the following properties:
	\begin{itemize}
		\item $\widetilde{H}_X \circ \widetilde{\alpha}_X = 0$,
		\item $\widetilde{r}_X \circ \widetilde{\alpha}_X=\Id$,
		\item $\delta(\widetilde{H}_X)=\Id-\widetilde{\alpha}_X *\circ \widetilde{r}_X$
	\end{itemize}
	where $\delta(\widetilde{H}_X)=\partial^{n-1}\circ \widetilde{H}_X + \widetilde{H}_X\circ \partial^{n+1}$.
	
\end{thm}

\begin{proof}

	The first point is easy to check by definition. The second point follows from \ref{RostLem4.5}. For the last point, we prove that
	\begin{equation}
		\label{delta=1-a*r}
		\delta(H_{X,1})=\Id_{X\times \AA^1}-\pi^!_{X,1} *\circ r_{X,1}
	\end{equation}
	where $\delta(H_{X,1})=\partial^{n-1}\circ H_{X,1} + H_{X,1}\circ \partial^{n+1}$. 
	Consider the decomposition
	\begin{center}
		$\xymatrix{
			p_{1*}:X
			\ar@{{*}->}[r]^-{q_*}
			&
			X\times \AA^1 \times \PP^1
			\ar@{{*}->}[r]^-{\bar{p}_{1*}}
			&
			X\times \AA^1
		}$
	\end{center}
	where $q$ is the inclusion and $\bar{p}_1$ is the projection. Since $\bar{p}_{1*}$ is a morphism of complexes, we have
	\begin{center}
		$\delta(H_{X,1})=\bar{p}_{1*}\circ \delta(q_*)\circ [s-t]\circ p_2^*$.
	\end{center}
	Moreover,
	\begin{center}
		$\delta(q_*)=(i_{\Delta})_*\circ \partial_{\Delta}+ (i_{\infty})_*\circ \partial_{\infty}$
	\end{center}
	where $i_{\Delta}:X\times \Delta \to X\times \AA^1\times \PP^1$, $i_{\infty}:X\times \AA^1 \times \infty \to X \times \AA^1\times \PP^1$ are the inclusions and $\partial_{\Delta},\partial_{\infty}$ are the boundary maps for 
	$X\times \Delta\to X \times \AA^1 \times \AA^1$,
	$X\times \AA^1 \times \infty \to X \times (\AA^1 \times \PP^1-\Delta)$, respectively.
	\par Since $s-t$ is a parameter for $\Delta$, one finds,
	\begin{center}
		$\bar{p}_{1*}\circ (i_{\Delta})_*\circ \partial_{\Delta} \circ [s-t] \circ p_2^*=\Id$
	\end{center}
	according to Lemma \ref{RostLem4.5}.
	\par Let $W=\AA^1 \times \PP^1- (\Delta \cup \AA^1\times 0)$. Moreover, let $\tilde{p}_2$ be the restriction of $p_2$ to $U=X\times (W-\AA^1\times \infty)$ and let $\tilde{\partial}_\infty$ be the boundary map corresponding to the inclusion $X\times \AA^1 \times \infty \to X \times W$. Then
	\begin{center}
		$\partial_\infty \circ [s-t] \circ p^*_2 = \tilde{\partial}_\infty \circ [s-t] \circ \tilde{p}_2^*$.
	\end{center}
	Since $(s-t)/(-t)$ is a unit on $W$ with constant value $1$ on $X\times \AA^1 \times \infty$, one has
	\begin{center}
		$\tilde{\partial}_\infty \circ [s-t]\circ p_2^*=\tilde{\partial}_\infty \circ [-t] \circ \tilde{p}^*_2$
	\end{center}
	and thus
	\begin{center}
		$\bar{p}_{1*}\circ (i_\infty)_*\circ \partial_\infty \circ [s-t]\circ p^*_2
		=
		-\pi^!_{X,1} *\circ r_{X,1}$
	\end{center}
	Putting things together yields \ref{delta=1-a*r}. By induction, we obtain, for any number $n$,
	\begin{center}
		$\delta(H_{X,n})=\Id_{X\times \AA^1}-\pi^!_{X,n} *\circ r_{X,n}$
	\end{center}
	from which we can easily conclude by taking colimits.

\end{proof}

\begin{rem}
	The proof of Theorem \ref{thm_h_data} can also be used to give a second proof of the homotopy invariance property (see Theorem \ref{HomotopyInvariance}) which does not use spectral sequences and construct an explicit section of the quasi-isomorphism
	\begin{center}
		$\pi^!:{}^{\delta}C^*(X,\hM_q,*) 
		\to
		{}^{\delta}C^*(\AA^1_X, \hM_q, *)$.
	\end{center}
\end{rem}

\begin{rem}
	\label{homology_rost_schmid_coincide}
	By Theorem \ref{thm_h_data}, the (co)homology of $\hC(X,\hM_q, *)$ coincides with the (co)homology of $C(X,\hM_q,*)$.
\end{rem}

\subsection{Gysin morphisms for regular embeddings}\label{GysinMorphisms}

\begin{paragr}
	\label{GyMoRegEmbedding}
	We define Gysin morphisms for regular closed immersions and prove functoriality theorems. As always, the main tool is the deformation to the normal cone. We fix once a for all $\hM$ a homological MW-cycle module, $q$ an integer and $*$ a line bundle over $S$.
	
\end{paragr}

\begin{paragr}
	Let $i:Z\to X$ be a closed immersion. Let $t$ be a parameter of $\AA^1$ and let 
	\begin{center}
		${q:X\times_S (\AA^1\setminus \{0\}) \to X}$
	\end{center}
	be the canonical projection. Denote by $D=D_ZX$ the deformation space such that ${D=U \sqcup N_ZX}$ where $U=X\times_S(\AA^1\setminus \{0\})$ (see \cite[§10]{Rost96} for more details).
	
	Consider the morphism 
	\begin{center}
		
		$J(X,Z)=J_{Z/X}:\hC^*(X,\hM_q,*)\to \hC^*(N_ZX,\hM_q,*)$
		
	\end{center} defined by the composition:
	\begin{center}
		
		$
		\xymatrix{
			\hC^{p}(X,\hM_q,*) \ar[r]^-{q^!} \ar@{-->}[d]^{J_{Z/X}} &
			\hC^{p}(U,\hM_q,*) \ar[d]^{[t]} \\ \hC^{p}(N_ZX,\hM_q,*)  &
			\hC^{p}(U,\hM_{q+1},*) \ar[l]^-\partial 
		}
		$
		
	\end{center}
	where the multiplication with $t$ is twisted by the isomorphism $\cotgb_{U/X}\simeq \AA^1_U$ (which only depends on $t$) and $\partial$ is the boundary map as in \ref{BoundaryMaps}. This defines a morphism (also denoted by $J(X,Z)$ or $J_{Z/X}$) by passing to homology.
	\par 
	Assume moreover that $i:Z\to X$ is regular of codimension $m$, the map $\pi : N_ZX \to Z$ is a vector bundle over $X$ of dimension $m$. By homotopy invariance (Theorem \ref{HomotopyInvarianceDeformedComplex}), we have an isomorphism
	\begin{center}
		
		$\pi^!:\hC^p(Z,\hM_q,*) \to \hC^{p}(N_ZX,\hM_q,*)$
	\end{center}
	where we have used the canonical isomorphism $\pi^!(N_ZX)=\cotgb_{N_ZX/Z}$. Denote by $r_{Z/X}=(\pi^!)^{-1}$ its inverse.
	
\end{paragr}
\begin{df}
	
	With the previous notations, we define the map
	\begin{center}
		$i^!:\hC^p(X,\hM_q,*)\to \hC^{p}(Z,\hM_q,*)$
		
	\end{center}
	by putting $i^!=r_{Z/X}\circ J_{Z/X} $ and call it the {\em Gysin morphism of $i$}.
\end{df}

The following lemmas are needed to prove functoriality of the previous construction (see Theorem \ref{GysinFunctoriality}).

\begin{lm}\label{Lem11.3}
	
	Let $i:Z\to X$ be a regular closed immersion and $g:V\to X$ be an essentially smooth morphism. Denote by $N(g)$ the projection from $N(V,{V\times_X Z}) =N_Z(X)\times_X V$ to $N_ZX$. Then
	\begin{center}
		
		$J(V,V\times_X Z)\circ g^!=N(g)^!\circ J(X,Z)$,
	\end{center}
	(up to the canonical isomorphism induced by $\cotgb_{N(V,Y\times_X V)/N(X,Y)}\simeq \cotgb_{V/X}\times_V N(V,Y\times_X V)$).
\end{lm}
\begin{proof}
	See \cite[Lemma 11.3]{Rost96}. This follows from our Proposition \ref{Prop4.1}.2, Proposition \ref{Prop4.4}.2 and Proposition \ref{Lem4.3}.2.
\end{proof}

\begin{lm}\label{Lem11.4}
	Let $Z\to X$ be a closed immersion and let $p:X\to Y$ be essentially smooth. Suppose that the composite 
	\begin{center}
		
		$q:N_ZX\to Z \to X \to Y$
	\end{center}
	is essentially smooth of same relative dimension as $p$. Then 
	\begin{center}
		
		$ J(X,Z)\circ p^!=q^!$
	\end{center}
	(up to the canonical isomorphism induced by $\cotgb_{N_ZX/Y}\simeq  \cotgb_{V/Y}\times_X N_ZX$).
\end{lm}
\begin{proof} Same as \cite[Lemma 11.4]{Rost96} except that Rost only needs the composite morphism
	\begin{center}
		$\xymatrix{
			f:D(X,Z) \ar[r] &  X\times_S \AA^1 \ar[r]^{p\times \Id} & Y\times_S \AA^1
		}$ 
	\end{center} to be flat. We need moreover (in order to use our Proposition \ref{RostLem4.5}) the fact that $f$ is essentially smooth which is true because it is flat and its fibers are essentially smooth.
\end{proof}
\begin{lm}\label{Lem11.6}
	Let $\rho:T\to T'$ be a morphism, let $T_1',T_2'\subset T'$ be closed subschemes and let $T_i=T\times_{T'} T_i'$ for $i=1,2$. \par Put $T_3=T\setminus (T_1\cup T_2)$, $T_0=T_1 \cap  T_2$, $\tilde{T_1}=T_1\setminus T_0$, $\tilde{T_2}=T_2\setminus T_0$ and let $\partial^3_1, \partial^1_0, \partial^3_2, \partial^2_0$ be the boundary maps for the closed immersions
	\begin{center}
		
		$\tilde{T_1}\to T \setminus T_2$, $T_0\to T_1$, $\tilde{T_2}\to T\setminus T_1$, $T_0\to T_2$,
	\end{center}
	respectively. Then
	\begin{center}
		
		$0= \partial^1_0\circ \partial^3_1+\partial^2_0\circ \partial^3_2:[T_3,*]\bullet\!\!\! \to [T_0,*]$
	\end{center}
	at the homology level (recall the notation introduced in \ref{GeneralizedCorr}).
\end{lm}
\begin{proof}
	
	Corresponding to the set theoretic decomposition of $T$ we have
	\begin{center}
		
		$\hC_*(T,\hM_q,*)=\hC_*(T_0,\hM_q,*)\oplus
		\hC_*(\tilde{T_1},\hM_q,*) \oplus 
		\hC_*(\tilde{T_2},\hM_q,*)\oplus \hC_*(T_3,\hM_q,*).$
	\end{center}
	Then $d_T\circ d_T=0$ gives the result.
\end{proof}

\begin{lm}\label{Lem11.7}
	Let $T=\overline{ D}=\overline{ D}(X,Y,Z)$, $T_1=\overline{ D}|(\{0\}\times \AA^1)$, $T_2=\overline{ D}(\AA^1\times \{0\})$ where $\overline{ D}$ is the double deformation cone (see \cite[§10.5]{Rost96}). We keep the notations of Lemma \ref{Lem11.6}. Then ${T_3=X\times_k (\AA^1\setminus \{0\})\times (\AA^1\setminus \{0\})}$ and $T_0=\overline{ D}|\{0,0\}$. Let $\pi:T_3\to X$ be the projection and let $t,s$ be the coordinates of $\AA^2$, so that $T_1=\{t=0\}$, $T_2=\{s=0\}$.
	\par Let $Z\to Y\to X$ be regular closed immersions. Then 
	\begin{center}
		
		$\partial^1_0\circ \partial^3_1\circ [t,s]\circ \pi^! = J(N_YX,N_YX|Z)\circ J(X,Y)$,\\
		$\partial^2_0\circ \partial^3_2\circ [s,t] \circ \pi^! = J(N_ZX,N_ZY)\circ J(X,Z)$.
	\end{center}
\end{lm}
\begin{proof}
	
	Same as \cite[Lemma 11.7]{Rost96} except one has to be careful with the twists (this uses Lemma \ref{Lem11.3}).
\end{proof}

\begin{thm} \label{Thm13.1} \label{GysinFunctoriality}
	Let $l:Z\to Y$ and $i:Y\to X$ be regular closed immersions of respective codimension $n$ and $m$. Then $i\circ l$ is a regular closed immersion of codimension $m+n$ and (up to a canonical isomorphism)
	\begin{center}
		
		$(i\circ l)^!= l^!\circ i^!$ 
	\end{center}
	as morphism $\hC^p(X,\hM_q,*)\to \hC^{p}(Z,\hM_q,*)$.
\end{thm}
\begin{proof}
	
	The assertion follows from \ref{Lem11.3}, \ref{Lem11.6}, \ref{Lem11.7} and \ref{Lem11.4} as in \cite[Theorem 13.1]{Rost96}.
\end{proof}

We conclude with one interesting property.

\begin{prop}[Base change for regular closed immersions]
	\label{MWmodBaseChangeRegular}
	Consider the cartesian square
	\begin{center}
		$\xymatrix{
			Z \ar[r]^i \ar[d]^g & X \ar[d]^f \\
			Z'\ar[r]^{i'} & X'
		}$
	\end{center}
	where $i,i'$ are regular closed immersions and $f,g$ are proper morphisms. Suppose moreover that we have a canonical isomorphism of virtual vector spaces $N_{Z'}X' \simeq N_ZX\times_X X'$. Then (up to the canonical isomorphism induced by the previous isomorphism)
	\begin{center}
		
		$g_*\circ i^! = i'^!\circ f_*$.
	\end{center}
\end{prop}
\begin{proof}
	It suffices to prove that the following diagram is commutative (recall the notation introduced in \ref{GeneralizedCorr}):
	\begin{center}
		
		$\xymatrixcolsep{4pc}\xymatrix{
			X \ar@{{*}->}[r]^-{q^!} \ar@{{*}->}[d]^{f_*} \ar@{}[rd]|-{(1)} 
			&
			X\times (\AA^1\setminus\{0\}) \ar@{{*}->}[r]^{[t]} \ar@{{*}->}[d]^{(f\times \Id)_*} \ar@{}[rd]|-{(2)} &
			X\times (\AA^1\setminus\{0\})  \ar@{{*}->}[r]^-{\partial} \ar@{{*}->}[d]^{(f\times \Id)_*} \ar@{}[rd]|-{(3)} 
			&
			\NN_ZX   \ar@{{*}->}[r]^-{\simeq} \ar@{{*}->}[d]^{h_*} \ar@{}[rd]|-{(4)} 
			&
			Z \ar@{{*}->}[d]^{g_*} \\
			X' \ar@{{*}->}[r]^-{q'^!} 
			&
			X'\times (\AA^1\setminus\{0\}) \ar@{{*}->}[r]^{[t]}  
			&    
			X'\times (\AA^1\setminus\{0\})  \ar@{{*}->}[r]^-{\partial'} & \NN_{Z'}X'   \ar@{{*}->}[r]^-{\simeq} & Z'
		}$
	\end{center}
	with obvious notations (see the definition of Gysin morphisms).
	The (cartesian) squares (1) and (4) commute by the base change theorem for essentially smooth morphisms (see Proposition \ref{Prop4.1}.3). The squares (2) and (3) commute by Proposition \ref{Lem4.2}.1 and Proposition \ref{Prop4.4}.1, respectively.
\end{proof}

\subsection{Gysin morphisms for lci projective morphisms}
\begin{paragr}

	We define Gysin morphisms for lci projective morphisms
	and prove functoriality theorems (see \cite[§5]{Deg08n2} for similar results in the classical oriented case). One could also define Gysin morphisms for morphisms between two essentially smooth schemes as in \cite[§12]{Rost96}.

\end{paragr}

\begin{lm}\label{GysinLem5.9}
	
	Consider a regular closed immersion $i:Z\to X$ and a natural number $n$. Consider the pullback square
	
	\begin{center}
		
		$\xymatrix{
			\PP^n_Z \ar[r]^l \ar[d]^q & \PP^n_X \ar[d]^p \\
			Z \ar[r]^i & X.
		}$
	\end{center}
	Then $l^!\circ p^!=q^!\circ i^!$ (up to the canonical isomorphism induced by $\cotgb_q-q^!\NN_i\simeq l^!\cotgb_p-\NN_l$).
\end{lm}
\begin{proof}
	
	This follows from the definitions and Lemma \ref{Lem11.3}.
\end{proof}

\begin{lm}\label{GysinLem5.10}
	Consider a natural number $n$ and an essentially smooth scheme $X$. Let $p:\PP_X^n\to X$ be the canonical projection. Then for any section $s:X\to \PP^n_X$ of $p$, we have $s^!p^!=\Id$ (up to a canonical isomorphism).
\end{lm}
\begin{proof}
	We can assume that $X=\Spec k$, then apply rule \ref{itm:R3d'} (see also the proof of Lemma \ref{RostLem4.5}).
\end{proof}

\begin{lm} \label{GysinLem5.11}
	Consider the following commutative diagram:
	\begin{center}
		
		$\xymatrix{
			& \PP^n_X \ar[rd]^p & \\
			Y \ar[ru]^i \ar[rd]_{i'} & & X \\
			& \PP^m_X \ar[ru]_q & 
		}$
	\end{center}
	where $i,i'$ are regular closed immersions and $p,q$ are the canonical projection. Then $i^!\circ p^!={i'}^!\circ q^!$ (up to the canonical isomorphism induced by $\NN_{i'}-(i')^!\cotgb_q \simeq \NN_i - i^!\cotgb_p$).
\end{lm}
\begin{proof} Let us introduce the following morphisms:
	\begin{center}
		
		$\xymatrix{
			&   &   \PP_X^n \ar[rd]^p&    \\
			Y \ar@/^/[rru]^i \ar[r]|-\nu \ar@/_/[rrd]_{i'} & \PP^n_X\times_X \PP^m_X \ar[ru]_{q'} \ar[rd]^{p'} & & X \\
			&   &   \PP^m_X \ar[ru]_q &   
		}$
	\end{center}
	Applying Proposition \ref{Prop4.1}.2, we are reduced to prove $i^!=\nu^!q'^!$ and ${i'}^!=\nu^!p'^!$. In other words, we are reduced to the case $m=0$ and $q=\Id_X$.
	\par In this case, we introduce the following morphisms:
	\begin{center}
		
		$\xymatrix{
			Y  \ar@/^/^i[rrd]  \ar[rd]^s \ar@2{-}[rdd] &      &        \\
			& \PP^n_Y \ar[r]_l \ar[d]^q & \PP^n_X  \ar[d]_p \\
			& Y     \ar[r]^{i'}  &  X.    
		}$
	\end{center}
	Then the lemma follows from Lemma \ref{GysinLem5.9}, Lemma \ref{GysinLem5.10} and Theorem \ref{Thm13.1}.
\end{proof}

\begin{df}
	Let $f:Y\to X$ be a projective lci morphism of smooth $S$-schemes. Consider a factorization 
	$\xymatrix{
		Y
		\ar[r]^i
		&
		\PP^n_X
		\ar[r]^p
		&
		X
	}$
	of $f$ into a regular closed immersion followed by the canonical projection.
	We define the Gysin morphism associated to $f$ as the morphism 
	\begin{center}
		
		$f^!= i^!\circ p^!:\hC^*(X,\hM_q,*)\to \hC^{*}(Y,\hM_q,*)$.
	\end{center}
\end{df}

\begin{prop}
	
	Consider projective morphisms $\xymatrix{Z \ar[r]^g & Y \ar[r]^f & X}$.
	Then (up to the canonical isomorphism induced by $\cotgb_{fg}\simeq \cotgb_g + g^!\cotgb_f$):
	\begin{center}
		$g^!\circ f^!=(f\circ g)^!$.
		
	\end{center} 
\end{prop}
\begin{proof} We choose a factorization $\xymatrix{Y \ar[r]^i & \PP^n_X \ar[r]^p & X}$ 
	(resp. $\xymatrix{Z\ar[r]^j & \PP^m_X \ar[r]^q & X}$) of $f$ (resp. $fg$) and we introduce the diagram
	\begin{center}
		
		$\xymatrix{
			&     &    \PP^m_X  \ar@/^2pc/	[rrddd]^q &    &   \\
			&      &   \PP^n_X\times_X \PP^m_X \ar[u]|-{p'} \ar[rd]^{q'} &   &   \\
			&  \PP^m_Y \ar[rd]^{q''} \ar[ru]^{i'}&                 &   \PP^n_X \ar[rd]|-{p}&   \\
			Z \ar[rr]|-{g} \ar[ru]|-{k}  \ar@/^2pc/[rruuu]^j &  &      Y   \ar[ru]^i      \ar[rr]|-{f}    &        &  X
		}$ 
	\end{center}
	in which $p'$ is deduced from $p$ by base change, and so on for $q'$ and $q''$. Then, by using the factorization given in the preceding diagram, the proposition follows from \ref{GysinLem5.9}, \ref{Thm13.1}, \ref{GysinLem5.11} and \ref{Prop4.1}.2.
\end{proof}

\begin{rem}
	\label{rem_functoriality_lci}
	As a consequence of the previous results, we can define pullback for any morphisms between smooth schemes. Indeed, a map $f:Y \to X$ of smooth schemes can always be factorized as follows:
	\begin{center}
		$\xymatrix{Y \ar[r]^-i & Y \times X \ar[r]^-p & X}$ 
	\end{center}
	where $i(y)=(y,f(y))$ and $p(y,x)=x$. Thus $f$ is lci.
	
\end{rem}

\begin{prop}[Base change for lci morphisms]\label{MWmodBaseChangelci}
	
	Now consider a cartesian square of schemes
	\begin{center}
		
		$\xymatrix{
			X' \ar[r]^{f'} \ar[d]_{g'} & Y' \ar[d]^g \\
			X \ar[r]_f & Y
		}$
	\end{center}
	with $f$ proper 

	and $g$ lci.
	Assume that the square is \textit{transversal} in the sense that the canonical map $f^{'!}\cotgb_g
	\to 
	\cotgb_{g'}
	$ is an isomorphism.
	We have (up to the canonical isomorphism induced by the previous isomorphism):
	\begin{center}
		$f'_* \circ g'^! = g^!\circ f_*$
	\end{center}
	which is compatible with horizontal and
	vertical compositions of transversal squares.
\end{prop}

\begin{proof}
	
	It suffices to consider the case where $g$ is the projection of a projective bundle or a regular closed immersion. It follows from Proposition \ref{Prop4.1} in the first case and from Proposition \ref{MWmodBaseChangeRegular} in the second.
\end{proof}

\begin{prop}[Localization]
	\label{MWmod-localization}
	For any closed immersion $i:Z\to X$ with complementary open immersion $j:U\to X$, one has a long
	exact sequence
	\begin{align*}
		&\hC^n(Z,\hM_q,*) \xrightarrow{i_*}
		\hC^n(X,\hM_q,*)
		\xrightarrow{j^!} 
		\hC^n(U,\hM_q,*)
		\xrightarrow{\partial_i} \hC^{n+1}(Z,\hM_q,*),
	\end{align*}
	which is functorial with respect to proper covariance and lci contravariance.

\end{prop}

\section{Rost transform}

\label{sec:Rost_transform}

\subsection{Twisted Borel-Moore homology theories}

\begin{num}\label{num:bivariant}
	Let $S$ be a scheme and $\E$ be a spectrum $S$.
	For any $S$-scheme $X$ with structural morphism $p$,
	one puts $\E_X=p^*\E$, therefore producing a cartesian section of $\SH$ over the category
	of $S$-schemes.
	
	Recall from \cite{DJK} that for any morphism of finite type\footnote{Thanks
		to the work of A. Khan, we do not need the separatedness condition any more.} $f:Y \rightarrow X$,
	and pair $(n,v) \in \ZZ \times \uK(X)$, one defines the \emph{bivariant theory} of $Y/X$ in degree $n$ and twist $v$ with coefficients in $\E$ as:
	\begin{equation}\label{eq:bivariant}
		\E_n(Y/X,v)=[f_!(\Th(v))[n],\E_X].
	\end{equation}
	Actually, this is the $n$-th homotopy group of the space $\underline \E(Y/X,v)=\Map(f_!(\Th(v)),\E_X)$.
	Note the cohomology represented by $\E$ is then obtained from the particular case $Y=X$:
	$\E^n(X,v)=\E_{-n}(X/X,-v)$.
	
	Theses theories satisfies a \emph{twisted} variant of Fulton-MacPherson bivariant formalism,
	for which it is useful to introduce the following categories:
\end{num}
\begin{df}
	\label{df:virt}
	We define the \emph{fibred category of virtual schemes} $\virt$ as the category whose objects are pairs
	$(X,v)$ where $X$ is a scheme and $v$ a virtual vector bundle over $X$,
	and with morphisms $(Y,w) \rightarrow (X,v)$ the pairs $(f,\phi)$ 
	where $f:Y \rightarrow X$ is a morphism of schemes
	and $\phi:w \rightarrow f^{-1}(v)$ an isomorphism of virtual vector bundles over $Y$. The composition is defined in an obvious way.
	
	Moreover, let $(Y,w)$ and $(X,v)$ are two objects of $\virt$. A \emph{twisted morphisms} 
	\begin{center}
		$(Y,w) \rightarrow (X,v)$
	\end{center}
	is defined to be a pair $(f,\phi)$ where $f:Y \rightarrow X$ is an essentially lci morphism of schemes
	and $\phi:w \rightarrow \cotgb_f+f^{-1}(v)$ an isomorphism of virtual vector bundles over $Y$.
	The composition of two twisted morphisms
	\begin{center}
		$(Z,w) \xrightarrow{(g,\psi)} (Y,v) \xrightarrow{(f,\phi)} (X,u)$
	\end{center}
	is given by the composition of morphisms of schemes $fg=f \circ g$ and by the following composite isomorphism:
	$$
	w \xrightarrow{\psi} \cotgb_g+g^{-1}v \xrightarrow{Id+g^{-1}(\phi)} \cotgb_g+g^{-1}\cotgb_f+g^{-1}f^{-1}(u) \simeq \cotgb_{fg}+(fg)^{-1}(u)
	$$
	where the last isomorphism is induced by the canonical isomorphism of virtual vector bundles
	(coming from the distinguished triangles of cotangent complexes).
	This defines another category denoted by $\virtlci$. 
	
	Note that we will denote a morphism $(f,\phi)$ of one of the above two types simply by $f$
	when $\phi$ is clear.
\end{df}
Obviously, the category $\virt$ is fibred over the category of schemes.
We will see the application of these categories to bivariant formalism in \Cref{ex:hlg&coh}.

\begin{df}
	A \emph{transversal square} will be the data of four morphisms pictured as follows
	\begin{equation}\label{eq:transv_sq}
		\begin{split}
			\xymatrix@=14pt{
				(Y,v')\ar^{(g,\phi')}[r]\ar_{(q,\psi')}[d] & (T,u')\ar^{(p,\psi)}[d] \\
				(X,v)\ar^{(f,\phi)}[r] & (S,u)
			}
		\end{split}
	\end{equation}
	such that the horizontal maps are twisted and the vertical one are plain,
	the underlying square of morphisms of schemes is not only commutative but cartesian,
	the canonical map induced by this cartesian square
	$$
	\xi:q^{-1}\tau_f \rightarrow \tau_g
	$$
	is an isomorphism
	and the following diagram of isomorphism of virtual bundles is commutative:
	$$
	\xymatrix@R=12pt@C=24pt{
		v'\ar^{\psi}[rr]\ar_{\psi'}[d] && \cotgb_g+g^{-1}(u')\ar^{\xi^{-1}+q^{-1}(\psi)}[d] \\
		q^{-1}(v)\ar^-{q^{-1}(\phi)}[r] & q^{-1}\cotgb_f+q^{-1}f^{-1}(u)\ar^\sim[r] & q^{-1}\cotgb_f+g^{-1}p^{-1}(u).
	}
	$$
\end{df}

\begin{ex}
	Consider a cartesian square of schemes
	$$
	\xymatrix@=14pt{
		Y\ar^{g}[r]\ar_{q}[d]\ar@{}|\Delta[rd] & T\ar^{p}[d] \\
		X\ar_{f}[r] & S
	}
	$$
	such that $f$ and $g$ are lci. Recall one says this square is tor-independent, or rather $X$ and $T$ are tor-independent over $S$,
	if $\Tor_i^S(\mathcal O_X,\mathcal O_T)=0$ if $i>0$ (e.g. $f$ or $p$ flat).
	Note that if one assume that $f$ is \emph{syntomic} (\emph{i.e.} lci and flat), then $g$ is automatically lci (in fact syntomic), and the square if tor-independent.
	
	It follows from \cite[Cor. 2.2.3]{Illusie} that the canonical
	map of cotangent complexes $q^*\cotg_f \rightarrow \cotg_g$ is a quasi-isomorphism, and therefore, the induced map
	of the respective associated virtual vector bundles $\xi:q^{-1}\tau_f \rightarrow \tau_g$ is an isomorphism.
	One deduces that for any virtual vector bundle $u$ over $S$, the following square is transversal in the above sense:
	$$
	\xymatrix@R=12pt@C=18pt{
		(Y,\tau_g+u_Y)\ar^-{(g,\Id)}[r]\ar_{(q,\xi^{-1}+\Id)}[d]\ar@{}|{\tilde \Delta}[rd] & (T,u_T)\ar^{(p,\Id)}[d] \\
		(X,\tau_f+u_X)\ar_-{(f,\Id)}[r] & (S,u)
	}
	$$
	where $u_X$ and the like denotes the pullback of $u$ over $X$. This is the prototype of all our transversal squares.

\end{ex}

\begin{df}\label{df:twisted_(co)hlg}
	Let $\Sany$ be a full subcategory of the category of schemes
	(see the forthcoming example for illustration).
	We let $\virt\Sany$ (resp. $\virtlci\Sany$) be the full subcategory
	of $\virt$ (resp. $\virtlci$) made by pairs $(X,v)$ such that $X$ is in $\Sany$,
	and $\virt\Sany^{prop}$ (resp. $\virtlci\Sany^{prop}$) the corresponding subcategory
	with same objects and morphisms $(f,\phi):(Y,v) \rightarrow (X,v)$ such that $f$
	is proper.
	
	Let $\A$ be an abelian category (resp. $\T$ a stable $\infty$-category). Denote by $\A^\ZZ$ be the category of graded objects of $\A$ (see \cite[tag 09MF]{stacks_project}). 
	
	An \emph{abelian (resp. derived) twisted Borel-Moore homology theory} $H$ over $\Sany$ with coefficients $\A$
	(resp. $\T$) will be the data
	of two functors:
	\begin{align*}
		\text{(natural) } H_*:\virt\Sany^{prop} \rightarrow \A^\ZZ (\text{resp. } \T), \qquad
		\text{(exceptional) } H^!:{\virtlci\Sany}^{op} \rightarrow \A^\ZZ (\text{resp. } \T)
	\end{align*}
	which agree on objects $H_*(X,v)=H^!(X,v)=:H(X,v)$,
	and for an isomorphism $\varphi:v \rightarrow v$ of virtual bundles over $X$,
	$H^!(\Id_X,\varphi)=H_*(\Id_X,\varphi)^{-1}$,
	and such that the following properties hold:
	\begin{enumerate}
		\item \emph{Base change}. For any transversal square $\Delta$ as in \eqref{eq:transv_sq},
		whose underlying schemes are in $\Sany$ and vertical morphisms are proper,
		one has a base change formula 
		\begin{center}
			$(f,\phi)^!(p,\psi)_*=(p,\psi')_*(g,\phi')^!$
			(resp. an equivalence $Ex_*^!(\Delta))$
		\end{center}
		which is compatible with horizontal and
		vertical compositions of transversal squares.
		\item \emph{Localization}. For any morphisms $i:(Z,v_0) \rightarrow (X,v)$
		and $j:(U,v_\eta) \rightarrow (X,v)$ in $\virt\Sany$ such that $i$ is a closed immersion
		and $j$ is the complementary open immersion, one has a long (resp. homotopy) 
		exact sequence
		\begin{align*}
			&H_n(Z,v_0) \xrightarrow{i_*} H_n(X,v)
			\xrightarrow{j^!} H_n(U,v_\eta) \xrightarrow{\partial_i} H_{n-1}(Z,v_0), \\
			\text{resp. } &H(Z,v_0) \xrightarrow{i_*} H(X,v)
			\xrightarrow{j^!} H(U,v_\eta),
		\end{align*}
		which is functorial with respect to proper covariance and lci contravariance
		(resp. functoriality follows from property (1)). 
	\end{enumerate}

	\par If the coefficients are omitted, then $\A$ is the category
	of abelian group $\ab$ (resp. $\T$ is the $\infty$-category of spectra $\Sp$).
	We will usually refer to $H$ as $H_*$ and call it a BM-homology $H_*$ for short.
	\par A morphism between two twisted Borel-Moore homology theory $\alpha:(H_*,H^!) \to (H'_*, H'^!)$ is given by two natural transformations $\alpha_*:H_* \to H'_*$ and $\alpha^!:H^! \to H'^!$ which agree on objects $\alpha_*(X,v) = \alpha^!(X,v)$.
	\begin{center}
		
	\end{center} 
\end{df}

\begin{rem}
	\begin{enumerate}
		\item Of course, it is somewhat cumbersome to have to make explicit the isomorphisms
		of virtual bundles, that nevertheless rigorously always appear.
		Most of the time we will thus ignore them and simply denote by $f$ what should be $(f,\phi)$
		as the isomorphism is clear. Therefore, the functoriality of BM-homology
		will be denoted by $f_*$ and $p^!$.
		\item The derived version is stronger than the abelian one,
		as given any t-structure on $\T$, one recovers from a derived twisted BM-homology
		an abelian one after application of the associated homological functor
		$H_*$.\footnote{We are a bit fuzzy about the identifications appearing in the
			base change property but one can consider them as identities in practice.}
		\item Recall that in general,
		one can highlight four types of theories out of the six functors formalism of motivic $\infty$-categories
		(see \cite{DJK}),
		with special properties as follows:
		\begin{itemize} 
			\item cohomology (resp. homology) twisted theories are contravariant (resp. covariant)
			with respect to $\virt\Sany$ and covariant (resp. contravariant) with respect
			to $\virtlci\Sany^{prop}$. They satisfy the full base change formula as stated above,
			but localization must be restricted to \emph{special} regular closed immersions
			(such as closed immersions between smooth schemes over some base, or between regular schemes).
			\item Borel-Moore homology (resp. properly supported cohomology) twisted theories are covariant
			(resp. contravariant)
			with respect to $\virt\Sany^{prop}$ and contravariant (resp. covariant) with respect
			to $\virtlci\Sany$. They satisfy the full base change formula as stated above,
			and localization with respect to all closed immersions.\footnote{In particular,
				these theories only depend on the reduced scheme structure.}
		\end{itemize}
		\item Borel-Moore homology theories are a particular case of bivariant theories
		when one restrict the latter to morphisms with target a fixed (base) scheme.
		\item Cohomology of a scheme $X$ with support in a closed subscheme $Z$
		is a particular case of Borel-Moore homology: indeed, it is the BM-homology of $Z$, seen
		as an $X$-scheme (as remarked in \cite[1.2.5]{Deg16}). Therefore, the coniveau spectral
		sequences based on cohomology with support can always be reduced to the case
		of niveau spectral sequences based on Borel-Moore homology, up to re-indexing.
	\end{enumerate}
\end{rem}

\begin{ex}\label{ex:hlg&coh}
	Let $S$ be a fixed base scheme and $\E$ be a motivic spectrum over $S$
	(or more generally an object of $\T(S)$ for a triangulated motivic $\infty$-category $\T$).
	
	Consider the notation of \Cref{num:bivariant}.
	Then one deduces that $\underline \E(-/S,-)$ is a derived twisted BM-homology theory over $\sft_S$:
	the (natural) covariance and localization follow from the six functors formalism
	(see \cite[Section 2.1]{DJK}), the (exceptional) contravariance follows from the existence
	of trace maps (\emph{loc. cit.} Th. 4.2.1) and base change from \emph{loc. cit.} Prop. 4.2.2.
	
	As explained above, one deduces that $\E_*(-/S,-)$ is an abelian twisted BM-homology theory over $\sft_S$.

\end{ex}
\begin{rem}
	The cohomological case will be treated in Section \ref{sec:MWmodcoh}.	
\end{rem}

\begin{ex}\label{ex:hlg&MW-premodules}
	Let $S$ be a fixed base scheme and $M$ be a homological cycle MW-premodule over $S$ as in \Cref{defMWmodules}.
	\begin{enumerate}
		\item Then $(\spec E,v) \mapsto M(E,v)$ is in fact a twisted BM-homology over $\pts_S$
		with values in graded abelian groups. Note that this assertion actually corresponds
		to data \ref{itm:D1'}, \ref{itm:D2'} (resp. \ref{itm:D1}, \ref{itm:D2}) and relations \ref{itm:R1a'}, \ref{itm:R1b'}, \ref{itm:R1c'} (resp. \ref{itm:R1a}, \ref{itm:R1b}, \ref{itm:R1c}).
		\item Assume moreover that $(S,\delta)$ is a dimensional scheme and $M$ is a MW-module.
		Then the functor $(X,v) \mapsto A_*^\delta(X,v,M)$ is a twisted BM-homology
		over $\seft_S$ with values in graded abelian groups. Indeed, the base change property follows from Proposition \ref{MWmodBaseChangelci} and the localization property from Proposition \ref{MWmod-localization}.
		
	\end{enumerate}
\end{ex}

\begin{num}\label{num:action_cohtp_on_BMhlg}
	We will use a last structure on BM-homologies coming from an $S$-spectrum $\E$.
	We let $\cohtp^n(X,v)=[\un_X,\Th(v)[n]]_{\SH(X)}$ be the cohomology represented
	by the sphere spectrum --- called the \emph{stable $\AA^1$-cohomotopy theory}. The cup-product defines a $\ZZ \times \uK(X)$-graded ring
	structure on $\cohtp^*(X,*)$. Moreover, if $X/S$ is of finite type,
	there is a $\ZZ \times \uK(X)$-graded module structure on $\E_*(X/S,*)$ over this ring:
	$$
	\cohtp^n(X,v) \otimes \E_m(X/S,w) \rightarrow \E_{m-n}(X/S,w-v), x \otimes \rho \mapsto x.\rho.
	$$
	Note moreover that it follows easily from the six functors formalism that we have the following
	formulas:
	\begin{equation}\label{eq:action_cohtp_on_BMhlg}
		\begin{split}
			f^!(x.\rho)&=f^*(x).f^!(\rho) \\
			f_*\big(f^*(x).\rho\big)&=x.f_*(\rho) \\
			f_*(x.f^!(\rho))&=f_!(x).\rho
		\end{split}
	\end{equation}
	where on the third line, $f_!$ is the Gysin morphism on stable $\AA^1$-cohomotopy (\cite[4.3.3(iii)]{DJK}).
\end{num}

\subsection{Preparations}

We will give a series of interesting properties of BM-homology theories,
used in the proof of the main theorem of this section but also interesting on their own.
We start with the following definition (compare \cite[before Def. 2.6]{MorelLNM}, \cite[\S 4.3]{CD19}).
\begin{df}\label{df:continuity}
	Let $S$ be a scheme and $\hat H_*$ be an abelian homology theory over $\seft_S$.
	We say $\hat H_*$ is \emph{continuous} if for any pair $(X,v)$ in $\virt\seft_S$
	such that $(X,v)=\plim_{i \in I} (X_i,v_i)$, for any pro-object $(X_i,v_i)_{i \in I}$
	such that the transition maps $X_j \rightarrow X_i$ are affine \'etale,
	the canonical map induces an isomorphism:
	\begin{equation}\label{eq:continuity}
		\hat H_*(X,v) \simeq \ilim_{i \in I^{op}} \hat H_*(X_i,v_i).
	\end{equation}
\end{df}

The following proposition is classical (and will be used as in \cite[Def. 3.1.13]{BD1}).
\begin{prop}\label{prop:kan_extension_BMhlg}
	Let $\A$ be an cocomplete abelian category.
	Let $H_*$ be a BM-homology theory on $\sft_S$ with coefficients in $\A$.
	Then there exist a unique BM-homology theory $\hat H_*$ on $\seft_S$ such that
	\begin{center}
		$\hat H_*:(\virt\seft_S)^{prop} \rightarrow \A$ 
		\\ (resp. $\hat H^!:(\virtlci\seft_S)^{op} \rightarrow \A$)
	\end{center}
	is the right Kan extensions of  
	\begin{center}
		$\hat H_*:(\virt\sft_S)^{prop} \rightarrow \A$
		\\ (resp. $\hat H^!:(\virtlci\sft_S)^{op} \rightarrow \A$)
	\end{center}
	along the inclusion $\nu:\virt\sft_S \rightarrow \virt\seft_S$
	(resp. $\virtlci\sft_S \rightarrow \virtlci\seft_S$).\footnote{To make the convention clear, one has natural isomorphisms:
		$\hat H_* \circ \nu \xrightarrow \sim H_*, \hat H^! \circ \nu \xrightarrow \sim H^!$.
	}
	Moreover, the extension $\hat H_*$ is continuous.
\end{prop}
\begin{proof}
	Indeed, $H$ is contravariant with respect to \'etale morphisms
	and the \'etale contravariant functoriality is compatible both with the proper covariance, the lci contravariance and the residue morphisms of $H$.
	As any $(X,v)$ in $\virt\seft_S$ is a projective limit of a pro-object $(X_i,v_i)$ satisfying the condition in the previous definition,
	and such that $X_i/S$ is of finite type, one can take \eqref{eq:continuity} 
	as a definition of $\hat H_*$ and one can check that it gives the required right Kan extensions.
\end{proof}

\begin{rem}
	\label{Bivariant_groups}
	In the case of the BM-homology represented by a motivic spectrum $\E$,
	one can also derive the above proposition from the six functors formalism.
	In fact,  we have shown in \cite[App. B]{DJK} that the functoriality $f^!$ of $\SH$ can be extended 
	from $f:X \rightarrow S$ of finite type to the case where $f$ is essentially of finite type.
	Therefore, for $f$ essentially of finite type, on can extends formula \eqref{eq:bivariant} and define the \textit{bivariant groups}:
	\begin{equation}
		\E_{n}(X/S,v)=\Hom_{\T(X)}(\Th_X(v)[n],f^!\E).
	\end{equation}
	and this extends formula \eqref{eq:bivariant}.
	
	One can further show that the functor $f^!$ admits a left pro-adjoint\footnote{compare to the construction of Deligne in \cite{HartRD}}:
	$$
	f_!:\SH(X) \rightarrow \pro{\SH(S)},
	$$
	or more generally, the obvious functor $f^!:\pro{\SH(S)} \rightarrow \pro{\SH(X)}$ admits a left adjoint
	$\pro f_!$, and the above one is obtained after composing on the right with the constant pro-object functor.
\end{rem}

\begin{rem}\label{rem:ext_action_cohtp_on_BMhlg}
	By continuity of the stable homotopy category, the action of stable $\AA^1$-cohomotopy on 
	the BM-homology $\E_*$ on $\sft_S$ represented by a spectrum $\E$ over $S$
	(recalled in \Cref{num:action_cohtp_on_BMhlg}) naturally extends to the extension $\hat \E_*$
	on $\seft_S$ obtained in the previous proposition. By passing to the colimit,
	the same is true for the formulas \eqref{eq:action_cohtp_on_BMhlg}.
	
	Note in particular that when considering a point $x:\Spec(E) \rightarrow S$ in $\pts(S)$,
	we get a structure of $\cohtp^*(E,*)$-module on $\hat E_*(E/S,*)$ together with formulas
	\eqref{eq:action_cohtp_on_BMhlg}. In particular,
	according to Morel's theorem \cite[Cor. 1.25]{MorelLNM},
	we get an action of $K_n^{MW}(E)=\cohtp^{-n}(E,\tw n)$.
\end{rem}

\begin{num}
	Note that given a scheme $X$, while the group automorphism of the rank $0$ vector bundle over $X$ is just $0$,
	one has $\Aut_{\uK(X)}(0_X)=\GG(X)$. In particular, given any virtual bundle $v$ over $X$,
	and any invertible function $a \in \GG(X)$, one gets an automorphism
	$\tilde a:v\simeq v+0_X \rightarrow v+0_X \simeq v$.
\end{num}
\begin{df}\label{df:GW_action_BMhlg}
	Let $H_*$ be an $R$-linear homology theory on $\seft_S$.
	For any scheme $X/S$ essentially of finite type, and any $a \in \GG(X)$, one denote by
	$\gamma_{\ev a}:H_*(X,v) \rightarrow H_*(X,v)$ the map induced by $\tilde a$ using the covariant
	functoriality of $H_*$.
\end{df}
The following formulas are immediate from this definition:
\begin{equation}\label{eq:GW_PF}
	f_* \gamma_{\ev{f^*a}}=\gamma_{\ev{a}} \circ f_*, \ f^! \gamma_{\ev{a}}=\gamma_{\ev{f^*a}} \circ f^!.
\end{equation}
Moreover, if $Z \rightarrow X$ is a closed immersion and $U$ its associated open subscheme, letting $a_Z$ (resp. $a_U$) be the restriction of
$a$ to $Z$ (resp. $U$), one also get from the functoriality of localization long exact sequences,
and especially the anti-commutativity statement:

$$
\partial_{Z/X} \circ \gamma_{\ev{a_U}}=\gamma_{\ev{a_Z}} \circ \partial_{Z/X}.
$$

\begin{rem}\label{rem:compare_GW_actions}
	We have abused the notation here as the map $\gamma_{\ev a}$ a priori depends on the unit $a$,
	and not just on its class in $\GW(\GG(X))$.
	
	Consider however the case of the BM-homology $\E_*$ represented by a spectrum $\E$,
	and its extension $\hat \E_*$ to $\seft_S$ (\Cref{prop:kan_extension_BMhlg}).
	Then for any point $\Spec(E) \rightarrow S$, and any unit $a \in E^\times$,
	one can easily check the relation:
	$$
	\gamma_{\ev a}(x)=\ev a.x
	$$
	where on the right hand-side, $\ev a$ is the image of $a$ in $\cohtp^{-1}(X,\tw 1) \simeq \GW(E)$,
	and we have considered the action of stable $\AA^1$-cohomotopy on $\hat \E_*$ (\Cref{rem:ext_action_cohtp_on_BMhlg}).
	Indeed, after unraveling the definition, the main argument is \cite[6.1.3]{Morel_Pi}.
\end{rem}

\begin{num}
	Consider the assumptions and notations of the previous definition \ref{df:GW_action_BMhlg}.
	Let $X$ be a scheme essentially of finite over $S$, and $v$ a virtual bundle over $X$.
	Let $\mathcal L=\det(v)$, $r=\mathrm{rk}(v)$ be its determinant and rank. 
	A global non-zero section $l \in \mathcal L(X)-\{0\}=\mathcal L(X)^\times$ 
	of the line bundle $\mathcal L$
	uniquely corresponds to a trivialization: $\mathcal O_X \xrightarrow \sim \mathcal L$,
	that we still denote by $l$.
	Moreover, given $l, l' \in \mathcal L(X)^\times$, there exists a unique invertible function
	$\lambda \in \GG(X)$ such that $l'=\lambda.l$. In fact, $\GG(X)$
	acts faithfully on $\mathcal L(X)-\{0\}$.
\end{num}
\begin{prop}\label{prop:twists_hlg_semilocal}
	Consider the above notation and assume that $X$ is connected semi-local.
	Then any virtual bundle $v$ over $X$ has constant rank $r$
	and there exists an isomorphism $\varphi:v \rightarrow \tw r$ of virtual bundles.
	Choose such an isomorphism $\varphi$,
	and let $l_\varphi=\det(\varphi):\mathcal L \rightarrow \mathcal O_X$ be the induced isomorphism.
	We consider the map:
	$$
	(*):H_n(X,\tw r) \otimes_{\ZZ[\GG(X)]} \ZZ[\mathcal L(X)^\times] \rightarrow H_n(X,v),
	x \otimes l \mapsto \gamma_{\ev{l_\varphi \circ l}} \circ (Id_X,\varphi^{-1})_*(x)
	$$
	where $\GG(X)$ acts on $H_n(X,\tw r)$ according to \Cref{df:GW_action_BMhlg}
	and on $\mathcal L(X)^\times$ as indicated above, and the composition
	$l_\varphi \circ l:\mathcal O_X \xrightarrow \sim \mathcal O_X$ is seen as an invertible
	function on $X$.
	Then (*) is an isomorphism, independent of the choice of $\varphi$.
\end{prop}
\begin{proof}
	The rank of a virtual bundle on $X$ is a Zariski locally constant function on $X$,
	so it is constant as $X$ is connected. As $X$ is semi-local, any vector bundle
	of constant rank is free of the same rank. This implies the first assertion.
	
	It is clear that $(*)$ is an isomorphism. 
	It remains to prove it is independent of $\varphi$. Coming back to the definition
	we must compare two isomorphisms of virtual bundles over $X$, which is expressed
	by proving the commutativity of the following diagram:
	$$
	\xymatrix@R=8pt@C=20pt{
		v\ar@{=}[r]\ar@{=}[d] & v \otimes \cO_X\ar^{Id \otimes l}[r]
		& v \otimes \mathcal{L}
		\ar^-{Id \otimes l_\varphi}[r]
		& v \otimes \cO_X\ar@{=}[r] & v\ar^-{\varphi^{-1}}[r] & <r>\ar@{=}[d] \\
		v\ar@{=}[r] & v \otimes \cO_X\ar^{Id \otimes l}[r]
		& v \otimes \mathcal{L}
		\ar^-{Id \otimes l_\psi}[r]
		& v \otimes \cO_X\ar@{=}[r] & v\ar^-{\psi^{-1}}[r] & <r>
	}
	$$
	where $\phi$ and $\psi$ are two choices of isomorphisms.
	As $X$ is semi-local, and all the virtual bundles involved are of the same rank,
	it suffices to apply the determinant functor to check the commutativity
	of this diagram. Then it is obvious (by definition of $l_\varphi$ and $l_\psi$).
\end{proof}

\begin{rem}
	The normalization procedure described in the previous proposition is quite general,
	and applies to any scheme $X$
	such that the graded determinant functor $\uK(X) \rightarrow \ZZ_X \times \uPic(X)$
	is an equivalence of categories and any functor $F:\uK(X) \rightarrow \smod R$,
	giving a canonical isomorphism (functorial in $v$):
	$$
	F(\tw{\rk(v)}) \otimes_{R[\GG(X)]} R[\Gamma(X,\det(v))^\times] \rightarrow F(v).
	$$
\end{rem}

Given the preceding proposition, it is relevant to introduce
another notation which will be closed to the convention used for (co)homology theories
like Chow-Witt groups, higher Witt and Grothendieck-Witt groups.
\begin{df}\label{df:twists_hlg_semilocal}
	Consider the assumptions of the previous proposition,
	so $X$ is semi-local.
	For a line bundle $\mathcal L$ over $X$ and an integer $n \in \ZZ$,
	one puts:
	$$
	H_{n,i}(X)\Gtw{\mathcal L}:=H_{n-2i}(X,\tw i) \otimes_{\ZZ[\GG(X)]} \ZZ[\mathcal L^\times].
	$$
\end{df}
Thus, the preceding proposition can be rewritten as an isomorphism:
$$
H_{n+2\rk(v),\rk(v)}(X)\Gtw{\det v} \simeq H_n(X,v).
$$

\begin{num}\label{num:A1invariance&splitting}
	Let $S$ be a scheme and $H_*$ be an abelian BM-homology on $\seft_S$.
	We assume that $H_*$ satisfies the homotopy invariance property for any $X/S$ essentially of finite type,
	any virtual bundle $v$ over $X$, the Gysin map induced by the projection $p:\AA^1_X \rightarrow X$ 
	$$
	p^!:H_*(X,v) \xrightarrow \sim H_*(\AA^1_X,p^{-1}v+<1>)
	$$
	is an isomorphism. Combining this property with the localization long exact sequence
	induced by the zero section $s_0:X \rightarrow \AA^1_X$, we get a split short exact sequence:
	$$
	0 \rightarrow H_n(X,v-\tw 1) \xrightarrow{q^!} H_n(\GGx X,q^{-1}v) \xrightarrow{\partial_t} H_{n-1}(\GGx X,v) \rightarrow 0
	$$
	where $\partial_t$ is the residue map associated with $s_0$, $q:\GGx X \rightarrow X$ the canonical projection.
	
	A canonical splitting of this short exact sequence is given by the retraction $s_1^!$ of $q^!$,
	$s_1:X \rightarrow \GGx X$ being the unit section. We let $\sigma_t^X$ be the corresponding section of $\partial_t$.
	
	Then it follows formally from this definition and the properties of a BM-homology
	that for $f:Y \rightarrow X$ proper (resp. lci), one has:
	\begin{equation}\label{eq:splitting_functorial}
		\begin{split}
			(1_{\GG} \times f)_* \sigma_t^Y&=\sigma_t^X \circ f_* \\
			\text{resp. } (1_{\GG} \times f)^! \sigma_t^X&=\sigma_t^Y \circ f^!.
		\end{split}
	\end{equation}
\end{num}
\begin{df}
	Consider the previous hypothesis and notation. Let $u \in \GG(X)$ be a (global) unit on $X$,
	corresponding to the morphism of schemes $u_X:X \rightarrow \GGx X$.
	We then define an action of $u$ on $H_*(X,v)$ by the formula:
	$$
	\gamma_{[u]}:H_n(X,v) \xrightarrow{\sigma_t^X} H_{n+1}(\GGx X,q^{-1}v-\tw 1) \xrightarrow{u_X^!} H_{n+1}(X,v-\tw 1).
	$$
	This defines an action of $\GG(X)$ on $H_*(X,v+\tw *)$.
\end{df}
Note that $\gamma_{[u]}$ is uniquely determined by the equality: $\gamma_{[u]} \circ \partial_t=(1_X)^!-(u_X)^!$.
In particular, $\gamma_{[1]}=0$.

\begin{rem}\label{rem:compare_K1MW-actions}
	Beware that the notation $\gamma_{[u]}$ is a slightly abusive in the general context,
	as this morphism depends on $u$ and not only of the class of $u$ in $K_1^{MW}(\GG(X))$.
	
	However, this is not abusive in the case of the (extended) BM-homology $\hat E_*$ on $\seft_S$
	associated with a spectrum $\E$ over $S$ (see \Cref{ex:hlg&coh} and \Cref{prop:kan_extension_BMhlg}).
	
	First, one can apply the exact same construction to the cohomology theory $\E^*$,
	using the natural contravariant functoriality of $\E^*$ in place of the Gysin morphisms,
	and the localization long exact sequence for $s_0$ (which exists as it is a closed immersion
	between smooth $X$-schemes). One has only to be careful of different twists.
	For the stable $\AA^1$-cohomotopy $\cohtp^*$ we obtain, for any $S$-scheme $X$, a morphism:
	$$
	\gamma_{[u]}:\cohtp^n(X,v) \rightarrow \cohtp^{n-1}(X,v+\tw 1),
	$$
	and it is immediate, by definition of the element $[u] \in \cohtp^{1,1}(X)$, that
	for any $x \in \cohtp^n(X,v)$, $\gamma_{[u]}(x)=[u].x$.
	
	Second, by definition of the action of stable $\AA^1$-cohomotopy on $\hat \E_*$
	(\Cref{rem:ext_action_cohtp_on_BMhlg}), one similarly deduces
	for $X/S$ essentially of finite type, and any $\rho \in \hat E_*(X/S,v)$:
	\begin{equation}
		\gamma_{[u]}(\rho)=[u].\rho.
	\end{equation}
\end{rem}

The following proposition stands as a kind of universal computation of Gysin morphisms,
to be compared with the intersection with effective Cartier divisors appearing on different
homological context (singular homology, Chow groups,...)
\begin{prop}\label{prop:Gysin_divisor}
	Let $H_*$ be a BM-homology theory satisfying $\AA^1$-homotopy invariance
	as in \Cref{num:A1invariance&splitting}.
	
	Let $i:Z \rightarrow X$ be the closed immersion of a principal (effective Cartier) divisor,
	with a given parametrization $\pi:X \rightarrow \AA^1$ and $j:U \to X$ the complementary open immersion.
	We still denote by $\pi$ the unit on $U=X-Z$ obtained by restriction. 
	Then one gets the following relation:
	$$
	i^!=\partial_{Z/X} \circ \gamma_{[\pi]} \circ j^!.
	$$
\end{prop}
\begin{proof}
	Applying the functoriality of the localization long exact sequence
	with respect to the transversal square:
	$$
	\xymatrix@=10pt{
		Z\ar^i[r]\ar_i[d] & X\ar^{\pi_X}[d] \\
		X\ar^-{s_0}[r] & \AA^1_X
	}
	$$
	one obtains the following commutative diagram:
	$$
	\xymatrix@R=10pt@C=20pt{
		H_*(\GGx U,*)\ar^{\partial_t}[r]\ar_{\phi^!}[d] & H_*(X,*)\ar^{i^!}[d] \\
		H_*(U,*)\ar^-{\partial_{Z/X}}[r] & H_*(Z,*-\tw{N_{ZX}})
	}
	$$
	where $\phi:U \xrightarrow{\pi_U} \GGx U \xrightarrow{1_{\GG} \times j} \GGx X$.
	The relation then follows as:
	$$
	i^!=i^! \partial_t \sigma^X_t=\partial_{Z/X} \phi^! \sigma^X_t
	= \partial_{Z/X} (\pi_U)^! (1_{\GG} \times j)^! \sigma^X_t
	\stackrel{\eqref{eq:splitting_functorial}}= \partial_{Z/X} (\pi_U)^! \sigma^U_t j^! 
	=\partial_{Z/X} \gamma_{[u]} j^! 
	$$
\end{proof}

\begin{rem}
	This formula is a far reaching generalization of both \cite[Prop. 2.6.5]{Deg5}
	and \cite[Cor. 12.4]{Rost96}.
	One can actually use it to uniquely characterize the refined fundamental classes
	of \cite{DJK} associated with closed immersions.\footnote{With the notation
		of \emph{op. cit.}, using deformation of  the normal cone,
		one reduces to the case of the closed immersion of $N_ZX \rightarrow D_ZX$,
		which is a principal divisor parametrized by $t:D_ZX \rightarrow \AA^1$.}
	Note also that one can deduce the following normalization property:
	$$
	\partial_t \gamma_{[t]} q^!=\Id_{H_*(X,v)}
	$$
	where $t$ is the tautological unit on $\GGx X$ and $q:\GGx X \to X$ the canonical projection.
\end{rem}

\subsection{Homological Rost transform and Gersten complexes}

\begin{num}\label{num:delta_niveau_ssp}
	The following definitions were prepared in \cite[\S3.1]{BD1}
	and \cite[\S3]{ADN} (see also \cite[\S3]{Jin}).
	This is an elaboration of Grothendieck's theory of Cousin complexes
	and of Bloch-Ogus niveau spectral sequence in homology theories.
	
	Let $(S,\delta)$ be a dimensional scheme, and $\E$ be a motivic spectrum over $S$.
	We consider the continuous extension $\hat \E_*(-/S)$ of the abelian BM-homology over $\sft_S$
	associated with $\E_*(-/S)$ (\Cref{df:twisted_(co)hlg} and \Cref{prop:kan_extension_BMhlg}).
	
	Let now $X/S$ be essentially of finite type.
	In \cite[Def. 3.1.5]{BD1} and also \cite[Def. 3.2.3]{ADN},
	one has build the so-called \emph{$\delta$-niveau spectral sequence} which takes the following form:
	$$
	^{\delta}E^1_{p,q}(X,\E)=\bigoplus_{x \in X_{(\delta = p)}} \hat{\E}_{p+q}(x/S) \Rightarrow \E_{p+q}(X/S).
	$$
	If moreover, $v$ is a virtual bundle over $X$, with rank $r$ and determinant $\mathcal L$,
	applying the same construction to the spectrum $\E_X\tw v$
	and taking into account \Cref{prop:twists_hlg_semilocal} and \Cref{df:twists_hlg_semilocal},
	one gets a twisted form of the $\delta$-niveau spectral sequence:
	\begin{equation}\label{eq:twisted_niveau}
		^{\delta}E^1_{p,q}(X,\E\tw v)=\bigoplus_{x \in X_{(\delta = p)}} \hat{\E}_{p+q-2r,-r}(x/S)\Gtw{\mathcal L_x} \Rightarrow \E_{p+q}(X/S,v).
	\end{equation}
	The following definition is a twisted variant of \cite[3.2.4]{ADN}.
\end{num}
\begin{df}
	Consider the above notation and assumption.
	We define the \emph{twisted Gersten complex} of $(X,v)$ with coefficients in $\E$ as
	the line $q=0$ of the spectral sequence \eqref{eq:twisted_niveau}:
	$$
	{}^{\delta}C_*(X,\E\tw v):={}^{\delta}E^1_{*,0}(X,\E\tw v).
	$$
\end{df}

\begin{num}\label{num:fiberhlg}
	Consider the notation of \Cref{num:delta_niveau_ssp}, and fix an integer $p \in \ZZ$.
	Recall from \cite[Def. 3.2.3]{BD1} that one defines\footnote{Our definition avoids the use of \emph{$S$-models} in \cite[Definition 3.2.1]{BD1}.} the $p$-th fiber $\delta$-homology of $\E$
	evaluated at a pair $(x,n)$ where $x:\Spec(E) \rightarrow S$ is a point and $n \in \ZZ$ as:
	\begin{align}\label{eq:fiberhlg}
		\hat H_p^\delta(\E)(E,n)=\hat \E_{p+n}(x/S,\tw{\delta(x)-n})=\hat \E_{p+2\delta(x),\delta(x)-n}(x/S).
	\end{align}
	In \emph{op. cit.}, we did not discuss the properties of these groups. Actually, we will now equip them
	with a canonical MW-premodule structure: \Cref{defMWmodules}. First:
	\begin{itemize}
		\item covariant functoriality of the BM-homology $\hat \E_*$ with respect to finite morphisms induces \ref{itm:D2'},
		\item the action of stable $\AA^1$-cohomotopy (\Cref{num:action_cohtp_on_BMhlg}) together with Morel's theorem gives \ref{itm:D3'},
		\item the residue morphisms from the localization property of the BM-homology $\hat \E_*$ gives \ref{itm:D4'}.
	\end{itemize}
	Let now $v$ be a virtual bundle over $\Spec(E)$,
	with rank $r$ and determinant $\cL$.
	Then \Cref{prop:twists_hlg_semilocal} together with \Cref{rem:compare_GW_actions} gives
	the following canonical identification:
	$$
	\hat H_p^\delta(\E)(E,n,\cL)=\hat \E_{p+n}(E/S,v+\tw{\delta(E)-n-r}),
	$$
	where the left hand-side is defined according to formula \eqref{eq:df_twists_preMW}.
	Thus, contravariant lci functoriality with respect to morphisms of $S$-points
	of BM-homology gives data \ref{itm:D1'}.
\end{num}

\begin{thm}
	\label{thm_Rost_transform}
	Consider the above notation.
	
	Then, for any $p \in \ZZ$, $\hat H_p^\delta(\E)$ is a (homological) MW-cycle module over $S$, functorial in $\E$.
	In particular, we get a canonical functor, called the \emph{Rost transform}:
	\begin{align*}
		\dR: \SH(S) &\rightarrow \CatMW_S\\
		\E &\mapsto \hat \E=\hat H_0^\delta(\E).
	\end{align*}
	Finally, if $v$ is a virtual bundle over $X$, of rank $0$ and determinant $\cL$,
	one has a canonical isomorphism of graded complexes, natural in $\E$:
	\begin{equation}\label{eq:Gersten&RS-complexes}
		{}^{\delta}C_*(X,\E\tw v) \simeq 
		{}^{\delta}C_*(X,\hat  \E,\cL).
	\end{equation}
\end{thm}
\begin{proof}
	The proof is very technical, but it just consists in putting together the facts proved earlier.
	For the first point, it is sufficient to treat the case $p=0$ as $\hat H_0^\delta(\E[-p])=\hat H_p^\delta(\E)$.
	
	We have already seen the existence of the structure maps of a MW-cycle premodule
	(\Cref{defMWmodules}) before the proof.
	According to \Cref{ex:hlg&MW-premodules} applied to the twisted homology theory $\hat \E_*$
	restricted to $\pts_S$, we not only get data \ref{itm:D1'}, \ref{itm:D2'} but also relations
	\ref{itm:R1a'}, \ref{itm:R1b'}, \ref{itm:R1c'} (resp. \ref{itm:R1a}, \ref{itm:R1b}, \ref{itm:R1c}).
	Then \Cref{rem:ext_action_cohtp_on_BMhlg} not only gives \ref{itm:D3'} but also all relations \ref{itm:R2a}.
	As explained above,
	data \ref{itm:D4'} comes from the residue map in the localization sequence for $\hat \E_*$
	associated with the closed immersion $\Spec(\kappa) \rightarrow \Spec(\cO)$
	(of $S$-schemes essentially of finite type). Relation \ref{itm:R3c'} is immediate,
	while \ref{itm:R3b'} (resp. \ref{itm:R3a'}) follows from the functoriality of localization
	exact sequences with respect to covariant (resp. exceptional contravariant) functoriality.
	Finally, relation \ref{itm:R3d'} follows from \Cref{prop:Gysin_divisor}
	(taken into account \Cref{rem:compare_K1MW-actions})
	and relation \ref{itm:R3e'} from formulas \eqref{eq:GW_PF} (taken into account \Cref{rem:compare_GW_actions}).
	
	To prove that $\hat H_0^\delta(\E)$ is a MW-cycle module (the case $p=0$ is again sufficient),
	we need only to prove the identification \eqref{eq:Gersten&RS-complexes}, both for
	groups and for morphisms, as this will show that the right hand-side is indeed a complex whose
	differentials are well-defined (showing in particular \ref{itm:FD'} and \ref{itm:C'} for $\hat \E$,
	from \Cref{defMWmodules}).
	But this now follows readily from the previous construction,
	and the definition of the $\delta$-niveau exact couple which induces the  $\delta$-niveau spectral sequence.
	We refer the reader to \cite[3.2.6]{ADN} for more details.
\end{proof}

\begin{rem}
	Given a virtual bundle $v$ over $S$ of arbitrary rank $r$ and determinant $\cL$,
	the isomorphism \eqref{eq:Gersten&RS-complexes} extends to the following formula
	isomorphism of complexes:
	$$
	{}^{\delta}C_*(X,\E\tw v)_m \simeq {}^{\delta}C_{*-r}(X,\hat  \E,\cL)_{m+r}.
	$$
	which still corresponds to an isomorphism of graded complexes.
\end{rem}

\begin{num}
	Let $\T$ be a motivic $\infty$-category. As $\SH$ is the universal motivic $\infty$-category (see \cite[Cor. 2.39]{Rob15}),
	we get an adjunction of motivic $\infty$-category:
	$$
	\varphi^*:\SH \leftrightarrows \T:\varphi_*
	$$
	such that $\varphi_*$ commutes with $f_*$, $p^!$ and is weakly monoidal
	(see \cite[2.4.53]{CD19}).
	In particular, given an object $\E$ in $\T(S)$, the twisted homology theory
	associated with $\E$ in $\T(S)$ for $S$-schemes $f:X \rightarrow S$ (as in \Cref{num:bivariant}),
	coincides with the one associated with $\varphi_*(\E)$ in $\SH(S)$:
	\begin{align*}
		\E_*(X/S,v)&=\Hom_{\T(S)}(f_!(\Th(v,\T))[n],\E)
		\simeq \Hom_{\T(S)}(f_!(\varphi^*\Th(v))[n],\E) \\
		& \simeq \Hom_{\T(S)}(\varphi^*f_!(\Th(v))[n],\E) \simeq \Hom_{\SH(S)}(f_!(\Th(v))[n],\varphi_*\E).
	\end{align*}
	One can define the fiber $\delta$-homology associated with $\E$ internally inside $\T$,
	and then we get the formula:
	$$
	\hat H_p^\delta(\E) \simeq \hat H_p^\delta(\varphi_*\E).
	$$
\end{num}
\begin{cor}\label{cor:functor:T->KMW}
	Consider the previous assumptions.
	Then we get a canonical functor
	\begin{align*}
		\T(S) &\rightarrow \CatMW_S\\
		\E &\mapsto \hat \E
	\end{align*}
	such that $\hat \E=\widehat{\varphi_*(\E)}$. For any virtual vector bundle $v$ of rank $0$ and
	determinant $\cL$, there exists a canonical isomorphism of complexes, natural in $\E$:
	$$
	{}^{\delta}C_*(X,\E\tw v) \simeq 
	{}^{\delta}C_*(X,\hat  \E,\cL).
	$$ 
\end{cor}

\begin{ex}
	\begin{enumerate}
		\item The above theorem applies to all (twisted) homology theory representable in $\SH$. 
		In particular, according to \cite{Jin2},
		for any scheme $S$ essentially of finite type over a regular scheme $\Sigma$,
		$p:S \rightarrow \Sigma$,
		the theorem can be applied to $p^!\KGL_\Sigma$ and we obtain (using the $\PP^1$-periodicity
		of K-theory) that the functor:
		$$
		(E,n) \mapsto K'_n(E)=K_n(E)
		$$
		given by Quillen algebraic K-theory, is a MW-cycle module over $S$, equal to $H_0^\delta(p^!\KGL_\Sigma)$.
		\item If $S$ is a $\ZZ[1/2]$-scheme, one can apply the previous theorem to the spectrum
		representing higher Grothendieck-Witt (resp. Witt) groups (we will give more details about that in a subsequent paper).
		
	\end{enumerate}
\end{ex}

\subsection{From cycle modules to homotopy derived complexes}

\begin{paragr}
	Let $\hM$ be a homological MW-cycle module over $S$. Then with the preceding convention,
	for a smooth $S$-scheme $X$, we define $\dH  \hM(X)$ as the $\ZZ$-graded complex,
	with $q$-th graded component and cohomological degree $p$ given by:
	$$
	\derR \Gamma(X,\dH  \hM_q)=\hC^{p}(X,\hM_{q},\cO_X).
	$$
\end{paragr}

	\begin{thm}\label{thm:functor_KMW->DA}
		Let $\hM$ be a homology MW-cycle modules.
		Then $\dH  \hM(X)$ is an $\AA^1$--local, Nisnevich fibrant, $\Omega$-fibrant element of $\DA(S)$. Moreover, we get a functor $$
		\dH:\CatMW_S \rightarrow \DA(S).
		$$

	\end{thm}
	
	\begin{proof} The fact that the above definition yields an element $\dH  \hM(X)$ of $\DA(S)$ follows from the results in Section \ref{sec:Cycle_to_homotopy_modules} (in particular, see \ref{HomotopyComplexIsNisnevich}, \ref{HomotopyInvarianceDeformedComplex}, and \ref{rem_functoriality_lci}).
		\par To see that $\dH$ is a well-defined functor, we consider $\alpha: \hM \to \hM'$ a morphism of homological MW-cycle modules. In particular, for any field $E$ and any integer $n$, we have a map
		\begin{center}
			$\alpha_{E,n}: \hM_n(E) \to \hM'_n(E)$
		\end{center}
		which commutes with the data \ref{itm:D1'}, \ref{itm:D2'}, \ref{itm:D3'}, \ref{itm:D4'}. If $X$ is a (smooth) scheme, $p,q$ two integers, and $*$ a line bundle over $X$, then we have a morphism of $R$-modules
		\begin{center}
			$\alpha_{\#}: 
			{}^{\delta}C_p(X, \hM_q, *) \to
			{}^{\delta}C_p(X, \hM'_q, *)
			$ 
		\end{center}
		according to \ref{MWmod_change_of_coefficients}. This map is in fact a morphism of complexes because $\alpha$ commutes with \ref{itm:D4'}. Then, by definitions, this map is functorial in $X$ with respect to the basic maps Subsection \ref{FiveBasic} because $\alpha$ is compatible with the data \ref{itm:D1'}, \ref{itm:D2'}, \ref{itm:D3'}, \ref{itm:D4'}.Moreover, we can pass to the limit as in Definition \ref{df:functorial_complex} and obtain an induced map\begin{center}
			$\widetilde{\alpha}_{\#}: 
			\hC^p(X, \hM_q, *) \to
			\hC^p(X, \hM'_q, *).
			$ 
		\end{center}
		As above, this is again functorial in $X$ with respect to the basic maps \ref{Basic_maps_homotopy_RS_complex}. In particular, this commutes with the Gysin morphisms defined for lci maps (see \ref{rem_functoriality_lci}). Finally, we can conclude that we have a well-defined induced map
		\begin{center}
			$\dH (\alpha) : \dH \hM \to \dH \hM'$.
		\end{center}
		
	\end{proof}

	\begin{prop}
		\label{Prop_cohomology_coincide_for_smooth}

			Let $X$ be a smooth scheme over $S$ and $v \in \uK(X)$ a virtual vector bundle over $X$ of rank $r$. There is an isomorphism
			$$
			\phi_{X,v}
			:
			{}^{\delta}A^{p-r}(X,\hM_{q-r},\det(v)^\vee)
			\to
			H^{p,\{q\}}(\Th(v), \dH \hM)
			$$
			which commutes with proper pushforward maps and lci pullback maps.
		\end{prop}
		
		\begin{proof}
		
			Without loss of generality, we may assume that $v$ is a vector bundle over $X$. Since $\dH\hM$ is $\Omega$-fibrant, it suffices to consider the morphism in $\DA^{\operatorname{eff}}(S)$. Thanks to the fact that $\dH  \hM$ is $\AA^1$--local and Nisnevich local, the group $H^{p,\{q\}}(\Th(v), \dH \hM)$ is naturally isomorphic to 
			\begin{center}
				$H^p( R\Gamma(X, \dH  \hM_q)\{ -v \})$
			\end{center}
			which is naturally isomorphic to the cohomology of the Rost-Schmid complex of $\hM$ according to \ref{homology_rost_schmid_coincide}, the latter being ${}^{\delta}A^{p-r}(X,\hM_{q-r},\det(v)^\vee)$ by definition. The functorialities follow from the definitions.

		\end{proof}

		\begin{rem}
			Keeping the previous notations, we have in particular:
			$$
			(\dH\hM)^p(X,v)
			=
			H^{p, \{ 0 \} }
			(\Th (-v), 
			\dH\hM)
			\simeq 
			{}^{\delta}A^{p-r}(X,\hM_{-r},\det(v)).
			$$
			
		\end{rem}

		\begin{prop}
			\label{prop_HM_SL_oriented}
			Let $\hM$ be a homological Milnor-Witt cycle module. The complex $\dH  \hM$ is $\textbf{SL}$-oriented.
		\end{prop}
		
		\begin{proof}
			The statement follows from Proposition \ref{Prop_cohomology_coincide_for_smooth}.
		\end{proof}

		\begin{prop}
			\label{Prop_homology_coincide}
			
			Let $X/S$ be separated of finite type, and $v \in \uK(X)$, of rank $r$, there is an isomorphism
			$$
			\psi_{X,v}:
			{}^{\delta}A_{p+r}(X,\hM_{q-r},\det(v)^\vee)
			\to
			H_{p,\{-q\}}^{BM}(\Th(v),\dH \hM)
			$$
			which induces a morphism of twisted Borel-Moore homology theories (see Definition \ref{df:twisted_(co)hlg}).

		\end{prop}
		
		\begin{proof}
		
			In order to prove the general case, we can use a Zariski hyper-cover of $X$ and thus it suffices to prove the case where $X$ is affine.
			\par If $X$ is an affine scheme, then it is smoothable in the sense that there exists $\Xi$ a smooth scheme over $S$ and $i:X \to \Xi$ a closed immersion (not necessarily regular) with open complementary map $j$. Denote by $d$ the dimension of $\Xi$. Without loss of generality, we may assume that $v$ is a vector bundle over $\Xi$. Recall that we denote by $\cotgb_{X/S}$ the cotangent complex of $X/S$ and by $\detcotgb_{X/S}$ its determinant.
			\par 
			We obtain the following commutative diagram:

			\begin{center}
				\resizebox{1.1\textwidth}{!}{
					$\xymatrix{
						{}^{\delta}A_{p+r}(X,\hM_{q-r},\det(v)^\vee)
						\ar[r]^{i_*}
						\ar@{-->}[ddd]_{\simeq}^{\exists !\psi_{X,v}}
						&
						{}^{\delta}A_{p+r}(\Xi,\hM_{q-r},\det(v)^\vee)
						\ar[r]^{j^*}
						\ar[d]^{=}
						&
						{}^{\delta}A_{p+r}(\Xi -X,\hM_{q-r},\det(v)^\vee)
						\ar[r]^-{\partial}
						\ar[d]^{=}
						&
						\dots 
						\\
						{}
						&
						{}^{\delta}A^{d-p-r}
						(\Xi, \hM_{q-r+d}, \det(v)^{\vee} \otimes \detcotgb_{\Xi/S})
						\ar[r]^-{j^*}
						\ar[d]^{\simeq \ref{Prop_cohomology_coincide_for_smooth} }
						&
						{}^{\delta}A^{d-p}
						(\Xi-X, \hM_{q-r+d},\det(v)^{\vee} \otimes \detcotgb_{\Xi - X/S})
						\ar[r]^-{\partial}
						\ar[d]^{\simeq \ref{Prop_cohomology_coincide_for_smooth} }
						&
						\dots
						\\
						{}
						&
						H^{-p, \{ q\} }
						(
						\Th(v-\cotgb_{\Xi/S}),
						\dH  \hM)
						\ar[r]^{j^*}
						\ar[d]^{\simeq}
						&
						H^{-p, \{ q\} }
						(
						\Th(v-\cotgb_{\Xi-X/S}),
						\dH  \hM)
						\ar[r]^-{\partial}
						\ar[d]^{\simeq}
						&
						\dots
						\\
						H_{p,\{-q\}}^{BM}(\Th(v),\dH \hM)
						\ar[r]^{i_*}
						&
						H_{p,\{-q\}}^{BM}(\Th(v),\dH \hM)
						\ar[r]^{j^*}
						&
						H_{p,\{-q\}}^{BM}(\Th(v_{|\Xi - X}),\dH \hM)
						\ar[r]^-{\partial}
						&
						\dots 		
					}$
				}
			\end{center}
			and thus there exists a unique dotted arrow $\psi_{X,v}$ which is functorial by construction.
		\end{proof}

		\begin{rem}
			
			As a consequence, with the conventions of \cite{DJK}, there is a natural isomorphism:
			$$
			(H^\delta\hM)_p(X,v)
			=
			H^{BM}_{p, \{ 0 \} }
			(\Th (v),
			\dH \hM) 
			\simeq 
			{}^{\delta}A_{p+r}(X,\hM_{-r},\det(v)^\vee).
			$$

		\end{rem}

\section{The adjunction theorem}

\label{sec:Equivalence}

\subsection{Statement of the main theorem}

Recall the following definition from \cite[2.1.1]{BD1}.
\begin{df}\label{df:perverse_htp_t}
	Let $\T$ be a motivic $\infty$-category, and $(S,\delta)$ a scheme with a dimension function.
	
	One defines the category $\T(S)_{t_\delta\geq0}$ of (homologically) non-$t_\delta$-negative objects of $\T(S)$
	as the closure under extensions, positive shifts and colimits of the subcategory with objects:
	$$
	f_!(\un_X)(n)[n+\delta(X)], f:X \rightarrow S \text{ separated of finite type}, n \in \ZZ.
	$$
	As $\T(S)$ is presentable, there exists a unique $t$-structure on $\T(S)$
	whose non-negative objects are $\T(S)_{t_\Delta\geq0}$. We call it
	the (perverse) $\delta$-homotopy $t$-structure, denoted by $t_\delta$.
\end{df}
Note in particular that an object $\E$ of $\T(S)$ is \emph{non-$t_\delta$-positive} if and only if
for any $S$-scheme $X$ of finite type,
$$
\forall (p,q)\in \ZZ^2, p>\delta(X) \Rightarrow H^{BM}_{p,\Gtw q}(X/S,\E)=0.
$$
Moreover, if this is the case, this statement holds when $X/S$ is essentially separated of finite type.

\begin{rem}
	Assuming $\T$ is defined over a category subcategory $\mathscr S$ of the category $\sft_S$ of separated $S$-schemes of finite type,
	for any scheme $X$ in $\mathscr S$, we let $\delta^X$ denote the unique dimension function on $X$ obtain by extension from
	the dimension function $\delta$ on $S$ and still denote by $t_\delta$ the $t$-structure on $\T(X)$ associated with $\delta^X$.
\end{rem}

\begin{ex}
	These definitions will be used for the motivic $\infty$-categories: $\SH$, $\DA$, $\DM$ and $\MGL-mod$.
\end{ex}

\begin{num}
	In \cite{BD1}, the $\delta$-homotopy $t$-structure on a motivic $\infty$-category $\T$ over
	a category of schemes $\mathscr S$ was studied, under two additional assumptions. 
	The first, called (Resol) and stated in \emph{loc. cit.} 2.4.1 concerns the category $\mathscr S$
	and ask for a suitable form of resolution of singularities. 
	The second one concerns the motivic $\infty$-category $\T$ and vanishing of certain Borel-Moore/bivariant
	group with coefficient in the unit object $\un_S$ of $\T$. One says that $\T$ is \emph{homotopically compatible}.
	
	These two assumptions are satisfied in the following cases.
	Let $R \subset \QQ$ be a subring, and $\T$ is the $R$-localization of one of the motivic $\infty$-categories $\SH$, $\DA$, $\DM$ and $\MGL-mod$.
	We assume one of the following hypothesis:
	\begin{enumerate}
		\item[(H1)] $R=\QQ$, $\mathscr S$ is the category of $S_0$-schemes essentially of finite type for $\dim(S_0) \leq 3$.
		\item[(H2)] $R=\ZZ[1/p]$ $\mathscr S$ is the category of $k$-schemes where $k$ is a field of characteristic exponent $p$.
	\end{enumerate}
	Then $\mathscr S$ satisfies (Resol) and $\T$ is homotopically compatible.
\end{num}

We need the following lemma.
\begin{lm}
	Let $(S,\delta)$ be a dimension scheme and $\hM$ be a homological $\KMW$-module.
	Then the object $\dH(\hM)$ of $\DA(S)$ is non $t_\delta$-positive.
\end{lm}
\begin{proof}
	This follows from \Cref{df:perverse_htp_t}, the isomorphism of \Cref{Prop_homology_coincide}
	and the definition of Milnor-Witt homology with coefficients in $\hM$ (\Cref{DefinitionComplex}).
\end{proof}

In particular, we get a functor $\dH:\CatMW_S \rightarrow \DA(S)_{t_\delta \leq 0}$ from \Cref{thm:functor_KMW->DA}.
Moreover, we obtain from \Cref{cor:functor:T->KMW} a functor $\dR:\DA(S)_{t_\delta \leq 0} \rightarrow \CatMW_S$.
\begin{thm} 
	\label{thm_main_adjunction_theorem}
	Consider the above notation.
	\begin{enumerate}
		\item  There exists an additive functors:
		$$
		\dR:\DA(S,R) \rightarrow \CatMW_S
		$$
		which is a right inverse of $\dH$: $\dR \circ \dH \simeq \Id$.
		Thus the functor $\dH$ is faithful.
		\item If the category of $S$-schemes satisfies (Resol) and $\DA(-,R)$ is homotopically compatible,
		the previous adjunction induces an equivalence of (additive) categories:
		$$
		\dH:\CatMW_S \leftrightarrows \DA(S,R)^\heartsuit:\dR
		$$
		where the right hand-side is the heart for the perverse $\delta$-homotopy $t$-structure.
	\end{enumerate}
\end{thm}
\begin{proof}
	The main problem is to define the unit $\dH \circ \dR \rightarrow \Id$
	and counit $\Id \rightarrow \dR \circ \dH$ of the adjunction. The proof will be completed in \ref{end_of_proof_main_theorem}.
\end{proof}

\label{subsec_filtrations}
\subsection{Filtrations on complexes}

\begin{num}
	The references for these notions are \cite[Sec. I]{DelH2}, \cite{BBD} and \cite[Appendix A]{BeilPerv}, except that we will adopt other conventions for numbering filtrations.
	
	Let $\A$ be an abelian category. Recall that a decreasing (resp. increasing) filtration of an object $A$ of $\A$ is a sequence of monomorphisms:
	\begin{align*}
		... \rightarrow F^{p+1} \rightarrow F^{p} \rightarrow ... \\
		\text{resp. } ... \rightarrow F_{p-1} \rightarrow F_{p} \rightarrow ...
	\end{align*}
	in the full subcategory of $\mathscr A/A$ composed of objects whose structural maps are monomorphisms and is denoted by $F^*$ or simply $F$.
	There exists an equivalence between decreasing and increasing filtrations
	given by the functor $F_p \mapsto F^{-p}$; hence we consider by default the increasing ones.
	Morphisms of filtrations of a given object $A$ are obviously defined (see \cite[Tag 0121]{stacks_project}).
	
	One also says that $(A,F)$ is a filtered object of $\A$. Morphisms of filtered objects
	are also obviously defined (they are not required to be monomorphisms).
	
	The filtration $F^*$ is bounded above (resp. below) if $0 \rightarrow F^p$ (resp. $F^p \rightarrow A$) is an isomorphism for $p\gg 0$ (resp. $p \ll 0$).
	It is \textit{bounded} (or \textit{finite}, see \cite[Tag 0121]{stacks_project}) if it is bounded both below and above.
	Note that one says that $F^*$ is separated (resp. exhausted) if 
	$0$ is the colimit (resp. $A$ is the limit) of the diagram 
	$$
	\hdots \rightarrow F^p \rightarrow F^{p+1} \rightarrow \hdots
	$$
	More suggestively: $0 \rightarrow \colim_{p\to\infty} F^p$
	(resp. $\lim_{p \to -\infty} F^p \rightarrow A$) is an isomorphism.
	
	One defines the shift functor on filtrations of $A$ by the formula: $s(F^*)^p=F^{p-1}$
	(convention of \cite[A.2]{BeilPerv}). Note that $s$ is an equivalence of categories.
	
	Given a filtration $F^*$ of an object $A$ of $\A$, and $p \in \ZZ$, we define the $p$-th graded part as: $\gr^p_F=\gr^p_F(A)=F^p/F^{p+1}$.
	
	In particular: $\gr^p_{sF}=\gr^{p-1}_F$.
\end{num}

\begin{rem}\label{rem:ocatedron_ab}
	If one relaxes the monomorphism condition (this is something that we will do
	for $\infty$-categories), one can pass from decreasing to increasing filtrations of a given object $A$
	by considering diagrams of the form:
	$$
	\xymatrix@=10pt{
		& F_p\ar[rr]\ar[ld] && F_{p+1}\ar[rd]\ar[ld] & \\
		A\ar[rd] &\ar@{}|{*}[l] & F_{p+1}/F_p\ar[ld]\ar@{-->}[lu]\ar@{}|{*}[u]\ar@{}|{*}[d] &\ar@{}|{*}[r]& A\ar[ld] \\
		& F^p\ar[rr]\ar@{-->}[uu] && F^{p+1}\ar@{-->}[lu]\ar@{-->}[uu] & \\
	}
	$$
	where the diagram (*) indicates exact sequence and the dotted arrow must be thought as a boundary map
	(so does not exist in the case $A$ is in an abelian category), and the other part of the diagram are
	to be thought as commutative. This is nothing else than an octahedron diagram and will be used
	extensively below.
	
	Note that more precisely in the above diagram: $F_*$ is an increasing filtration by monomorphisms (subobjects),
	and $F^*$ is a decreasing "filtration" by epimorphisms (quotients) - one could say "cofiltration"!
	
	The important thing here is the graded object associated with the filtration $F_*$ coincides
	with the graded object associated with the cofiltration $F^*$.
\end{rem}

\begin{rem}
	The convention for exponent/underscore for decreasing/increasing filtrations does not coincide
	with the usual conventions of grading for cohomological/homological complexes.
	That is somewhat unfortunate. However, when we will consider $\infty$-categories,
	I propose to use notation $K(p)$, $K(+\infty)$ rather than exponent/underscore!
\end{rem}

\begin{rem}
	The notion of filtrations/filtered objects as above makes sense in any category with monomorphisms
	(as the category of schemes)! Of course to be able to define the associated graded object,
	one has to ask for cokernels.
\end{rem}

\begin{rem}
	A $\ZZ$-graded object of a category $\C$ is an object of $\C^\ZZ$,
	where $\ZZ$ is seen as a \textbf{discrete} category.
	If $\C$ admits coproducts, there is a realization of $\ZZ$-graded object in $\C$:
	$$
	\C^\ZZ \rightarrow \C, (A^n)_{n \in \ZZ} \mapsto \sum_{n \in \ZZ} A_n,
	$$
	but we prefer to work with the sequences rather than with their realizations.
\end{rem}

\begin{num}\label{num:canonical_filtrations}
	We will consider filtrations on the abelian category $\Comp(\A)$ of complexes of $\A$.
	
	Let $K=K^*$ be such a complex, with cohomological conventions.
	One define the \emph{canonical} filtration of $K^*$ as:
	$$
	F_p^{can}=\tau^{\leq p}(K)=\begin{cases}
		K^n & n<p \\
		\Ker(d^p_K) & n=p \\
		0 & n>p.
	\end{cases}
	$$
	Indeed, this defines an \emph{increasing} filtration of $K$
	$$
	\hdots \rightarrow \tau^{\leq p}(K) \rightarrow  \tau^{\leq p+1}(K) \rightarrow \hdots
	$$
	whose graded objects are $\gr^p_{can}(K)=H^p(K)[-p]$.
	To comply with the tradition, we will consider the opposite decreasing filtration
	$F^p_{can}(K)=\tau^{\leq -p}(K)$.
	
	One defines the \emph{naive} (or "bête") filtration as:
	$$
	F^p_{naiv}=\sigma^{\geq p}(K)=\begin{cases}
		0 & n<p \\
		K^n & n\geq p.
	\end{cases}
	$$
	This is a decreasing filtration of $K$:
	$$
	\hdots \rightarrow \sigma^{\geq p+1}(K) \rightarrow  \sigma^{\geq p}(K) \rightarrow \hdots
	$$
	whose graded objects are $\gr^p_{naive}(K)=K^p[-p]$.
\end{num}

\begin{rem}
	The notation for the canonical filtration is unfortunate...
	I am following \cite{DelH2}, except that I use cohomological notations for the truncations
	(which seems more appropriate!)
\end{rem}

\begin{rem} One can apply \Cref{rem:ocatedron_ab} to the short exact
	sequences of complexes:
	\begin{align*}
		0 \rightarrow \tau^{\leq p}(K) \rightarrow K \rightarrow \tau^{>p}(K) \rightarrow 0 \\
		0 \rightarrow \sigma^{\geq p}(K) \rightarrow K \rightarrow \sigma^{<p}(K) \rightarrow 0
	\end{align*}
	Where:
	\begin{align*}
		\tau^{>p}(K)^n&=K^n, n>p; \mathrm{coKer}(d^{p-1}:K^{p-1} \rightarrow K^p)), n=p; 0, n < p. \\
		(\sigma^{<p}(K))^n&=K^n, n<p; 0, n \geq p.
	\end{align*}
	Note there is currently a mistake in \cite[Tag 0118]{stacks_project} for the definition of $\tau^{>p}=\tau^{\geq p-1}$ ! 
	
\end{rem}

\begin{num}
	\label{notation_filtered_category}
	Following \cite[3.1]{BBD} and \cite[V.1]{Illusie},
	one can define the filtered derived category of $\A$.
	
	A filtered complex $(K,F)$ of $\A$ is a complex $K$ equipped with an increasing filtration $F$,
	\emph{i.e.} a filtered object of the abelian category $\Comp(\A)$.
	This is also a complex of the category $F\A$ !
	
	According to \cite{Schneiders}, the category $F\A$ of filtered objects
	of the abelian category $\A$ is only \emph{quasi-abelian} (in particular an additive exact category).
	One can however consider its derived category $\Der(F\A)$ according to \emph{loc. cit.}\footnote{Recall
		this category is equipped with both a left and a right t-structure.
		The heart of the left t-structure is an abelian category $\mathcal L F \A$ which is
		the left abelianization of $\A$ - there is a universal functor $\A\rightarrow \mathcal L F \A$
		which sends strict exact sequences to exact sequences (see \emph{loc. cit.}).
		The heart of the right t-structure is an abelian category which is the right abelianization
		of $\A$  - there is a universal functor $\mathcal R F \A\rightarrow \A$
		which sends exact sequences to strict exact sequences (use opposite categories, or see \emph{loc. cit.}).}
	
	The following statement can be found in \cite{Illusie}: for any integer $p$, the functor $\gr^p:F\A \rightarrow \A$
	extended to complexes preserves quasi-isomorphisms. In particular, it induces a functor
	$$
	\gr^p:\Der(F\A) \rightarrow \Der(\A)
	$$
\end{num}
\begin{df}
	Let $f:(K,F) \rightarrow (L,G)$ be a morphism of filtered complexes.
	One says $f$ is a \emph{graded quasi-isomorphism}
	(or  weak equivalence) if for any integer $p \in \ZZ$, the induced morphism of complexes
	$$
	\gr^p_F(K) \rightarrow \gr^p_G(L)
	$$
	is a quasi-isomorphism.
	
	The (big) filtered derived category $\DF(\A)$ associated to $\A$ is the localization of the category of filtered
	complexes with respect to graded quasi-isomorphisms.
	
	More directly, the category $\DF(\A)$ is the universal localization of $\Der(F\A)$ such that
	the functors $\gr^p$, $p \in \ZZ$ detects isomorphisms.
	We also consider the categories $\DF^\epsilon(\A)$ of bounded above, below, both, for $\epsilon=-, +, b$, respectively.
\end{df}
One can check that $\DF(\A)$ has the structure of an \emph{f-category} (\cite[App. A]{BeilPerv}).
The functor $s$ induces an exact functor $s:\DF(\A) \rightarrow \DF(\A)$
(which is not the shift functor!)

\begin{num}
	We now consider the canonical $t$-structure on $\Der(\A)$.
	Note however that an arbitrary $t$-structure would work as well
	(this case will be treated more generally for $\infty$-categories below).
	
	Let $(K,F)$ be a filtered complex. One associates to $(K,F)$ an exact couple 
	(with cohomological conventions) as:
	$$
	\xymatrix@=10pt{
		H^{p+q}(F^{p+1})\ar[rr] &&  H^{p+q}(F^p)\ar[ld] \\
		&H^{p+q}(\gr^p_F)\ar^{+1}[lu]&
	}
	$$
	and set $D^{p,q}=H^{p+q}(F^p)$, $E^{p,q}=H^{p+q}(\gr^p_F)$.
	Write $EC(K,F)$ this exact couple.
	This in turn defines a spectral sequence and a filtration
	on $H^*(K)$. For example, the spectral sequence converges if $(K,F)$ is bounded.
	Recall in particular that $E^{*,q}$ is a complex, and these
	complexes form the $E_1$-page of the previously mentioned spectral
	sequence.
	We will put $E^*(K,F)=E^{*,0}$.
	
	The exact couple $EC(K,F)$ is functorial in the filtered complex $(K,F)$,
	sends quasi-isomorphisms of filtered complexes to isomorphisms of exact couples,
	but \emph{does not} sends graded quasi-isomorphism to isomorphisms.
	On the other hand, for any $q \in \ZZ$,
	the complex $E^{*,q}$ is not only functorial,
	but also send graded quasi-isomorphisms to \emph{isomorphisms} of complexes.
	
	In particular, we get a well defined functor:
	$$
	E^*:\DF(\A) \rightarrow \Comp(\A).
	$$
	Note that one gets the relation $E^*\big((K,F)[i])\big)=E^{*,i}$.
\end{num}

Here is an important lemma (we will only use point (1)).
\begin{lm}\label{lm:filtrations_key}
	Let $K$ be a complex of $\A$, seen as an object of $\Comp(\A)$.
	We consider the filtrations of \Cref{num:canonical_filtrations}.
	\begin{enumerate}
		\item Let us consider the filtered complex $(K,F^*_{naive})$.
		Then $E^*\big((K,F^*_{naive})[i])\big)=0$ for $i \neq 0$:
		in other words, $(K,F^*_{naive})$ is in the heart of the filtered derived category
		$\DF(\A)$.
		
		Moreover, as recalled previously, one has a (very) canonical isomorphism
		$$
		\epsilon^p:K^p[-p] \xrightarrow \sim \gr^p_{naive}(K).
		$$ 
		These isomorphisms for $p \in \ZZ$ define an isomorphism of complexes:
		$$
		\epsilon:K^* \xrightarrow \sim E^*(K,F^*_{naive}).
		$$
		In other words, the differentials of the complex $E_1^{*,0}$,
		$0$-th line of the associated spectral sequence,
		are identified (through the isomorphism $\epsilon$) with the original differentials of the complex
		$K$.
		\item Let us consider the filtered complex $(K,F^{*}_{can})$.
		Again $E^*\big((K,F^*_{can})[i])\big)=0$ for $i \neq 0$:
		in other words, $(K,F^*_{can})$ is in the heart of the filtered derived category
		$\DF(\A)$.
		
		Moreover, the differential in the $E_1$-term:
		$H^p(\gr^p_{can}) \rightarrow H^{p+1}(\gr^{p+1}_{can})$ is the boundary map
		associated with the exact sequence of complexes:
		$$
		0 \rightarrow H^{-p}(K^{-p})[p] \rightarrow F_{-p}^{can}(K)/F_{-p-2}^{can}(K)
		\rightarrow H^{-p-1}(K^{-p-1})[p-1] \rightarrow 0.
		$$
	\end{enumerate}
\end{lm}
This is an easy calculation, based on the fact the differential of the $E_1$-term
is obtained from the snake lemma (as \emph{the} "boundary map" associated with
an exact sequence of complexes) (see \cite[Tag 012N]{stacks_project}).

\begin{rem}
	At this point, the construction of the functor $E^*$ would make sense
	if we consider arbitrary (or bounded above/below) filtrations,
	provided one lands in the category of unbounded complexes (or the relevant one).
\end{rem}

\begin{df}
	One defines the canonical $t$-structure on the triangulated category
	underlying the $f$-category $\DF(\A)$ as the $t$-structure
	such that $(K,F)$ is in degree $\geq i$ (resp. $\leq i$) if for all
	$p \in \ZZ$,
	$\gr^p_F(K)$ is in degree $\geq p+i$ (resp. $\leq p+i$).
\end{df}
In other words, this is the universal $t$-structure on $\DF(\A)$ such that
for all $p\in \ZZ$, the functor $\gr^p_{-}(-)[-p]:\DF(\A) \rightarrow \Der(\A)$ is $t$-exact.

\begin{rem}
	\begin{enumerate}
		\item The above definition works if one replaces the canonical t-structure of $\Der(\A)$
		by an arbitrary one. The next theorem can also be extended to that context.
		\item The above definition can be translated in term of the 
		"complex" functor $E^*$ as follows:
		$(K,F)$ is in degree $\geq i$ (resp. $\leq i$) if for all
		$p<i$ (resp. $p>i$), $E^*\big((K,F)[p]\big)=0$.
		This strongly suggest the next theorem which generalizes \cite[3.1.8]{BBD}, or \cite[A.3]{BeilPerv}
		to the unbounded case.
	\end{enumerate}
	
\end{rem}

\begin{num}
	Let us introduce some notation.
	The naive filtration functor defines a well-defined functor
	$$
	\Comp(\A) \rightarrow \Comp(F\A).
	$$
	One readily checks that its essential image in $\DF(\A)$ lands into the category $\DF(\A)^{t=0}$
	for the above $t$-structure.
	We write:
	$$
	\pi_{nv}:\Comp(\A) \rightarrow \DF(\A)^{t=0}
	$$
	the corresponding functor.
\end{num}

\begin{thm}\label{thm:BBD}
	The functors $\pi_{nv}:\Comp(\A) \rightarrow \DF(\A)^{t=0}$ and the functor $E^*:\DF(\A)^{t=0} \rightarrow \Comp(\A)$ (induced by $E^*$) form an equivalence of categories.
	
	In particular, $\pi_{nv}$ is both a left and right adjoint to $E^*$ and we
	get natural isomorphisms:
	\begin{align*}
		Id & \xrightarrow \sim E^* \circ \pi_{nv} \\
		Id & \xrightarrow \sim \pi_{nv} \circ E^*.
	\end{align*}
	Finally, the composite functor $\DF(\A) \xrightarrow{E^*} \Comp(\A) \simeq \DF(\A)^{t=0}$
	is the canonical cohomology functor associated with the $t$-structure
	on $\DF(\A)$.
\end{thm}
\begin{proof}
	According to \Cref{lm:filtrations_key}, we get for any complex $K$ a natural isomorphism:
	$$
	K \rightarrow E^*(K,F^*_{naive})=E^* \circ \pi_{nv}(K).
	$$
	This already proves that $\pi_{nv}$ is a right inverse to $E^*$.
	In particular, $E^*$ is essentially surjective (and $\pi_{nv}$ is faithful).
	
	Therefore, to get the first assertion,
	it suffices to prove that $E^*$ is fully faithful (in view of \cite[IV.4, Th. 1]{MacLane}).
	We use Illusie's spectral sequence \cite[3.1.3.5]{BBD}
	extended to unbounded complexes. For two given objects $K$ and $L$ of $\DF(\A)^{t=0}$, it has the form:
	$$
	E_1^{p,q}=\prod_{i-j=p, p\geq 0} \Hom_{\Der(\A)}^{p+q}(\gr^iK,\gr^jL)
	\Rightarrow \Hom_{\DF(\A)}^{p+q}(K,L)
	$$
	where $\Hom^{a}=H^a\derR \Hom$. Note already that, by definition of the spectral sequence,
	it is located in the region $p\geq0$ (as indicated in our notation).
	
	As $K$ and $L$ are in the heart, one gets:
	$$
	\gr^iK=A^i[i], \gr^jL=B^j[j]
	$$
	for any $i,j\in \ZZ$, where $A$ and $B$ are some fixed objects in $\A$.
	By construction, we have $A^i=E^i(K)$ and $B^j=E^j(L)$.
	In particular, for $p=i-j$, 
	$$
	\Hom_{\Der(\A)}^{p+q}(\gr^iK,\gr^jL)
	=\Hom_{\Der(\A)}^{p+q-i+j}(A^i,B^j)=\Hom_{\Der(\A)}^{q}(A^i,B^j).
	$$
	As $A^i$ and $B^j$ are in the heart of a $t$-structure, the spectral sequence
	is further located in the region $q\geq 0$.
	In particular, for any $(p,q) \in \NN \times \NN$, 
	$d_r^{p,q}:E_r^{p,q} \rightarrow E_r^{p+r,q-r+1}$
	vanishes for $r>q+1$, and the spectral sequence degenerates
	at $E_{q+1}^{p,q}$. Moreover, as $E_\infty^{p,q}$ is located
	in the region $p\geq0, q\geq0$, the filtration on its abutment
	is finite and in particular, it converges (without
	boundedness assumptions). (This argument is classical: it should be in \cite{McCleary01} !)
	
	Moreover, it degenerate at $E_2^{0,0}$ which is the kernel of a map:
	$$
	\prod_{i \in \ZZ} \Hom_{\Der(\A)}(A^i,B^i) \rightarrow \prod_{i \in \ZZ} \Hom_{\Der(\A)}(A^i,B^{i+1})
	$$
	where the map sends $(f^i)$ to $f^i \circ d_B^i-d^i_A \circ f^i$ (up to a sign), where we have denoted by $d^i_A$
	(resp. $d^i_B$) the $i$-th derivation of the complex $E^*(K)$
	(resp. $E^*(L)$), according
	to \cite[beginning of the proof of 3.1.8]{BBD}. 
	
	So $E_2^{0,0} \simeq \Hom_{\Comp(\A)}(E^*(K),E^*(L))$.
	As Illusie's spectral sequence is concentrated in the region $p,q \geq 0$,
	we finally get an edge isomorphism:
	$$
	\Hom_{\DF(\A)}(K,L)
	\xrightarrow \sim E_2^{0,0} \simeq \Hom_{\Comp(\A)}(E^*(K),E^*(L))
	$$
	as required (one checks it agrees with the map induced by the functorial
	structure of $E^*$).
	
	The second observation is an obvious consequence of the first. The last result follows from the isomorphism 
	\begin{align*}
		Id & \xrightarrow \sim \pi_{nv} \circ E^*.
	\end{align*}

\end{proof}

For us, the important corollary is as follows.
\begin{cor}
	Let $K$ be a filtered complex in the heart of the category $\DF(\A)$.
	Then there exists a canonical isomorphism:
	$$
	K \rightarrow \pi_{nv} \circ E^*(K)
	$$
	natural in $K$, and such that after applying $E^*$, the isomorphism
	$$
	E^*(K) \rightarrow E^*(\pi_{nv} \circ E^*(K))=E^* \circ \pi_{nv}(E^*(K)))
	$$
	is the canonical one defined in \Cref{lm:filtrations_key}(1).
\end{cor}

\begin{rem}
	As suggested by \cite[3.1.8]{BBD},
	the heart of $\DF(\A)$ can be identified with the image in $\DF(\A)$
	of the complexes $K$ equipped with their naive filtration.

\end{rem}

\subsection{Cousin and Cohen-Macaulay complexes}

\begin{num}
	Let $(X,\delta)$ be a dimensional scheme.\footnote{Note that $X$ needs not be noetherian here, nor qcqs.}
	We let co$\delta=-\delta$ be associated codimension function.
	
	Codimensional (resp. dimensional) $\delta$-flags of $X$, also known as co$\delta$-flags of $X$ (resp. $\delta$-flags of $X$), are filtrations:
	\begin{align*}
		... \rightarrow Z^{p+1} \rightarrow Z^{p} \rightarrow ... \\
		\text{resp. } ... \rightarrow Z_{p-1} \rightarrow Z_{p} \rightarrow ...
	\end{align*}
	in the category of closed immersions into $X$ such that for any integer $p$, we have co$\delta(Z^p)\geq p$
	(resp. $\delta(Z_p) \leq p$).
	This defines a set ordered by inclusion denoted by $\Flag^\delta(X)$ (resp. $\Flag_\delta(X)$). This is a cofiltered (resp filtered) ordered set.
	Note that the usual change of indices $Z^p \mapsto Z_{-p}$
	gives a bijection $\Flag^\delta(X) \rightarrow \Flag_\delta(X)$.
	
	One defines the $\delta$-coniveau (resp. $\delta$-niveau)
	filtration of $X$ as the filtration of $X$ in the category of pro-schemes (resp. ind-schemes) given for any integer $p$ by:
	\begin{align*}
		F^p_\delta=\plim_{Z^* \in \Flag^\delta(X)} Z^p, \\
		\text{resp. } F_p^\delta=\ilim_{Z_* \in \Flag_\delta(X)} Z_p.
	\end{align*}
	The following notion is not usually considered but can be useful.
	We can define the complement of  the $\delta$-coniveau (resp. $\delta$-niveau)
	filtration of $X$ as the filtration of $X$ in the category of pro-schemes (resp. ind-schemes):
	\begin{align*}
		\bar F^p_\delta=\plim_{Z^* \in \Flag^\delta(X)} (X-Z^p), \\
		\text{resp. } \bar F_p^\delta=\ilim_{Z_* \in \Flag_\delta(X)} (X-Z_p).
	\end{align*}
\end{num}

\begin{num}\textit{Variance of categories of flags}\label{num:flags_variance}
	
	The category of co$\delta$-flags is contravariantly (resp. covariantly) functorial with respect
	to flat (resp. proper lci) morphisms. 
	The category of co$\delta$-flags is covariantly (resp. contravariantly) functorial with respect
	to proper (resp. flat and lci).
	
	Indeed, let $f:Y \rightarrow X$ be a morphism essentially of finite type. We let $\delta_Y$ or simply $\delta$, be the
	dimension function associated on $Y$ with $f$. Fix $Z \in \Flag^{\delta}(X)$ and $T \in \Flag_{\delta}(Y)$
	\begin{enumerate}
		\item If $f:Y \rightarrow X$ is flat, then $Z^p \subset X$, co$\delta(Z^p) \geq p$ implies $co\delta(f^{-1}(Z^p)) \geq p$ \emph{i.e.} $\delta(Z^p) \leq -p$ implies $\delta(f^{-1}Z^p) \leq -p$.
		\item If $f:Y \rightarrow X$ is proper, then $T_p \subset Y$, $\delta(T_p) \leq p$ implies $\delta(f(T_p)) \leq p$.
	\end{enumerate}
	The first property is a consequence of the going-up theorem. The second property is just the definition of the dimension function $\delta_Y$.
	
	For the exceptional variance, there will be a shift in the filtration; we need to have a notion
	of relative dimension $d$ of $f$, and the following:
	\begin{enumerate}
		\item If $f:Y \rightarrow X$ flat and lci, then $Z \subset X$, $\delta(Z) \leq p$ implies $\delta(f^{-1}(Z_p)) \leq p+d$
		\item If $f:Y \rightarrow X$ is proper (universally) equidimensional, then $T_p \subset Y$, $\delta(T_p) \leq p$ implies $\delta(f(T_p)) \leq p-d$.
	\end{enumerate}
	
	There is a special "phantom" functoriality: 
	assume $X$ is a universally catenary noetherian scheme (endowed with a dimension function $\delta$) and let $Z$ be a divisor on $X$ (therefore everywhere of codimension $1$). Let $t$ be a codimension $p$ point of $X-Z$, and let $z$ be a point of $Z$ which lies in the closure of $t$. Then $z$ is a specialization of $t$ and there exists a finite non-empty chain of immediate specializations relating $t$ to $z$, so we have $\delta(z)\leqslant\delta(t)-1$, and therefore $z$ has codimension at least $p$ in $Z$.
	As a consequence, if $T$ is a closed subscheme of $X-Z$ of codimension at least $p$, then the intersection $\overline{T}\cap Z$ of the Zariski closure of $T$ in $X$ with $Z$ has codimension at least $p$ in $Z$.
	
	So we get a morphism $\sigma_{X,Z}:\Flag^\delta(X-Z) \rightarrow \Flag^\delta(Z)$.

\end{num}

\begin{num}As a consequence of the previous paragraph, the co$\delta$-niveau spectral sequence at the level of pro-schemes is functorial with respect to flat morphisms.
	The $\delta$-niveau spectral sequence at the level of ind-schemes is functorial with respect to proper morphisms.
	
\end{num}

\begin{ex}
	For a co$\delta$-flag $Z^*$ over $X$,
	one can consider the commutative diagram of sheaves of sets over $X_\zar$:
	
	\begin{align} \label{Commutative_diagram_filtration}
		\xymatrix@=10pt{
			& X-Z^p\ar[rr]\ar[ld] && X-Z^{p+1}\ar[rd]\ar[ld] & \\
			X\ar[rd] && X-Z^{p+1}/X-Z^p\ar[ld]\ar@{-->}[lu] && X\ar[ld] \\
			& X/X-Z^p\ar[rr]\ar@{-->}[uu] && X/X-Z^{p+1}.\ar@{-->}[lu]\ar@{-->}[uu] & \\
		}	
	\end{align}
	The dotted arrow indicate that the triangle of object is defines is an exact sequence (of sheaves of sets).\footnote{This also corresponds to an octahedron diagram in the 
		corresponding $\infty$-category!}
	\par Assume that $F$ is a Zariski sheaf over $X$. 
	If we apply the functor $\uHom(-,F)$, we obtain the following diagram:
	$$
	\xymatrix@=10pt{
		& j^p_*F_{U^p}\ar@{-->}[rd]\ar@{-->}[dd] && j^{p+1}_*F_{U^{p+1}}\ar[ll]\ar@{-->}[dd] & \\
		F\ar[ru] && \uG_{Z^p/Z^{p+1}}(F)\ar@{-->}[rd]\ar[ru] && F\ar[lu] \\
		& \uG_{Z^p}(F)\ar[ru]\ar[lu] && \uG_{Z^{p+1}}(F)\ar[ll]\ar[ru] & \\
	}
	$$
	with the same convention as above for the dotted arrows, and with the notations
	of \cite[Chap. IV]{HartRD} for the sheaf of sections with support and so on.
	
	Note that the family of sheaves $(\uG_{Z^p}(F))_{p}$ form a \underline{decreasing} filtration of the sheaf $F$.
	This is the $\delta$-coniveau filtration of $F$, and is central in \cite[Chap. IV]{HartRD}.
	However, the notion of Cousin complex of $F$ (see Definition \ref{def_Cousin_complex}) already involves the \emph{derived} version
	of this filtration, introduced below.
\end{ex}

\begin{num}
	We now consider the diagram of the previous example, but taken
	in the small Zariski $\infty$-topos of $X$:
	$$
	\xymatrix@=10pt{
		& X-Z^p\ar[rr]\ar[ld] && X-Z^{p+1}\ar[rd]\ar[ld] & \\
		X\ar[rd] & \ar@{}|-{*}[l] & X-Z^{p+1}/X-Z^p\ar[ld]\ar@{-->}[lu]\ar@{}|-{*}[u]\ar@{}|-{*}[d] & \ar@{}|-{*}[r] & X\ar[ld] \\
		& X/X-Z^p\ar[rr]\ar@{-->}[uu] && X/X-Z^{p+1}\ar@{-->}[lu]\ar@{-->}[uu] & \\
	}
	$$
	Now the $*$-triangles indicate homotopy exact sequences, the dotted arrow being the boundary map,
	and the other triangles are commutative.
	
	The next step is to work within one of the following $\infty$-topos $\T$:
	\begin{enumerate}
		\item Presheaves over $X_\zar$, $X_\nis$ or $X_\et$ either of spaces, pointed spaces, or $R$-complexes.
		\item Sheaves over $X_\zar$, $X_\nis$ or $X_\et$  either of spaces, pointed spaces, or $R$-complexes.
	\end{enumerate}
	In fact, we only need a monoidal $\infty$-category $\T$ with internal Hom (e.g.,
	a presentable monoidal $\infty$-category), and with an $\infty$-functor:
	$$
	\iota:X_\zar \rightarrow \T.
	$$
	In particular, if $K$ is an object of $\T$ and we apply the functor $K(-)=\uHom_\T(\iota(-),K)$
	to the diagram \ref{Commutative_diagram_filtration}, we get:
	$$
	\xymatrix@=10pt{
		& K(X-Z^p)\ar@{-->}[rd]\ar@{-->}[dd] && K(X-Z^{p+1})\ar[ll]\ar@{-->}[dd] & \\
		K(X)\ar[ru] & \ar@{}|-{*}[l] & K(X-Z^{p+1}/X-Z^p)\ar@{-->}[rd]\ar[ru]\ar@{}|-{*}[u]\ar@{}|-{*}[d]
		& \ar@{}|-{*}[r] & K(X)\ar[lu] \\
		& K(X/X-Z^p)\ar[ru]\ar[lu] && K(X/X-Z^{p+1})\ar[ll]\ar[ru] & \\
	}
	$$
	with the same conventions as above. Note that in the cases (1) and (2), $K(X)=K$ as $X$ is the unit
	object of the monoidal structure.
	
	In any case, one gets a decreasing filtration $F^p_{X,\delta}(K)=K(X/X-Z^p)$ of the object $K(X)$,
	in the $\infty$-categorical sense.

	Let us now focus on the case where $\T=\Der(\Sh(X_t,R))$, for $t=\zar,\nis,\et$.
	We get a functor:
	$$
	F^*_{X,\delta}:\Der(\Sh(X_t,R)) \rightarrow \DF(\Sh(X_t,R))
	$$
	also denoted by $F^*_\delta$ when $X$ is clear.
	
	Warning: the object $F^*_{X,\delta}(K)$ is in fact a well-defined
	filtered object of $K$, in the $\infty$-categorical sense.
	One should be careful that $\DF(\Sh(X_t,R)) $ is in fact a localization
	of the category of filtered object of $K$ with respect to graded quasi-isomorphisms.

	One can now reformulate the definition of Cohen-Macaulay complexes
	from \cite[Def. p. 247, Prop. IV.3.1]{HartRD}.
\end{num}
\begin{df} 
	\label{def_Cohen_Macaulay_complex}
	($t=\zar,\nis, \et$)
	Let $K$ be an object of $\Der(\Sh(X_t,R))$.
	One says that $K$ is Cohen-Macaulay if the filtered complex
	$F^*_{X,\delta}(K)$ is in the heart of the $\DF(\Sh(X_t,R))$.
	\par We denote by $\Der(\Sh(X_t,R))_{CM}$ the full subcategory of $\Der(\Sh(X_t,R))$ consisting of Cohen-Macaulay complexes.
\end{df}
This implies that the spectral sequence attached to the filtered object $F^*_\delta(K)$
has a first page which is concentrated on the line $q=0$.
In particular, one gets a canonical edge isomorphism:
$$
H^p(E^*(F^*_{\delta}(K)) \xrightarrow \sim H^p(K).
$$
We want to upgrade this isomorphism into a "canonical"/"unique" quasi-isomorphism.

Recall also the following definition.
\begin{df} 
	\label{def_Cousin_complex}
	Let $t=\zar$ (resp. $\nis$). 
	Let $K$ be an object of $\Comp^b(Sh(X_t,R))$.
	One says that $K$ is Cousin (or a Cousin complex)  
	if for any $p$, the canonical maps
	$$
	K^p \leftarrow \uG_{Z^p}(K^p) \rightarrow \uG_{Z^p/Z^{p+1}}(K^p)
	$$
	are isomorphisms (one says $K^p$ is "concentrated on the $Z^p/Z^{p+1}$-skeleton").
	We denote by $\epsilon^p:K^p \rightarrow \uG_{Z^p/Z^{p+1}}(K^p)$ the corresponding isomorphism.
	\par We denote by $\Comp^b(\Sh(X_t,R))_{Cous}$ the full subcategory of $\Comp^b(Sh(X_t,R))$ consisting of Cousin complexes.
\end{df}
This condition also implies that for any $p \in \ZZ$, 
one has a canonical isomorphism:
$$
K^p \simeq \sum_{x \in X^{(p)}} i_{x*}(K^p_x).
$$
In particular, the complex $K$ is flasque (resp. satisfies the BG-property).

Moreover, it implies that the complex
$$
\uG_{Z^p}(K)
$$
(underived functor) is isomorphic to the naive $p$-filtrated part $\sigma^{\geq p}(K)$
of $K$. In other words, the $\delta$-coniveau filtration on a Cousin
complex is the naive filtration.

More generally, one gets:
\begin{thm}[Suominen]
	\label{thm_Suominen}
	Let $t=\zar,\nis$.
	The functor:
	$$
	E^*:\Der(\Sh(X_t,R))_{CM} \rightarrow \Comp^b(\Sh(X_t,R))_{Cous}
	$$
	where the left (resp. right) hand-side is the full subcategory made
	by Cohen-Macaulay (resp. Cousin) complexes, 
	is an equivalence of categories, with quasi-inverse induced by the natural functor
	$$
	\pi:\Comp^b(\Sh(X_t,R)) \rightarrow \Der^b(\Sh(X_t,R)).
	$$
\end{thm}
In fact, one has a commutative diagram:
$$
\xymatrix@=20pt{
	\Der(\Sh(X_t,R))_{CM}\ar[r]\ar[d] & \Comp^b(\Sh(X_t,R))_{Cous}\ar[d] \\
	\DF(\Sh(X_t,R))^{t=0}\ar^-{E^*}[r] & \Comp^b(\Sh(X_t,R))
}
$$
where the horizontal maps are equivalences of categories.

Note finally that $\Comp^b(\Sh(X_t,R))_{Cous}$
is stable under kernels in the abelian category $\Comp^b(\Sh(X_t,R))$.
On the other hand, this does not seem to be true for cokernels or extensions.
\begin{proof}
	
	Indeed, as the $\delta$-coniveau filtration on a Cousin complex
	is the naive filtration , one gets the commutative diagram:
	$$
	\xymatrix@=20pt{
		\Comp^b(\Sh(X_t,R))_{Cous}\ar^{(F^{\prime*}_{X,\delta})}[r]\ar@{^(->}[d] & \DF(\Sh(X_t,R))^{t=0}\ar@{=}[d] \\
		\Comp(\Sh(X_t,R))\ar^-{\pi_{nv}}[r] & \DF(\Sh(X_t,R))^{t=0}\ar^-{E^*}[r] & \Comp(\Sh(X_t,R))
	}
	$$
	This implies that the upper horizontal arrow is fully faithful.
	By definition, the category $\Der(\Sh(X_t,R))$ is precisely the inverse image
	of the category $\DF(\Sh(X_t,R))^{t=0}$ under the functor
	$$
	F_{X,\delta}^*:\Der(\Sh(X_t,R)) \rightarrow \DF(\Sh(X_t,R)).
	$$
	So $\Der(\Sh(X_t,R))_{CM}$ can be identified with the essential image of $F^{\prime*}_{X,\delta}$.
	Therefore, the following map is an equivalence of categories
	$$
	F^{\prime*}_{X,\delta}:\Comp^b(\Sh(X_t,R))_{Cous} \rightarrow \Der(\Sh(X_t,R))_{CM}
	$$
	with quasi-inverse $E^*$, which by definition sends $\Der(\Sh(X_t,R))_{CM}$
	into $\Comp^b(\Sh(X_t,R))_{Cous}$.
\end{proof}

In particular, for any Cohen-Macaulay complex $K$, one gets a canonical quasi-isomorphism:
$$
K \rightarrow \pi \circ E^*(K).
$$
Unwinding the definitions, $E^*(K)$ is the Cousin complex associated with $K$.
As \Cref{thm:BBD} seems to be true for unbounded complexes,
the same proof seems to work when we will apply it to an object coming
from $\DA(S)^{t_\delta=0}$, once equipped with its "glued"/global
$\delta$-coniveau filtration=the coniveau filtration extended functorially
to the $S$-smooth site.

To make the formalism works, one should work with the abelian category
$\A=\Sp_\nis^{\GG}(S,\ZZ)$
of Nisnevich sheaves of $\GG$-spectra on $\Sm_S$, which are $\AA^1$-local
and $\Omega$-spectra.
I think we can both build the glued coniveau filtration and extend
the previous theorem to these particular filtered complexes !

\label{subsec_sheaves_with_specializations}

\subsection{Sheaves with specialization}

\begin{num}
	\label{def_cat_smsm}
	We let $\Smsm_S$ be the category of smooth affine $S$-schemes, with morphisms the smooth maps, equipped with the induced Nisnevich topology.
	We get an inclusion 
	$$
	\iota_{sm}:\Smsm_S \rightarrow \Sm_S
	$$ 
	(in fact, a continuous morphism of Nisnevich sites).
	\par 
	We let $\PSh^{sm}(S,R)$ (resp. $\Sh^{sm}(S,R)$) be the category of presheaves (resp. Nisnevich sheaves) over $\Smsm$.

\end{num}

\begin{df}
	\label{def_cat_smsp}
	We define, and denote by $\Smsp_S$, the free additive category whose objects are the smooth $S$-schemes, and whose set of morphisms is generated by:
	\begin{itemize}
		\item [\namedlabel{itm:sp1}{(Dsp1)}]a symbol $[f]:Y \to X$ for each smooth morphisms of smooth $S$-schemes $f:Y\to X$,
		\item [\namedlabel{itm:sp2}{(Dsp2)}] a symbol $[s_X^t]:Z \to X-Z$ for each smooth $S$-scheme $X$ and morphism $t:X \to \AA^1$ such that $Z=V(t)=t^{-1}( \{ 0 \} )$ is smooth over $S$,
	\end{itemize}
	modulo the relations:
	\begin{itemize}
		\item [\namedlabel{itm:Rsp1}{(Rsp1)}] For any $f,g$ composable smooth morphisms of smooth $S$-schemes, we have 
		\begin{center}
			$[f \circ g]=[f] \circ [g]$,
		\end{center}
		and, for any smooth $S$-scheme $X$, we have $[\Id_X]=\Id_X$.
		\item [\namedlabel{itm:Rsp2}{(Rsp2)}] let $(X,t)$ as in \ref{itm:sp2}, and let $f:Y\to X$ be a smooth morphism of smooth $S$-schemes. Denote by $u=t\circ f$ (hence $V(u)$ is smooth over $S$), and set $g:V(u) \to V(t)$ and $h:Y-V(u) \to X-V(t)$ the induced maps. We have
		\begin{center}
			$[s_X^t] \circ [g] = [h] \circ [s_Y^{u}]$.			
		\end{center}
		\item [\namedlabel{itm:Rsp3}{(Rsp3)}] For $(X,t)$ as in \ref{itm:sp2}, we assume that $i:Z \rightarrow X$ admits a retraction $p$. Let $q:(X-Z) \rightarrow Z$ be the restriction of $p$.		
		We add the relation:
		\begin{center}
			$[q] \circ [s_X^t] = \Id_Z$.
		\end{center}
	\end{itemize}	
\end{df}

We therefore have a canonical inclusion functor 
$$
{\iota^{sp}_{sm}}:\Smsm_S \rightarrow \Smsp_S.
$$

\begin{df}
	\label{def_sheaves_with_specializations}
	A presheaf with specializations will be a presheaf $F$ of $R$-modules on $\Smsp_S$ which commutes with finite coproducts.
	We say that $F$ is a Nisnevich sheaf with specializations if its restriction to $\Smsm_S$ is a sheaf for the Nisnevich
	topology.
	
	We let $\PSh^{sp}(S,R)$ (resp. $\Sh^{sp}(S,R)$) be the category of presheaves (resp. Nisnevich sheaves) with specializations.
	
\end{df}
\begin{num}
	Any presheaf with specializations $F^{sp}$ induces canonically a presheaf on $\Smsm$ denoted by $F=F^{sp} \circ \iota_{sm}^{sp}$. By abuse of notation, we will write $F$ instead of $F^{sp}$ and say that $F$ \textit{has a structure of presheaf with specializations} (or \textit{has specializations}).
\end{num}

\begin{df}
	Let $X$ be a smooth $S$-scheme. We denote by 
	$$
	\Lsp[X]= \Hom_{\Smsp_S}( - , X)
	$$ 
	the presheaf with specializations represented by $X$. Indeed, we have
	$$
	\Hom_{\Smsp_S}( Y \bigsqcup Y' , X) 
	=
	\Hom_{\Smsp_S}( Y , X)
	\oplus
	\Hom_{\Smsp_S}( Y' , X)
	$$
	for any smooth $S$-schemes $X,Y,Y'$.
\end{df}

\begin{num}
	If $F$ is a presheaf with specializations, the Yoneda lemma shows that
	$$
	F(X) \simeq \Hom_{\PSh^{sp}}( \Lsp[X], F)
	$$
	for any smooth $S$-scheme $X$.
\end{num}

By analogy with the theory of sheaves with transfers, we expect the following result. We only state this lemma as an expectation. In any case, one can always replace $\Lsp[X]$ with an appropriate Nisnevich sheafification.
\begin{lm}
	\label{lem_represented_presheaf_with_spec_is_Nisnevich}

	Let $Y$ be a smooth $S$-scheme. Then $\Lsp[Y]$, as a presheaf on $\Smsm$, is a Nisnevich sheaf. 
\end{lm}

\begin{num}
	In the following, we want to prove that the forgetful functor $\Sh^{sp}(S,R) \to \PSh^{sp}(S,R)$ has a left adjoint functor (called \textit{associated sheaf with specializations}) which extend the Nisnevich sheafification functor $a_{Nis}: \PSh^{sm}(S,R) \to \Sh^{sm}(S,R)$. 
\end{num}

\begin{num}
	Let $X$ be a smooth $S$-scheme, $U$ an $X$-scheme, and $n$ a natural number. We denote by $U^n_X$ the $n-$th power of $U$ as an $X$-scheme.
	\par 
	For any natural numbers $n,i$ such that $0 \leq i \leq n $, denote by $\delta^{n,i}$ the degeneracy map
	$$
	U^{n}_X \to U^{n-1}_X
	$$
	defined as the canonical projection to each factor except the $i-$th.
	\par If $U/X$ is a Nisnevich covering, we set
	$$
	d^{n-1}
	=
	\sum_{i=0}^n
	(-1)^i \delta^{n,i}
	$$
	and obtain the associated augmented Cech complex:
	\begin{center}
		$
		\xymatrix{
			\dots 
			\ar[r]
			&
			U^n_X
			\ar[r]^{d^n}
			&
			\dots 
			\ar[r]
			&
			U^2_X
			\ar[r]^{d^1}
			&
			U
			\ar[r]^{d^0}
			&
			X.	
		}$
	\end{center}

\end{num}

By analogy with the theory of sheaves with transfers, we expect the following result.
\begin{prop}
	\label{prop_cech_complex_is_exact}
	
	Let $X$ be a smooth $S$-scheme and $p:U\to X$ a Nisnevich covering. Then the complex of $\Sh^{sp}(S,R)$ 
	\begin{center}
		$
		\xymatrix{
			\dots 
			\ar[r]
			&
			\Lsp[U^n_X]
			\ar[r]^{d^n}
			&
			\dots 
			\ar[r]
			&
			\Lsp[U^2_X]
			\ar[r]^{d^1}
			&
			\Lsp[U]
			\ar[r]^{d^0}
			&
			\Lsp[X].	
		}$
	\end{center}
\end{prop}

\begin{lm}
	Let $F$ be a presheaf with specializations and denote by $F_{Nis}$ the associated Nisnevich sheaf on $\Smsm$ with structural map of presheaf $\theta : F \to F_{Nis}$. Then there exists a unique presheaf with specializations $F_{Nis}^{sp}$ such that
	\begin{enumerate}
		\item There exists a natural transformation $F \to F_{Nis}^{sp}$ as presheaves with specializations that extend $\theta$.
		\item We have $F_{Nis}^{sp} \circ \iota_{sm}^{sp} = F_{Nis}$.
	\end{enumerate} 
\end{lm}

\begin{proof}
	If we denote by $\CechH$ the $0-$th Cech cohomology functor for the Nisnevich topology, then it is well-known that $F_{Nis} \simeq \CechH \CechH (F)$. Without loss of generality, we replace the functor $(-)_{Nis}$ by $\CechH$.
	\par 
	As $F$ is a presheaf with specializations, we have a canonical inclusion
	\begin{center}
		$F(X) \simeq 
		\Hom_{\PSh^{sp}(S,R)}
		(\Lsp[X], F)
		\subset 
		\Hom_{\PSh^{sm}(S,R)}(\Lsp[X], F)
		$
	\end{center}
	for any smooth $S$-scheme $X$.
	\par Reciprocally, a natural transformation of presheaves over $\Smsm$ of the form
	\begin{center}
		$\Phi: F 
		\to 
		\Hom_{\PSh^{sm}(S,R)}( \Lsp[ - ], F)$
	\end{center}
	is equivalent to the data of a structure of presheaf with specializations (in that case, $\Phi$ is given by the above natural transformation according to Yoneda's lemma), and we have
	\begin{center}
		$\forall \alpha \in \Hom_{\Smsp_S}(Y,X),
		a \in F(X),
		F(\alpha).a
		=
		\Phi_X(a)_Y.\alpha.
		$
	\end{center}
	
	\begin{itemize}
		\item Assume that $\CechH F^{sp}$ is defined.
		\par Let $X$ be a smooth $S$-scheme and $a\in \CechH F$
		
		By definition, we have
		\begin{center}
			$\CechH F(X)
			=
			\lim_{U\to X}
			\Ker (F(U) \to F(U \times_X U)).
			$
		\end{center}
		The limit being a filtered limit, there exists a Nisnevich covering $p:U \to X$ and an element $a_U \in F(U)$ such that $F(p).a=a_U$. Let $\alpha \in \Hom_{\Smsp_S}(Y,X)$, according to Lemma \ref{lem_represented_presheaf_with_spec_is_Nisnevich}, the map $\Lsp[U] \to \Lsp[X]$ is an epimorphism, hence there exists a covering $V \to Y$ such that $p \circ \alpha_U = \alpha_{|V}$. Thus, we have the following commutative diagram
		\begin{center}
			$\xymatrix{
				\CechH F(X)
				\ar[r]
				\ar[d]
				&
				\Hom_{\PSh(S,R)}(\Lsp[X], \CechH F)
				\ar[d]
				\\
				\CechH (U)
				\ar[r]
				&
				\Hom_{\PSh(S,R)}(\Lsp[U], \CechH F)
				\\
				F(U)
				\ar[u]^{\theta}
				\ar[r]
				&
				\Hom_{\PSh(S,R)}(\Lsp[U], F)
				\ar[u]
			}$
		\end{center}
		where the bottom square is commutative by assumption. In other words, we have
		\begin{center}
			$(\CechH F)^{sp}(\alpha_{|V}).a
			=
			\theta(F(\alpha_U).a_U)
			\in 
			F(V).
			$
			
		\end{center}

		\item Reciprocally, we prove that the previous equation defines $\CechH F$ independently of any choice of covering $U\to X$.
		\par Let $U\to X$ be a Nisnevich covering. According to Lemma \ref{lem_represented_presheaf_with_spec_is_Nisnevich}, the functor $H:= \Hom_{\PSh^{sm}(S,R)}(- , \CechH F)$ is left exact and thus the following sequence
		\begin{center}
			$0
			\to 
			H(\Lsp[X])
			\to 
			H(\Lsp[U])
			\to 
			H(\Lsp[U\times_X U])$
		\end{center}
		is exact. As before, we have a map
		\begin{center}
			$\xymatrix{
				F(X) 
				\ar[r]
				&
				\Hom_{\PSh^{sm}(S,R)}(\Lsp[X],F)
				\ar[r]^{\theta_X}
				&
				\Hom_{\PSh^{sm}(S,R)}(\Lsp[X],\CechH F)
			}$		
		\end{center}
		and thus we have a commutative diagram
		\begin{center}
			$\xymatrix
			{
				0
				\ar[r]
				&
				H(\Lsp[X])
				\ar[r]
				&
				H(\Lsp[U])
				\ar[r]
				&
				H(\Lsp[U\times_X U])
				\\
				0
				\ar[r]
				&
				\Ker_U
				\ar[r]
				\ar@{-->}[u]
				&
				F[U]
				\ar[r]
				\ar[u]
				&
				F[U\times_X U]
				\ar[u]
			}$
		\end{center}	
		which is natural with respect to the covering $U\to X$. We have constructed a natural transformation
		\begin{center}
			$\phi^{\CechH}_X:
			\CechH F(X)
			\simeq 
			\lim_{U\to X}
			(\Ker_U)
			\to
			H(\Lsp[X])
			=
			\Hom_{\PSh(S,R)}(\Lsp[X], \CechH F).$
		\end{center}
		
		Let $p:U\to X$ be a Nisnevich covering, $a\in \CechH F(X)$ and $\alpha \in \Hom_{\Smsp_S}(Y,X)$. As in the first point, we can consider a Nisnevich covering $V\to Y$ and lifts $a_U \in F(U), \alpha_U \in \Hom_{\Smsp_S}(V,U)$. By construction, the following diagram
		\begin{center}
			$\xymatrix{
				0
				\ar[r]
				&
				\Hom_{\PSh^{sm}(S,R)}(\Lsp[X], \CechH F)
				\ar[r]
				&
				\Hom_{\PSh^{sm}(S,R)}(\Lsp[U], \CechH F)
				\\
				{}
				&
				\CechH F(X)
				\ar[u]
				&
				\Hom_{\PSh^{sm}(S,R)}(\Lsp[U],F)
				\ar[u]
				\\
				0
				\ar[r]
				&
				\Ker_U
				\ar[r]
				\ar[u]
				&
				F(U)
				\ar[u]
				\ar[lu]
			}$
		\end{center}
		is commutative. In particular, we have the expected relation:
		\begin{center}
			${\CechH F}^{sp}(\alpha_{|V}).a
			=
			\theta_Y(F(\alpha_U).a_U$.
		\end{center}
		We can check that $\Phi^{\CechH}$ is compatible with the composition and extends $\Phi$, which conclude the proof.

	\end{itemize}

\end{proof}

By the previous construction, we see that $F \to F_{Nis}^{sp}$ is a natural transformation, thus we have the following Corollary.
\begin{cor}
	\label{cor_compatibility_sheaves_with_specializations}
	The forgetful functor  $\Sh^{sp}(S,R) \rightarrow \PSh^{sp}(S,R)$ admits a left adjoint $a^{sp}_\nis$ such that the following diagram commutes:
	$$
	\xymatrix{
		\PSh^{sp}(S,R)\ar^{a^{sp}_\nis}[r]\ar_{{\iota^{sp}_{sm}}_*}[d] & \Sh^{sp}(S,R)\ar^{{\iota^{sp}_{sm}}_*}[d] \\
		\PSh^{sm}(S,R)\ar^{a_\nis}[r] & \Sh^{sm}(S,R) \\
	}
	$$
	where the vertical maps are the restriction maps induced by ${\iota^{sp}_{sm}}$, and where $a_\nis$ is the Nisnevich sheafification functor.
	
	The restriction map $o_*$ admits a left adjoint ${\iota^{sp}_{sm}}^*:\Sh(\Smsm_S,R) \rightarrow \Sh^{sp}(S,R)$.
	Moreover, the abelian category $\Sh^{sp}(S,R)$ is compatible with the Nisnevich topology
	in the sense of \cite[4.2.20]{CD19}.

\end{cor}

\begin{rem}
	The category $\Sh^{sp}(-,R)$ is an $\Sm$-fibred abelian category over the category of qcqs schemes.
	The exactness of ${\iota^{sp}_{sm}}^*$ implies that $\Sh^{sp}(S,R)$ is compatible with the Nisnevich topology as defined in \cite[5.1.9]{CD19}.
\end{rem}

\begin{num}
	\label{notation_derived_cat_with_specializations}
	Based on the previous remark,
	we can define as in \cite[Def. 5.3.22]{CD19} the (stable) $\AA^1$-derived category from the categories $\Sh(\Smsm_S,R)$ and
	$\Sh^{sp}(S,R)$ instead of $\Sh(\Sm_S,R)$ (see also \ref{def_derived_cat_with_specializations}). We get $\infty$-categories (associated to the underlying model category of Nisnevich descent)
	and natural adjunctions:
	$$
	\xymatrix@=12pt{
		\Der_{\AA^1}(\Smsm_S,R)
		\ar@<2pt>^{{\iota^{sp}_{sm}}^*}[r]
		\ar@<2pt>^-{{\iota_{sm}}^*}[d] & \Der_{\AA^1}(\Smsp_S,R)
		\ar@<2pt>^-{{\iota^{sp}_{sm}}_*}[l] \\
		\Der_{\AA^1}(\Sm_S,R).
		\ar@<2pt>^-{{\iota_{sm}}_*}[u] &
	}
	$$
\end{num}

The following lemma follows directly from the existence of the Nisnevich-descent model structure (\cite{CD19}) on the
category of Tate spectra with respect to the (Grothendieck) abelian category $\Sh^{sp}(?,R)=\Sh(\Smsp_?,R)$.
\begin{lm}
	Using the model category of \cite[\textsection 5.3]{CD3},
	an object $\mathrm K$ of $\Der_{\AA^1}(\Smsp_S,R)$ can be represented by a Nisnevich sheaf:
	$$
	K_*:\Smsp_S \rightarrow \Comp(R)^\ZZ
	$$
	which is Nisnevich-local, $\AA^1$-local, and equipped with an isomorphism $\sigma'_n:K_n \rightarrow (K_{n+1})_{-1}$ for any integer $n$, where $?_{-1}$ is Voevodsky's $(-1)$-construction. As an abuse of language, we say that $K_*$ is a \emph{good model}.
\end{lm}
\begin{proof}
	Take  fibrant resolution for the Nisnevich-local model structure on the category of $\GG$-Tate-spectra.
\end{proof}

\begin{num}
	Let $K$ be a good model of an object of $\Der_{\AA^1}(\Smsp_S,R)$, as in the above lemma.
	As in Definition \ref{df:functorial_complex}, we build a complex $\widetilde{K}_*(X)=\colim_n K_*(\AA^n_X)$ for any smooth scheme $X$. As in Subsection \ref{subsec_homotopy_cycle_complex}, this defines a presheaf
	$$
	\widetilde{K}_*:\Smsp_S \rightarrow \Comp(R)^\ZZ
	$$
	which is Nisnevich-local, $\AA^1$-local, and equipped with an isomorphism $\widetilde{\sigma}'_n:\widetilde{K}_n \rightarrow (\widetilde{K}_{n+1})_{-1}$ for any integer $n$. As in Theorem \ref{thm_h_data}, we see that $\widetilde{K}_*(X)$ is quasi-isomorphic to $K_*(X)$.
	
\end{num}

\begin{paragr}
	Let $i:Z\to X$ be a closed immersion. Let $t$ be a parameter of $\AA^1$ and let 
	\begin{center}
		${q:X\times_S (\AA^1\setminus \{0\}) \to X}$
	\end{center}
	be the canonical projection. Denote by $D=D_ZX$ the deformation space such that ${D=U \sqcup N_ZX}$ where $U=X\times_S(\AA^1\setminus \{0\})$ (see \cite[§10]{Rost96} for more details).
	
	Consider the morphism 
	\begin{center}
		
		$J(X,Z)=J_{Z/X}:\widetilde{K}_*(X)\to \widetilde{K}_*(N_ZX)$
		
	\end{center} defined by the composition:
	\begin{center}
		
		$
		\xymatrix{
			\widetilde{K}_p(X) \ar[r]^-{q^!} \ar@{-->}[d]^{J_{Z/X}} &
			\widetilde{K}_p(U) 
			\ar[dl]^{s^t} \\ 
			\widetilde{K}_{p}(N_ZX)  &
		}
		$
		
	\end{center}
	where $s^t$ is the map induced by the specialization \ref{itm:sp2}. 
	\par 
	Assume moreover that $i:Z\to X$ is regular of codimension $m$, the map $\pi : N_ZX \to Z$ is a vector bundle over $X$ of dimension $m$. By homotopy invariance, we have an isomorphism
	\begin{center}
		
		$\pi^!:\widetilde{K}_p(Z) \to \widetilde{K}_{p}(N_ZX)$.
	\end{center}
	Denote by $r_{Z/X}=(\pi^!)^{-1}$ its inverse.
	
\end{paragr}

\begin{df}
	
	Keeping the previous notations, we define the map
	\begin{center}
		$i^!:\widetilde{K}_p(X)\to 
		\widetilde{K}_{p}(Z)$
		
	\end{center}
	by putting $i^!=r_{Z/X}\circ J_{Z/X} $ and call it the {\em Gysin morphism of $i$}.
\end{df}

The following lemmas are needed to prove functoriality of the previous construction (see Theorem \ref{GysinFunctoriality}).

\begin{lm}
	\label{lem_gysin_map_for_smsp_scheaves_1}

	Let $i:Z\to X$ be a regular closed immersion and $g:V\to X$ be an essentially smooth morphism. Denote by $N(g)$ the projection from $N(V,{V\times_X Z}) =N_Z(X)\times_X V$ to $N_ZX$. Then
	\begin{center}
		
		$J(V,V\times_X Z)\circ g^!=N(g)^!\circ J(X,Z)$.
	\end{center}
\end{lm}
\begin{proof}
	This follows from axioms \ref{itm:Rsp1} and \ref{itm:Rsp2} (see also \cite[Lemma 11.3]{Rost96}).
\end{proof}

\begin{lm}\label{lem_gysin_map_for_smsp_scheaves_2}
	
	Let $Z\to X$ be a closed immersion and let $p:X\to Y$ be essentially smooth. Suppose that the composite 
	\begin{center}
		
		$q:N_ZX\to Z \to X \to Y$
	\end{center}
	is essentially smooth of same relative dimension as $p$. Then 
	\begin{center}
		
		$ J(X,Z)\circ p^!=q^!$.
	\end{center}
\end{lm}
\begin{proof} Same as \cite[Lemma 11.4]{Rost96} except that Rost only needs the composite morphism
	\begin{center}
		$\xymatrix{
			f:D(X,Z) \ar[r] &  X\times_S \AA^1 \ar[r]^{p\times \Id} & Y\times_S \AA^1
		}$ 
	\end{center} to be flat. We need moreover the fact that $f$ is essentially smooth which is true because it is flat and its fibers are essentially smooth. Then we can use the axiom \ref{itm:Rsp3} (instead of \cite[Lemma 4.5]{Rost96}) to conclude.
\end{proof}

\begin{thm} 
	
	\label{Thm_closed_Gysin_Functoriality_smsp_sheaves}
	Let $l:Z\to Y$ and $i:Y\to X$ be regular closed immersions of respective codimension $n$ and $m$. Then $i\circ l$ is a regular closed immersion of codimension $m+n$ and we have
	\begin{center}
		
		$(i\circ l)^!= l^!\circ i^!$ 
	\end{center}
	as morphism $\widetilde{K}_p(X)\to \widetilde{K}_{p}(Z)$.
\end{thm}
\begin{proof}
	The assertion follows from Lemma \ref{lem_gysin_map_for_smsp_scheaves_1}, Lemma \ref{lem_gysin_map_for_smsp_scheaves_2} and the definitions as in \cite[Theorem 13.1]{Rost96}.
\end{proof}

\begin{paragr}

	We define Gysin morphisms for lci projective morphisms
	and prove functoriality theorems (see \cite[§10]{Feld1} or \cite[§5]{Deg08n2} for similar results).

\end{paragr}

\begin{lm}\label{Lem_Gysin_for_smsp_sheaves_Lem5.9}

	Consider a regular closed immersion $i:Z\to X$ and a natural number $n$. Consider the pullback square
	
	\begin{center}
		
		$\xymatrix{
			\PP^n_Z \ar[r]^l \ar[d]^q & \PP^n_X \ar[d]^p \\
			Z \ar[r]^i & X.
		}$
	\end{center}
	Then $l^!\circ p^!=q^!\circ i^!$.
\end{lm}
\begin{proof}
	
	This follows from the definitions and Lemma \ref{lem_gysin_map_for_smsp_scheaves_1}.
\end{proof}

\begin{lm}\label{Lem_Gysin_for_smsp_sheaves_Lem5.10}
	
	Consider a natural number $n$ and an essentially smooth scheme $X$. Let $p:\PP_X^n\to X$ be the canonical projection. Then for any section $s:X\to \PP^n_X$ of $p$, we have $s^!p^!=\Id$ (up to a canonical isomorphism).
\end{lm}
\begin{proof}
	This follows from Axiom \ref{itm:Rsp3} (see also \cite[Proposition 2.6.5]{Deg5}).

\end{proof}

\begin{lm} \label{Lem_Gysin_for_smsp_sheaves_Lem5.11}
	Consider the following commutative diagram:
	\begin{center}
		
		$\xymatrix{
			& \PP^n_X \ar[rd]^p & \\
			Y \ar[ru]^i \ar[rd]_{i'} & & X \\
			& \PP^m_X \ar[ru]_q & 
		}$
	\end{center}
	where $i,i'$ are regular closed immersions and $p,q$ are the canonical projection. Then $i^!\circ p^!={i'}^!\circ q^!$.
\end{lm}
\begin{proof} Let us introduce the following morphisms:
	\begin{center}
		
		$
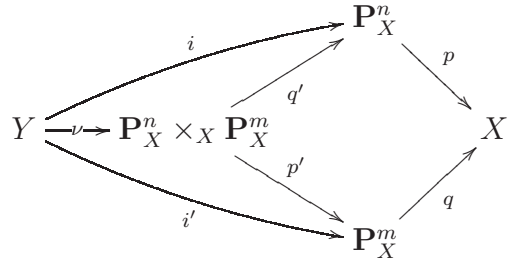
\xymatrix{
			&   &   \PP_X^n \ar[rd]^p&    \\
			Y \ar@/^/[rru]^i \ar[r]|-\nu \ar@/_/[rrd]_{i'} & \PP^n_X\times_X \PP^m_X \ar[ru]_{q'} \ar[rd]^{p'} & & X \\
			&   &   \PP^m_X \ar[ru]_q &   
		}$
	\end{center}
	By functoriality of the smooth pullback \ref{itm:Rsp1}, we are reduced to prove $i^!=\nu^!q'^!$ and ${i'}^!=\nu^!p'^!$. In other words, we are reduced to the case $m=0$ and $q=\Id_X$.
	\par In this case, we introduce the following morphisms:
	\begin{center}
		
		$
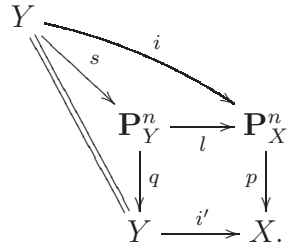
\xymatrix{
			Y  \ar@/^/^i[rrd]  \ar[rd]^s \ar@2{-}[rdd] &      &        \\
			& \PP^n_Y \ar[r]_l \ar[d]^q & \PP^n_X  \ar[d]_p \\
			& Y     \ar[r]^{i'}  &  X.    
		}$
	\end{center}
	Then the result follows from Lemma \ref{Lem_Gysin_for_smsp_sheaves_Lem5.9}, Lemma \ref{Lem_Gysin_for_smsp_sheaves_Lem5.10} and Theorem \ref{Thm_closed_Gysin_Functoriality_smsp_sheaves}.
\end{proof}

\begin{df}
	Let $f:Y\to X$ be a projective lci morphism of smooth $S$-schemes. Consider a factorization 
	$\xymatrix{
		Y
		\ar[r]^i
		&
		\PP^n_X
		\ar[r]^p
		&
		X
	}$
	of $f$ into a regular closed immersion followed by the canonical projection.
	We define the Gysin morphism associated to $f$ as the morphism 
	\begin{center}
		
		$f^!= i^!\circ p^!:\widetilde{K}_*(X)\to \widetilde{K}_{*}(Y)$.
	\end{center}
\end{df}

\begin{prop}
	
	Consider projective morphisms $\xymatrix{Z \ar[r]^g & Y \ar[r]^f & X}$.
	Then:
	\begin{center}
		$g^!\circ f^!=(f\circ g)^!$.
		
	\end{center} 
\end{prop}
\begin{proof} We choose a factorization $\xymatrix{Y \ar[r]^i & \PP^n_X \ar[r]^p & X}$ 
	(resp. $\xymatrix{Z\ar[r]^j & \PP^m_X \ar[r]^q & X}$) of $f$ (resp. $fg$) and we introduce the diagram
	\begin{center}
		
		$\xymatrix{
			&     &    \PP^m_X  \ar@/^2pc/	[rrddd]^q &    &   \\
			&      &   \PP^n_X\times_X \PP^m_X \ar[u]|-{p'} \ar[rd]^{q'} &   &   \\
			&  \PP^m_Y \ar[rd]^{q''} \ar[ru]^{i'}&                 &   \PP^n_X \ar[rd]|-{p}&   \\
			Z \ar[rr]|-{g} \ar[ru]|-{k}  \ar@/^2pc/[rruuu]^j &  &      Y   \ar[ru]^i      \ar[rr]|-{f}    &        &  X
		}$ 
	\end{center}
	in which $p'$ is deduced from $p$ by base change, and so on for $q'$ and $q''$. Then, by using the factorization given in the preceding diagram, the proposition follows from \ref{Lem_Gysin_for_smsp_sheaves_Lem5.9}, \ref{Thm_closed_Gysin_Functoriality_smsp_sheaves}, \ref{Lem_Gysin_for_smsp_sheaves_Lem5.11} and \ref{itm:Rsp1}.
\end{proof}

\begin{rem}
	\label{rem_functoriality_lci_for_smsp_sheaves}
	As a consequence of the previous results, we can define pullback for any morphisms between smooth schemes. Indeed, a map $f:Y \to X$ of smooth schemes can always be factorized as follows:
	\begin{center}
		$\xymatrix{Y \ar[r]^-i & Y \times X \ar[r]^-p & X}$ 
	\end{center}
	where $i(y)=(y,f(y))$ and $p(y,x)=x$. Thus $f$ is lci.
	
\end{rem}

\begin{num}
	
	We have constructed a functor
	$$
	\xymatrix@=12pt{
		{\mathfrak{I}^{sp}}:\Der_{\AA^1}(\Smsp_S,R)
		\ar@<2pt>[r] & \Der_{\AA^1}(\Sm_S,R)
	}
	$$
	such that ${\mathfrak{I}^{sp}}(K_*)=\widetilde{K}_*$ for any good model $K$ of an object of $\Der_{\AA^1}(\Smsp_S,R)$.
	Moreover, one can verify that the identity
	${\iota_{sm}}_*  \circ {\mathfrak{I}^{sp}}={\iota^{sp}_{sm}}_*$
	is true.
\end{num}

\begin{num}
	
	Reciprocally, if one has $K$ in $\Der_{\AA^1}(\Sm_S,R)$, we can define specialization maps:
	$$
	s_X^t:K_n(X) \xrightarrow{\gamma_t} K_{n+1}(X) \xrightarrow{\partial_{X,Z}} K_{n+1}(\Th(N_ZX)) \stackrel{t} \simeq K_{n+1}(\Th(\AA^1_Z)) \simeq K_n(X)
	$$
	where $\partial_{X,Z}$ is the residue map (deduced from the purity property applied to the smooth pair $(X,Z)$), and the isomorphism $t$ is induced by the global parameter $t$
	of $Z$ in $X$. Thus, we can obtain a functor
	$$
	\xymatrix@=12pt{
		{\mathfrak{J}_{sp}}:\Der_{\AA^1}(\Sm_S,R)
		\ar@<2pt>[r] & \Der_{\AA^1}(\Smsp_S,R)
	}
	$$
	
\end{num}
\begin{thm}
	\label{thm_equivalence_sheaves_smsp_and_sm}
	There exists a canonical equivalence of categories:
	$$
	\xymatrix@=12pt{
		{\mathfrak{I}^{sp}}:\Der_{\AA^1}(\Smsp_S,R)\ar@<2pt>[r] & \Der_{\AA^1}(\Sm_S,R):{\mathfrak{J}_{sp}}\ar@<2pt>[l]
	}
	$$
	such that ${\iota_{sm}}_* \circ {\mathfrak{I}^{sp}}={\iota^{sp}_{sm}}_*$.
\end{thm}
\begin{proof}

	It remains to check that ${\mathfrak{I}^{sp}} \circ {\mathfrak{J}_{sp}} \simeq \Id$ and ${\mathfrak{J}_{sp}} \circ {\mathfrak{I}^{sp}} \simeq \Id$, which is true by construction.

\end{proof}

\begin{rem}
	Let $\hM$ be a homological MW-cycle module.
	
	One gets that ${}^\delta C^*(\hM):X \mapsto {}^\delta C^*(X,\hM,\cO_X)$ is an object of $\Der_{\AA^1}(\Smsp_S,R)$.
	By construction, ${\mathfrak{I}^{sp}}({}^\delta C^*(\hM))=\dH(\hM)$.
\end{rem}

\begin{thm} \label{thm_extended_filtration_functorial}
	Let $(S,\delta)$ a dimensional scheme, and $K_*$ a good model of an object of $\Der_{\AA^1}(\Smsp_S,R)$.
	Then the $\delta$-coniveau filtration of $K|_{X_\nis}$ for each smooth $S$-scheme $X$
	can be extended as a filtration of the complex $K_*:(\Smsp_S)^{op} \to \Comp(R)^{\ZZ}$.
	This filtration is functorial in $K_*$ (compatible with the graduation),
	and coincide with the $\delta$-coniveau filtration
	when restricted to $\Smsm_S$. In other words, we have a canonical functor
	$$
	\Xi : \Der_{\AA^1}(\Smsp_S,R) \rightarrow F\Der_{\AA^1}(\Sm_S,R)
	$$
\end{thm}
\begin{proof}
	One uses the fact that $\delta$-flags are functorial with respect to smooth maps,
	and to specializations as explained in \Cref{num:flags_variance}.
\end{proof}

\begin{paragr}
	It follows from Theorem \ref{thm_equivalence_sheaves_smsp_and_sm} and Theorem \ref{thm_extended_filtration_functorial} that we have a canonical functor:
	$$
	\Xi \circ {\mathfrak{J}_{sp}} : \Der_{\AA^1}(\Sm_S,R) \rightarrow F\Der_{\AA^1}(\Sm_S,R).
	$$
	In particular, one can even produce a filtered complex $(K_*,F)$ of $\ZZ$-graded sheaves on $\Sm_S$
	whose restriction to $\Smsm_S$ is the usual $\delta$-coniveau filtration.
	
\end{paragr}

\begin{paragr}
	Since we consider that $\Der(\Sh(\Sm_S,R)^\ZZ)$ is equipped with its canonical $t$-structure, its associated filtered derived category $F\Der(\Sh(\Sm_S,R)^\ZZ)$ as a naturally
	associated $t$-structure, whose heart is equivalent to $\Comp(\Sh(\Sm_S,R)^\ZZ)$.
	The exact same result holds if one replaces $\Sm_S$ with $\Smsp_S$.
	
\end{paragr}

\begin{cor}
	Let $\mathrm K$ be an object of the heart of $\Der_{\AA^1}(\Sm_S,R)$.
	Then the filtered complex $(K_*,F)$ is in the heart of $F\Der(\Sh(\Smsp_S,R)^\ZZ)$.
	In particular, Beilinson's construction gives a functorial isomorphism:
	$$
	K_* \rightarrow E^*(K^*,F)
	$$
\end{cor}

\begin{paragr}
	\label{end_of_proof_main_theorem}
	\begin{proof}[End of proof of Theorem \ref{thm_main_adjunction_theorem}]
		\label{end_of_proof_main_theorem2}
		Keeping the previous notations, we see that $E^*(K^*,F)$ can be identified with ${}^\delta C^*(\dR(\mathrm K))$ (at least when restricted to $\Smsp_S$).
		Thus above map gives the desired adjunction map:
		$$
		\mathrm K \rightarrow \dH(\dR(\mathrm K)).
		$$
	\end{proof}	
\end{paragr}

\section{Cohomological MW-modules}

\label{sec:MWmodcoh}

The notion of \emph{Milnor-Witt cycle modules} is introduced by the second-named author in \cite{Feld1} over a perfect field which, after slight changes, can be generalized to more general base schemes (see \cite{BHP22} for the case of a regular base scheme). In the spirit of the general formalism of \emph{bivariant theories} (see \cite{Deg16}), we call these objects \emph{cohomological Milnor-Witt cycle modules}, to distinguish them from the notion introduced in Section~\ref{sec:homMW} above, considered as the \emph{homological} variants. Our main result on cohomological Milnor-Witt cycle modules will be a duality theorem relating these objects to their homological counterparts, which hold for possibly singular schemes, see Theorems~\ref{eq:cohdualori} and~\ref{thm:eqpin} below.

\subsection{Cohomological Milnor-Witt cycle premodules over a base}

\begin{df} 
	\label{def:cohMW}
	\begin{enumerate}

		\item
		If $S$ is a scheme, call an \textbf{$S$-field} the spectrum of a field essentially of finite type over $S$, and a \textbf{morphism of $S$-fields} an $S$-morphism between the underlying schemes. The collection of $S$-fields together with morphisms of $S$-fields defines a category which we denote by $\mathcal{F}_S$. We say that a morphism of $S$-fields is \textbf{finite} (resp. \textbf{separable}) if the underlying field extension is finite (resp. separable). 
		
		In what follows, we will denote for example $f:\Spec F\to\Spec E$ a morphism of $S$-fields, and $\phi:E\to F$ the underlying field extension.

		An \textbf{$S$-valuation} on an $S$-field $\operatorname{Spec}F$ is a discrete valuation $v$ on $F$ such that $\operatorname{Im}(\mathcal{O}(S)\to F)\subset\mathcal{O}_v$. We denote by $\kappa(v)$ the residue field, $\mathfrak{m}_v$ the valuation ideal and $N_v=\mathfrak{m}/\mathfrak{m}^2$.

		\item
		Let $S$ be a scheme and let $R$ be a commutative ring with unit. An \textbf{$R$-linear cohomological Milnor-Witt cycle premodule} over $S$ is a functor from $\mathcal{F}_S$ to the category of $\ZZ$-graded $R$-modules
		\begin{align}
			\begin{split}
				M:(\mathcal{F}_S)^{op}&\to \operatorname{Mod}_R^{\ZZ}\\
				\operatorname{Spec}E&\mapsto M(E)
			\end{split}
		\end{align}
		for which we denote by $M_n(E)$ the $n$-the graded piece, together with the following functorialities and relations:
		
		\noindent\textbf{Functorialities:}
		\begin{description}
			\item [\namedlabel{itm:D1}{(D1)}] 
			For a morphism of $S$-fields $f:\operatorname{Spec}F\to \operatorname{Spec}E$ or (equivalently) $\phi:E \to F$, a map of degree $0$
			\begin{align}
				\label{eq:D1}
				f^*=\phi_*:M(E)\to M(F);
			\end{align}
			
			\item [\namedlabel{itm:D3}{(D3)}] 
			For an $S$-field $\operatorname{Spec}E$ and an element $x\in \kMW_m(E)$, a map of degree $m$
			\begin{align}
				\gamma_x:M(E)\to M(E)
			\end{align}
			making $M(E)$ a left module over the lax monoidal functor $\kMW_?(E)$ (i.e. we have $\gamma_x\circ\gamma_y=\gamma_{x\cdot y}$ and $\gamma_1=\Id$).
		\end{description}
		
		Similar to the construction in~\eqref{eq:df_twists_preMW}, the axiom~\ref{itm:D3} allows us to define, for every $S$-field $\operatorname{Spec}E$ and every $1$-dimensional $E$-vector space $\cL$, a graded $R$-module
		\begin{align}
			\label{eq:cohmorel}
			M(E,\cL):=M(E)\otimes_{R[E^\times]}R[\cL^\times]
		\end{align}
		where $R[\cL^\times]$ is the free $R$-module generated by the nonzero elements of $\cL$, and the group algebra $R[E^\times]$ acts on $M(E)$ via $u\mapsto\langle u\rangle$ thanks to~\ref{itm:D3}. 
		\begin{description}
			\item [\namedlabel{itm:D2}{(D2)}] 
			For a finite morphism of $S$-fields $f:\operatorname{Spec}F\to \operatorname{Spec}E$ or $\phi:E \to F$, a map of degree $0$ 
			\begin{align}
				\label{eq:D2}
				f_!=\phi^!:M(F,\detcotgb_{F/E}^\vee)\to M(E);
			\end{align}
			
			\item [\namedlabel{itm:D4}{(D4)}] 
			For an $S$-field $\operatorname{Spec}E$ and an $S$-valuation $v$ on $E$, a map of degree $-1$
			\begin{align}
				\label{eq:D4}
				\partial_v:M(E)\to M(\kappa(v), N_v^\vee).
			\end{align}
			
		\end{description}
		
		\noindent\textbf{Relations:}
		\begin{description}
			\item [\namedlabel{itm:R1a}{(R1a)}] 
			The map $f^*$ is compatible with compositions.
			
			\item [\namedlabel{itm:R1b}{(R1b)}] 
			The map $f_!$ is compatible with compositions.
			
			\item [\namedlabel{itm:R1c}{(R1c)}] 
			Let $f:\operatorname{Spec}F\to \operatorname{Spec}E$ be a finite morphism of $S$-fields, let $g:\operatorname{Spec}L\to \operatorname{Spec}E$ be a separable morphism of $S$-fields, and let $R$ be the artinian ring $F\otimes_E L$. For each prime ideal $p\in \Spec R$, let $f_p:\operatorname{Spec}R/p\to \operatorname{Spec}L$ and $g_p:\operatorname{Spec}R/p\to \operatorname{Spec}F$ be the morphisms of $S$-fields induced by $f$ and $g$. Then
			\begin{align}
				g^*\circ f_!
				=
				\sum_{p\in \Spec R} (f_p)_!\circ (g_p)^*.
			\end{align}

			\item [\namedlabel{itm:R2}{(R2)}]

			Let $f:\operatorname{Spec}F\to \operatorname{Spec}E$ be a morphism of $S$-fields. Let $x\in \kMW_m(E)$ and $y\in \kMW_l(F)$. Then 
			\begin{description}
				\item [\namedlabel{itm:R2a}{(R2a)}] 
				We have
				\begin{align}
					f^*\circ\gamma_x
					=
					\gamma_{\phi_*(x)}\circ f^*.
				\end{align}
				\item [\namedlabel{itm:R2b}{(R2b)}] 
				If $f$ is finite, then 
				\begin{align}
					f_!\circ\gamma_{\phi_*(x)}
					=
					\gamma_x\circ f_!.
				\end{align}
				\item [\namedlabel{itm:R2c}{(R2c)}] 
				If $f$ is finite, then 
				\begin{align}
					f_!\circ\gamma_y\circ f^*
					=
					\gamma_{f_!(y)}.
				\end{align}
			\end{description}
			
			\item [\namedlabel{itm:R3a}{(R3a)}] 
			Let $f:\operatorname{Spec}F\to \operatorname{Spec}E$ be a morphism of $S$-fields. Let $w$ be an $S$-valuation on $F$ which restricts to a non-trivial valuation $v$ on $E$ of ramification index $1$, 
			and denote by $\overline{f}:\kappa(v)\to\kappa(w)$ the induced morphism of $S$-fields. Then 
			\begin{align}
				\partial_w\circ f^*=
				\overline{f}^*\circ\partial_v.
			\end{align}

			\item [\namedlabel{itm:R3b}{(R3b)}] 
			Let $f:\operatorname{Spec}F\to \operatorname{Spec}E$ be a finite morphism of $S$-fields, and let $v$ be an $S$-valuation on $E$. 
			For each extension $w$ of $v$, denote by $\phi_w:\kappa(w)\to\kappa(v)$ the induced morphism of $S$-fields. Then 
			\begin{align}
				\partial_v\circ f_!
				=
				\sum_w(f_w)_!\circ\partial_w.
			\end{align}
			
			\item [\namedlabel{itm:R3c}{(R3c)}] 
			Let $f:\operatorname{Spec}F\to \operatorname{Spec}E$ be a morphism of $S$-fields, and let $w$ be an $S$-valuation on $F$ which restricts to the trivial valuation on $E$. Then
			\begin{align}
				\partial_w\circ f^*=0.
			\end{align}
			
			\item [\namedlabel{itm:R3d}{(R3d)}] 
			Let $\phi$ and $w$ be as in~\ref{itm:R3c} and denote by $\overline{f}:\kappa(w)\to\kappa(v)$ be the induced morphism of $S$-fields. Then for any uniformizer $\pi$ of $w$, we have
			\begin{align}
				\partial_w\circ\gamma_{[\pi]}\circ f^*
				=
				\overline{f}^*.
			\end{align}
			
			\item [\namedlabel{itm:R3e}{(R3e)}] 
			Let $\operatorname{Spec}E$ be an $S$-field, $v$ be an $S$-valuation on $E$ and $u$ be a unit of $v$. Then
			\begin{align}
				\partial_v\circ\gamma_{[u]}=\gamma_{-[\bar{u}]}\circ\partial_v;
			\end{align}
			\begin{align}
				\partial_v\circ\gamma_{\eta}=\gamma_{-\eta}\circ\partial_v.
			\end{align}

		\end{description}
		
	\end{enumerate}

\end{df}

\begin{rem}
	\begin{enumerate}
		\item
		Note that our axioms are close to the ones in Section~\ref{sec:homMW} above and are different from the formalism in \cite{Feld1} and \cite{BHP22}, while remaining basically equivalent to the latter. The main point is that under the axiom~\ref{itm:D3}, twists by line bundles can be defined using the twist \emph{\`a la Morel} in~\eqref{eq:cohmorel} and are consequences of the other axioms.
		\item
		Note that we distinguish the notion of \emph{$S$-fields} from the one of \emph{$S$-points}: the latter refers to points on (the underlying topological space of) $S$, which are particular cases of $S$-fields.
		\item
		For simplicity, in what follows we omit the term ``$R$-linear'' in most cases.
	\end{enumerate}
\end{rem}

\begin{paragr}
	\label{num:difres}
	As in the homological case, we would like to promote \emph{MW-cycle premodules} into \emph{MW-cycle modules} (see Definition~\ref{df:homcycmod}). The main obstruction for doing so in the cohomological setting lies in the difference between the axioms for residue maps: indeed, the axiom~\ref{itm:D4} only requires residues for (discrete) valuation rings, while the axiom~\ref{itm:D4'} requires residues for all $1$-dimensional local domains. In some sense, this difficulty is similar to the problem of defining direct images in bivariant theory (see \cite{DJK}). Before dealing with the general case, we start with a simple case: for cohomological MW-cycle modules that are \emph{oriented}, the obstruction above disappears.
	
\end{paragr}

\subsection{Oriented Milnor-Witt cycle modules}

\begin{paragr}
	We start with the notion of \emph{oriented homological MW-cycle modules}:
\end{paragr}
\begin{df}

	Let $S$ be a scheme.
	We say that a homological MW-cycle module over $S$ is \textbf{oriented} if the action of $\eta \in \KMW_{-1}$ is trivial. Equivalently, an oriented homological MW-cycle module has the same axioms as a homological Milnor-Witt cycle module, except that the Milnor-Witt K-theory $\KMW$-action is replaced by a Milnor K-theory $\operatorname{K}^M$-action. We also call such an object a \textbf{homological Milnor cycle module} over $S$.
	
	We denote by $\CatM_S$ the category of homological Milnor cycle modules over $S$, considered as a full subcategory of the category of homological MW-cycle module $\CatMW_S$.

\end{df}

\begin{paragr}
	It turns out that property of being oriented is related to orientations in motivic homotopy (see \cite{DegliseOri}):
\end{paragr}
\begin{prop}
	\begin{enumerate}
		\item
		If $\hM$ is a homological Milnor cycle module, then the object $\dH(\hM)\in\DA(S)$ has a canonical $\operatorname{GL}$-orientation.
		\item
		Assume further that one of the following conditions is satisfied:
		\begin{itemize}
			\item
			the scheme $S$ has equal exponential characteristic $p$, and the coefficient ring $R$ is an $\mathbb{Z}[1/p]$-algebra;

			\item
			the coefficient ring $R$ is a $\mathbb{Q}$-algebra.
		\end{itemize}

		Then the object $\dH(\hM)\in\DA(S)$ lives in $\DM(S,R)$.
	\end{enumerate}
\end{prop}

\begin{paragr}
	We now turn to the cohomological case and begin with MW-cycle premodules.
	Similar to the homological case, we say that a cohomological MW-cycle premodule over $S$ is \textbf{oriented} if the action of $\eta \in \KMW_{-1}$ is trivial.
	
\end{paragr}

\begin{paragr}
	Here is an important property of oriented cohomological MW-cycle premodules, which also holds for the homological counterparts. Let $M$ be a cohomological MW-cycle premodule over $S$. If $\operatorname{Spec}E$ is an $S$-field and $\cL$ is a $1$-dimensional $E$-vector space, then any non-zero element $a\in \cL^\times$ induces an isomorphism $E^\times\simeq \cL^\times$, which induces an isomorphism
	\begin{align}
		\label{eq:oritriv}
		M(E)\simeq M(E,\cL)
	\end{align}
	which a priori depends on $a$. But if $M$ is further assumed oriented, since the action of the Milnor-Witt K-theory $\kMW_*$ factors through Milnor K-theory $\operatorname{K}^M_*$, the action of the group algebra $R[E^\times]$ on $M(E)$ factors through the action of $R$. That is, for any element $u\in E^\times$, the symbol $\langle u\rangle$ acts trivially on $M(E)$. It follows that the isomorphism~\eqref{eq:oritriv} \emph{does not depend} on the choice on the non-zero element $a\in \cL^\times$. This extra functoriality of the isomorphism~\eqref{eq:oritriv} gives a trivialization of all twists by line bundles in a way independent of the non-zero section, and enables us to rewrite two of the functorialities of $M$ as follows:
	\begin{description}
		\item [\namedlabel{itm:d2}{(d2)}] 
		For a finite morphism of $S$-fields $f:\operatorname{Spec}F\to \operatorname{Spec}E$, we have a map a map of degree $0$ 
		\begin{align}
			\label{eq:D2ori}
			f_!:M(F)\to M(E);
		\end{align}
		
		\item [\namedlabel{d4}{(d4)}] 
		For an $S$-field $\operatorname{Spec}E$ and an $S$-valuation $v$ on $E$, we have a map of degree $-1$
		\begin{align}
			\label{eq:D4ori}
			\partial_v:M(E)\to M(\kappa(v)).
		\end{align}
	\end{description}
	One can check that all the relations~\ref{itm:R1a} though~\ref{itm:R3e} continue to hold with the new functorialities~\eqref{eq:D2ori} and~\eqref{eq:D4ori}.

\end{paragr}

\begin{paragr}
	With the twists disappeared, we are now ready to define the residue maps. Let $S$ be an excellent scheme and let $M$ be an oriented cohomological MW-cycle premodule over $S$.
	\begin{enumerate}
		\item
		Let $\mathcal{O}$ be a $1$-dimensional local domain which is essentially of finite type over $S$, with fraction field $F$ and residue field $\kappa$. Then $\mathcal{O}$ is an excellent ring. Let $A$ be the integral closure of $\mathcal{O}$, which is regular, semi-local and finite over $\mathcal{O}$. Let $\{z_1,\cdots,z_n\}$ be the closed points of $A$, which correspond to valuations on $F$ with residue fields $\kappa(z_1),\cdots,\kappa(z_n)$ (indeed, the normalization morphism is finite because $S$ is excellent). Denote by $\phi_i:\operatorname{Spec}\kappa(z_i)\to\operatorname{Spec}\kappa$ the induced finite morphism of $S$-fields. We then define a map
		\begin{align}
			\label{eq:orires}
			\partial_{\mathcal{O}}:
			M_n(F)
			\underset{\eqref{eq:D4ori}}{\xrightarrow{(\partial_{z_i})_{i=1}^n}}
			\oplus_{i=1}^nM_{n-1}(\kappa(z_i)) 
			\underset{\eqref{eq:D2ori}}{\xrightarrow{\sum_{i=1}^n\phi_{i}^*}}
			M_{n-1}(\kappa).
		\end{align}
		
		\item
		Let $X$ be an $S$-scheme essentially of finite type, and let $x$ and $y$ be two points on $X$. We define a map
		\begin{align}
			\label{eq:oriress}
			\partial^x_y:M_*(x) \to M_{*-2}(y)
		\end{align}
		as follows: let $Z$ be the reduced Zariski closure of $x$ in $X$, and if $y\in Z^{(1)}$, then the map~\eqref{eq:oriress} is defined as the map $\partial_{\mathcal{O}_{Z,y}}$ in~\eqref{eq:orires} associated to the $1$-dimensional local domain $\mathcal{O}_{Z,y}$; otherwise we put $\partial^x_y=0$.
	\end{enumerate}

\end{paragr}

\begin{df}
	\label{def:cohori}
	Let $S$ be an excellent scheme. An \textbf{oriented cohomological Milnor-Witt cycle module} over $S$ is an oriented cohomological Milnor-Witt cycle premodule which in addition satisfies the conditions \ref{itm:fd} and \ref{itm:c}
	\begin{description}
		\item [\namedlabel{itm:fd}{(fd)}] {\sc Finite support of divisors.} Let $X$ be an irreducible normal $S$-scheme with generic point $\xi$ and let $\rho$ be an element of $M(\xi)$. Then $\partial^{\xi}_x(\rho)=0$ for all but finitely many $x\in X^{(1)}$.

		\item [\namedlabel{itm:c}{(c)}] {\sc Closedness.} Let $X$ be a local, integral $S$-scheme of dimension 2, with generic point $\xi$ and closed point $x_0$. Then
		\begin{align}
			0
			=
			\sum_{x\in X^{(1)}} \partial^x_{x_0} \circ \partial^{\xi}_x: 
			M_*(\xi)
			\to
			M_{*-1}(x_0).
		\end{align}

	\end{description}
	We also call such an object a \textbf{cohomological Milnor cycle module} over $S$. 
	We denote by $\CoCatM_S$ the category of cohomological Milnor cycle modules over $S$. 
	
\end{df}
Note that the axioms for cohomological Milnor cycle modules are quite close to Rost's original definition of cycle modules (see \cite{Rost96}).

\begin{ex}\label{ex:KM_cycle_module}
	According to \cite{Kato86}, Milnor-Witt K-theory defines
	a cohomological Milnor cycle module over any excellent scheme $S$.
\end{ex}

\begin{thm}
	\label{eq:cohdualori}
	Let $S$ be an excellent scheme. 
	Given a cohomological Milnor cycle module $M$ over $S$, for every $S$-field $\operatorname{Spec}E$, we let
	\begin{align}
		D(M)(E,n)=M_n(E).
	\end{align}
	Reciprocally, given a homological Milnor cycle module $\mathcal{M}$ over $S$, we let 
	\begin{align}
		\mathcal{D}(\mathcal{M})_n(E)=\mathcal{M}(E,n).
	\end{align}
	Then these two constructions establish two functors that are equivalences of categories inverse to each other
	\begin{align}
		D:
		\CoCatM_S 
		\simeq 
		\CatM_S:\mathcal{D}.
	\end{align}
\end{thm}
\begin{proof}
	As all the twists have disappeared, the proof is straightforward with the functorialities~\eqref{eq:D2ori} and~\eqref{eq:orires}. We leave the details to the reader (see also the proof of Theorem \ref{thm:eqpin}).
	
\end{proof}

\subsection{Pinning and duality}

\begin{paragr}
	We now deal with cohomological MW-cycle modules that are not necessarily oriented, in which case we can no longer expect to exploit the isomorphism~\eqref{eq:oritriv} to reasonably define residue maps mentioned in~\ref{num:difres}. We circumvent the difficulty by introducing a new structure on schemes called a \emph{pinning}.
	
\end{paragr}

\begin{paragr}
	Let $S$ be a (possibly singular) scheme and consider $\mathcal{F}_S$ the category of $S$-fields.

	Similar to the pseudo-functor $\virt$ in Definition~\ref{df:virt}, we have a pseudo-functor
	\begin{align}
		\begin{split}
			\lb:(\mathcal{F}_S)^{op}&\to \operatorname{Spaces}\\
			x&\mapsto \lb(x)
		\end{split}
	\end{align}
	where $\lb(x)$ is the space of line bundles on $x$. However, instead of the usual pull-back functoriality of line bundles, we consider the exceptional functor between line bundles: if $f:\operatorname{Spec}F\to \operatorname{Spec}E$ is a morphism of $S$-fields, we define
	\begin{align}
		\begin{split}
			f^!:\lb(E)&\to\lb(F)\\
			v&\mapsto f^*v\otimes\detcotgb_{F/E}^\vee.
		\end{split}
	\end{align}

	With this new functoriality we obtain another pseudo-functor
	\begin{align}
		\begin{split}
			\lb^!:(\mathcal{F}_S)^{op}&\to \operatorname{Spaces}\\
			x&\mapsto \lb(x).
		\end{split}
	\end{align}

\end{paragr}
\begin{paragr}
	In order to deal with Milnor-Witt cycle modules, we need to further take into account discrete valuation rings in addition to fields. Recall Definition \ref{S_DVR}:
	
\end{paragr}
\begin{df}
	\begin{enumerate}
		\item
		We call an \textbf{$S$-DVR} the spectrum of a discrete valuation ring which is essentially of finite type over $S$. 
		\item
		We denote by $\mathcal{G}_S$ the full subcategory of the category of $S$-schemes essentially of finite type, whose
		objects consist of $S$-fields and $S$-DVRs. 
	\end{enumerate}
\end{df}
\begin{rem}
	The category $\mathcal{G}_S$ encodes many types of relations: field morphisms, embeddings of the generic or closed points into a DVR, extensions of valuations, etc.
\end{rem}

\begin{paragr}
	A morphism in $\mathcal{G}_S$ is in particular a morphism between regular schemes, and is therefore a perfect morphism (see \cite[tag 0685]{stacks_project}). For a morphism $f:Y\to X$ in $\mathcal{G}_S$ , we define
	\begin{align}
		\begin{split}
			f^!:\lb(X)&\to\lb(Y)\\
			v&\mapsto f^*v\otimes\detcotgb_f^\vee.
		\end{split}
	\end{align}

	We obtain an extension of the previous pseudo-functor
	\begin{align}
		\begin{split}
			\lb^!:(\mathcal{G}_S)^{op}&\to \operatorname{Spaces}\\
			X&\mapsto \lb(X).
		\end{split}
	\end{align}
	
\end{paragr}

\begin{df}
	\label{df:pinning}
	A \textbf{pinning} (or \textit{\'epinglage}) on a scheme $S$ is a Cartesian section of $\lb^!$. In other words, a pinning on $S$ is a natural transformation of pseudo-functors
	\begin{align}
		\label{eq:morpf}
		\lambda:*\to \lb^!
	\end{align}
	where $*$ denotes the constant functor $(\mathcal{G}_S)^{op}\to\operatorname{Spaces}$ sending every object in $\mathcal{G}_S$ to the one-point space.

\end{df}

\begin{paragr}
	\label{num:pincon}

	Concretely, a pinning $\lambda$ is the following data:
	\begin{enumerate}
		\item 
		for an $S$-field $\operatorname{Spec}E$, a line bundle $\lambda_E$ on $\operatorname{Spec}E$;
		\item
		for a morphism of $S$-fields $\phi:\operatorname{Spec}F\to \operatorname{Spec}E$, an isomorphism $\phi^*\lambda_E\otimes\detcotgb_{F/E}^\vee\simeq\lambda_F$;
		\item
		for an $S$-DVR $\operatorname{Spec}A$, a line bundle $\lambda_A$ on $\operatorname{Spec}A$, together with isomorphisms $ j^*\lambda_A\simeq\lambda_\eta$ and $i^*\lambda_A\otimes(\mathfrak{m}/\mathfrak{m}^2)\simeq\lambda_s$, where $\mathfrak{m}$ is the maximal ideal of $A$ and $i:s\to \operatorname{Spec}A$ and $j:\eta\to \operatorname{Spec}A$ are the inclusions of the closed point and the generic point;
		\item 
		for a morphism $\phi:\operatorname{Spec}A\to \operatorname{Spec}E$ where $\operatorname{Spec}A$ is an $S$-DVR and $\operatorname{Spec}E$ is an $S$-field, an isomorphism $\phi^*\lambda_E\otimes
		\detcotgb_{A/E}^\vee\simeq\lambda_A$;
		
		\item
		for an extension $\phi:\operatorname{Spec}B\to \operatorname{Spec}A$ of $S$-DVRs, an isomorphism $\phi^*\lambda_A\otimes \detcotgb_{B/A}^\vee\simeq \lambda_B$.

	\end{enumerate}
	There are natural compatibilities with the isomorphisms above, which we do not write down.

\end{paragr}

\begin{ex}
	\label{ex:dualpin}
	Let $\mathcal{K}$ be a (coherent) dualizing complex on $S$ (see \cite[Chapter V, Proposition 2.1]{HartRD}). For every object $f:X\to S$ in $\mathcal{G}_S$, the object $f^!\mathcal{K}$ is a dualizing complex on $S$, and since $X$ is a regular scheme, there is a line bundle $K_{(X)}$ on $X$ and an integer $\mu_K(X)$ such that
	\begin{align}
		f^!\mathcal{K}=K_{(X)}[-\mu_K(X)].
	\end{align}
	Then there is a pinning $\lambda$ on $S$ such that $\lambda_X=(K_{(X)}^\vee$. Note that when $X$ is a point of $S$, the map $X\mapsto\mu_K(X)$ defines a codimension function on $S$ (see \cite[Chapter V, §3.4, 7.1]{HartRD}).
	
	In particular, every regular scheme (or even Gorenstein scheme) has a canonical pinning since the structure sheaf is a dualizing complex. For example, if $S$ is regular, then the pinning on $S$ above is such that for every $S$-field $\operatorname{Spec}E$ we have
	\begin{align}
		\lambda_E=\detcotgb_{E/S}^\vee.
	\end{align}
\end{ex}

\begin{rem}
	
	Example~\ref{ex:dualpin} is indeed the prototype of a pinning, which incarnates the data of ``local orientations'' at each point of a scheme. In Theorem~\ref{thm:eqpin} below we will establish a duality of Milnor-Witt cycle modules for schemes with pinnings: such an idea of relating dualizing complexes to duality already appeared in \cite{Fangzhou1} in the special case of real \'etale sheaves.

	See also \cite[Exp. XVII]{TravauxGabber} for a construction of similar spirit in the context of \'etale sheaves.
	
	Also note that since any scheme essentially of finite type over a regular scheme (or even a Gorenstein scheme) has a dualizing complex, the existence of a pinning is a relatively weak condition and does not really impose any restriction on the singularities of the scheme.
	
\end{rem}

\begin{paragr}
	\label{num:pinres}
	If $\lambda$ is a pinning on $S$, then for every scheme $X$ essentially of finite type  over $S$, we define a pinning $\lambda^!_{|X}$ on $X$ by the restriction of $\lambda$ to $\mathcal{G}_X$.
	
\end{paragr}

\begin{paragr}
	Let $\lambda$ be a pinning on $S$ and fix $M$ a cohomological Milnor-Witt cycle premodule. We have the following operations:
	\begin{enumerate}
		\item
		If $\operatorname{Spec}A$ is an $S$-DVR with fraction field $F$, maximal ideal $\mathfrak{m}$ and residue field $\kappa$, we define the following a map using~\ref{itm:D4}
		\begin{align}
			\label{eq:pinres}
			M_n(F,\lambda_{F})
			\simeq
			M_n(F)\otimes_{R[A^\times]}R[\lambda_A^\times]
			\xrightarrow{\eqref{eq:D4}}
			M_{n-1}(\kappa, (\mathfrak{m}/\mathfrak{m}^2)^\vee)\otimes_{R[A^\times]}R[\lambda_A^\times]
			\simeq
			M_{n-1}(\kappa, \lambda_{\kappa}).
		\end{align}
		\item
		For a finite morphism of $S$-fields $f:\operatorname{Spec}F\to \operatorname{Spec}E$, we define the following a map using~\ref{itm:D2}
		\begin{align}
			\label{eq:pinfinpf}
			f_*:M_n(F,\lambda_{F})
			\simeq
			M_n(F,f^*\lambda_{E}\otimes\detcotgb_{F/E}^\vee)
			\xrightarrow{\eqref{eq:D2}}
			M_n(E,\lambda_{E}).
		\end{align}
	\end{enumerate}
	
\end{paragr}

\begin{paragr}
	
	\label{num:respin}
	Let $(S,\lambda)$ be an excellent scheme with a pinning and let $M$ be a cohomological Milnor-Witt cycle premodule over $S$. 
	\begin{enumerate}
		\item
		Let $\mathcal{O}$ be a $1$-dimensional local domain which is essentially of finite type over $S$, with fraction field $F$ and residue field $\kappa$. Then $\mathcal{O}$ is an excellent ring. Let $A$ be the integral closure of $\mathcal{O}$, which is regular, semi-local and finite over $\mathcal{O}$. Let $\{z_1,\cdots,z_n\}$ be the closed points of $A$, which correspond to valuations on $F$ with residue fields $\kappa(z_1),\cdots,\kappa(z_n)$. Denote by $f_i:\operatorname{Spec}\kappa(z_i)\to\operatorname{Spec}\kappa$ the induced finite morphism of $S$-fields. We then define a map
		\begin{align}
			\label{eq:normres}
			\partial_{\mathcal{O}}:
			M_n(F,\lambda_{F})
			\underset{\eqref{eq:pinres}}{\xrightarrow{(\partial_{z_i})_{i=1}^n}}
			\oplus_{i=1}^nM_{n-1}(\kappa(z_i),\lambda_{\kappa(z_i)}) 
			\underset{\eqref{eq:pinfinpf}}{\xrightarrow{\sum_{i=1}^nf_{i*}}}
			M_{n-1}(\kappa,\lambda_{\kappa}).
		\end{align}
		
		\item
		Let $X$ be an $S$-scheme essentially of finite type, and let $x$ and $y$ be two points on $X$. We define a map
		\begin{align}
			\label{eq:normress}
			\partial^x_y:M_*(x,\lambda_x) \to M_{*-1}(y,\lambda_y)
		\end{align}
		as follows: let $Z$ be the reduced Zariski closure of $x$ in $X$, and if $y\in Z^{(1)}$, then the map~\eqref{eq:normress} is defined as the map $\partial_{\mathcal{O}_{Z,y}}$ in~\eqref{eq:normres} associated to the $1$-dimensional local domain $\mathcal{O}_{Z,y}$; otherwise we put $\partial^x_y=0$.
	\end{enumerate}

\end{paragr}

\begin{df}
	\label{df:cohMW}
	Let $(S,\lambda)$ be an excellent scheme with a pinning.
	A \textbf{cohomological Milnor-Witt cycle module} over $S$ is a cohomological Milnor-Witt cycle premodule which in addition satisfies the conditions \ref{itm:FD} and \ref{itm:C}
	
	\begin{description}
		\item [\namedlabel{itm:FD}{(FD)}] {\sc Finite support of divisors.} Let $X$ be an $S$-scheme. Let $\xi$ be a point on $X$ and let $\rho$ be an element of $M(\xi)$. Then $\partial^{\xi}_x(\rho)=0$ for all but finitely many $x$. 

		\item [\namedlabel{itm:C}{(C)}] {\sc Closedness.} Let $X$ be a local, integral $S$-scheme of dimension 2, with generic point $\xi$ and closed point $x_0$. Then
		\begin{align}
			0
			=
			\sum_{x\in X^{(1)}} \partial^x_{x_0} \circ \partial^{\xi}_x: 
			M_*(\xi,\lambda_{\xi})
			\to
			M_{*-2}(x_0,\lambda_{x_0}).
		\end{align}

	\end{description}
\end{df}

\begin{paragr}
	
	The notion of a cohomological MW-cycle module in Definition~\ref{df:cohMW} depends a priori on the choice of a pinning. The following result shows that it actually does not.
	
\end{paragr}
\begin{prop}
	Let $\lambda$ and $\lambda'$ be two pinnings on a scheme $S$, then they define the same cohomological Milnor-Witt cycle modules.
\end{prop}

\begin{proof}
	
	If $\lambda$ is a pinning on $S$ and $\cL$ is a line bundle on $S$, we define another pinning $\cL\cdot\lambda$ on $S$ such that for every object $f:X\to S$ in $\mathcal{G}_S$,
	\begin{align}
		(\cL\cdot\lambda)_X=\lambda_X\otimes f^*\cL.
	\end{align}
	It is straightforward that the two pinnings $\lambda$ and $\cL\cdot\lambda$ define the same cohomological MW-cycle modules.
	
	If $X$ is an integral local scheme and $\lambda,\lambda'$ are two pinnings on $X$, then there exists a line $\cL$ bundle on $X$ such that $\lambda'\simeq \cL\cdot\lambda$. Indeed, every line bundle over a local scheme is trivial, and a trivialization is uniquely determined by a global invertible section. Let $\xi$ be the generic point of $X$, and the isomorphism $\lambda_\xi\simeq\lambda'_\xi$ determines an invertible section $s$ at $\xi$.

	Then the set of points of $X$ at which $s$ is invertible is stable under immediate specializations (use the property that a pinning has sections over all discrete valuation rings and argue as in~\ref{num:respin}) and contains $\xi$, and therefore agrees with the whole $X$ since $X$ is integral. It follows that $s$ defines a global invertible section on $X$, which induces an isomorphism $\lambda'\simeq \cL\cdot\lambda$.

	Now we are ready to prove the proposition: for the axiom~\ref{itm:FD}, we may assume that $X$ is a $1$-dimensional local domain, and the result follows the fact that there is a line bundle $\cL$ on $X$ such that $\lambda'^!_{|X}\simeq \cL\cdot\lambda^!_{|X}$, where $\lambda^!_{|X}$ is defined as in~\ref{num:pinres}; the axiom~\ref{itm:C} follows from a similar argument.

\end{proof}

\begin{paragr}
	\label{CoCatMW}
	It follows that cohomological MW-cycle modules are well defined over any scheme with a pinning, and does not depend on the choice of a pinning. For a scheme $S$ with a pinning, we denote by $\CatMW_S$ (resp. $\CoCatMW_S$) the category of homological (resp. cohomological) MW-cycle modules over $S$. We now prove a general duality theorem for schemes with a pinning:

\end{paragr}

\begin{thm}
	\label{thm:eqpin}
	Let $(S,\lambda)$ be an excellent scheme with a pinning. Given a cohomological MW-cycle module $M$ over $S$, for every $S$-field $\operatorname{Spec}E$, we let
	\begin{align}
		D(M)_n(E)=M_n(E,\lambda_{E}).
	\end{align}
	Reciprocally given a homological MW-cycle module $\mathcal{M}$ over $S$, we let 
	\begin{align}
		\mathcal{D}(\mathcal{M})_n(E)=\mathcal{M}_n(E,\lambda_{E}^\vee).
	\end{align}
	Then these two constructions establish two functors that are equivalences of categories inverse to each other
	\begin{align}
		D:\CoCatMW_S\simeq\CatMW_S:\mathcal{D}.
	\end{align}
\end{thm}

\proof
It is not hard to see that the maps define functors $D$ and $\mathcal{D}$  that are inverse to each other. 
Keeping the previous notations, it remains to check that $\mathcal{D}(\mathcal{M})$ lives in $\CoCatMW_S$ and $D(M)$ lives in $\CatMW_S$.
We now define the functorialities~\ref{itm:D1'} through~\ref{itm:D4'}.
\begin{description}
	\item [\namedlabel{}{(D1')}] 
	For a morphism of $S$-fields $f:\operatorname{Spec}F\to \operatorname{Spec}E$, we define the map 
	\begin{align}
		f^!:D(M)_n(E)
		=
		M_n(E,\lambda_{E})
		\xrightarrow{\eqref{eq:D1}}
		M_n(F,\phi^*\lambda_{E})
		\simeq 
		M_n(F,\lambda_{F}\otimes\detcotgb_{F/E})
		=
		D(M)_n(F,\detcotgb_{F/E})
	\end{align}
	thanks to
	~\ref{itm:D1}.
	
	\item [\namedlabel{}{(D2')}] 
	For a finite morphism of $S$-fields $f:\operatorname{Spec}F\to \operatorname{Spec}E$, we have a map 
	\begin{align}
		\label{eq:pinD2}
		f_*:D(M)_n(F)
		=
		M_n(F,\lambda_{F})
		\xrightarrow{\eqref{eq:pinfinpf}}
		M_n(E,\lambda_{E})
		= 
		D(M)_n(E).
	\end{align}
	
	\item [\namedlabel{}{(D3')}] 
	Completely analogous to~\ref{itm:D3}.

	\item [\namedlabel{}{(D4')}] 
	If $\mathcal{O}$ is a $1$-dimensional local domain which is essentially of finite type over $S$, with fraction field $F$ and residue field $\kappa$, then we have a map
	\begin{align}
		\partial_{\mathcal{O}}:
		D(M)_n(F)
		=
		M_n(F,\lambda_{F})
		\xrightarrow{\eqref{eq:normres}}
		M_{n-1}(\kappa,\lambda_{\kappa})
		=
		D(M))_{n-1}(\kappa).
	\end{align}

\end{description}
For the rest of the axioms it suffices to proceed line by line and check that the corresponding versions match with each other, which makes use of the compatibilities of the isomorphisms stated in~\ref{num:pincon}. The proof for $\mathcal{D}(\mathcal{M})$ is similar, and we leave the details to the reader.
\endproof

\begin{ex}\label{ex:KMW_hlg_MW-module}
	According to \cite{CHN}, we know that $\GW$ is representable in $\SH(S)$,
	for any (qcqs) scheme $S$. Therefore, according to \Cref{thm_Rost_transform}, Hermitian K-theory induces a homological cycle
	module over $S$. After inverting $\eta$, we get that Witt groups induces
	a homological cycle module over $S$, extending the classical
	Gersten-Witt complex for arbitrary $\ZZ$-scheme.
	
	Using \Cref{ex:KM_cycle_module}, \Cref{eq:cohdualori} and the method of \cite{FaselCHW},
	we therefore derive that Milnor-Witt K-theory form a homological Milnor-Witt
	cycle module over any excellent scheme $S$.
\end{ex}

\begin{paragr}
	We have seen in Definition~\ref{df:virt} the fibered category of virtual vector bundles
	\begin{align}
		\begin{split}
			\virt^!:(\mathcal{G}_S)^{op}&\to \operatorname{Spaces}\\
			X&\mapsto \virt(X).
		\end{split}
	\end{align}
	We call a \textbf{dimensional pinning} on a scheme $S$ is a Cartesian section of $\virt^!$.
	
	Since the determinant functor induces an equivalence between the Picard categories 
	$\virt(X)$ and the category of graded line bundles on $X$ (see \cite[§4]{delignedeterminant}), a dimensional pinning is nothing but a pinning together with a codimension function. Example~\ref{ex:dualpin} shows that any dualizing complex on a scheme induces a dimensional pinning.
	
\end{paragr}

\subsubsection{Poincaré duality}
\label{DefinitionCohomologicalComplex}
If $X$ is a scheme with a dimensional pinning, and $\cohM$ is a cohomological MW-cycle module on $X$, then we can form a (cohomological) Rost-Schmid cycle complex ${}^{\delta}C_*(X,M,*)$ such that for any integer $p,q \in \ZZ$, and any line bundle $*$ over $X$:
\begin{align}
	{}^{\delta}C_p(X,M_q,*):=\oplus_{\delta(x)=p}M_{p+q}(x,\lambda_x \otimes *)
\end{align}
as in Definition~\ref{DefinitionComplex}, and we have a canonical isomorphism of abelian groups
\begin{align}
	{}^{\delta}C_p(X,M_q,*) \simeq {}^{\delta}C_p(X,D(M)_q,*)
\end{align}
which is compatible with the differentials.

\subsection{Cohomological basic maps}

\begin{paragr}
	\label{Cohomological basic maps}
	Fix $\cohM$ a cohomological MW-cycle module and fix $(X,\delta,\lambda)$ an $S$-scheme with a dimensional pinning.
	As in \ref{FiveBasic}, we can define basic maps on the cohomological Rost-Schmid complex ${}^{\delta}C^n(X,\cohM, *)$ and prove that the usual functoriality properties (exactly like we did in \ref{Compatibilities}).
	
\end{paragr}

\begin{paragr}{\sc Pushforward}
	Let $f:Y\to X$ be a $S$-morphism of schemes. Define
	\begin{center}
		
		$f_*:{}^{\delta}C_p(Y,\cohM_q,*)\to {}^{\delta}C_{p}(X,\cohM_q, *)$
	\end{center}
	as follows. If $x=f(y)$ and if $\kappa(y)$ is finite over $\kappa(x)$, then $(f_*)^y_x=\cores_{\kappa(y)/\kappa(x)}$. Otherwise, $(f_*)^y_x=0$.
\end{paragr}

\begin{paragr}{\sc Pullback} \label{CohomologicalpullbackBasicMap}
	Let $f:Y\to X$ be an {\em essentially smooth} morphism of schemes of relative dimension $s$. Suppose $Y$ connected. Define
	\begin{center}
		$f^!:{}^{\delta}C_p(X,\cohM_q,*) \to {}^{\delta}C_{p+s}(Y,\cohM_{q-s},*\otimes \detcotgb_f^{\vee})$
	\end{center}
	as follows. If $f(y)=x$, then $(f^!)^x_y= \res_{\kappa(y)/\kappa(x)}$. Otherwise, $(f^!)^x_y=0$. If $Y$ is not connected, take the sum over each connected component.

\end{paragr}

\begin{paragr}{\sc Multiplication with units}
	Let $a_1,\dots, a_n$ be global units in $\mathcal{O}_X^*$. Define
	\begin{center}
		$[a_1,\dots, a_n]:
		{}^{\delta}C_p(X,\cohM_q,*) \to 
		{}^{\delta}C_p(X,\cohM_{q+n},*)$
	\end{center}
	as follows. Let $x$ be in $X_{(\delta= p)}$ and $\rho\in \hM(\kappa(x),*)$. We consider $[a_1(x),\dots, a_n(x)]$ as an element of ${\KMW (\kappa(x),*)}$.
	If $x=y$, then put $[a_1,\dots , a_n]^x_y(\rho)=[a_1(x),\dots , a_n(x)]\cdot \rho) $. Otherwise, put $[a_1,\dots , a_n]^x_y(\rho)=0$.

\end{paragr}

\begin{paragr}{\sc Multiplication with $\eta$}
	Define
	\begin{center}
		
		$\eta:
		{}^{\delta}C_p(X,\cohM_q,*)
		\to {}^{\delta}C_p(X,\cohM_{q-1},*)$
	\end{center}
	as follows. If $x=y$, 
	then $\eta^x_y(\rho)
	=\gamma_{\eta}(\rho)$. 
	Otherwise, $\eta^x_y(\rho)=0$.
	
\end{paragr}

\begin{paragr}{\sc Boundary maps} \label{CohomologicalBoundaryMaps}
	Let $X$ be a scheme of finite type over $k$, let $i:Z\to X$ be a closed immersion and let $j:U=X\setminus Z \to X$ be the inclusion of the open complement. We will refer to $(Z,i,X,j,U)$ as a boundary triple and define
	\begin{center}

		$\partial=\partial^U_Z:
		{}^{\delta}C_{p}(U,\cohM_q,*) 
		\to {}^{\delta}C_{p-1}(Z,\cohM_q,*)$
	\end{center}
	by taking $\partial^x_y$ to be as the definition in \ref{2.0.1} with respect to $X$. The map $\partial^U_Z$ is called the boundary map associated to the boundary triple, or just the boundary map for the closed immersion $i:Z\to X$.
	
\end{paragr}

\section{Appendix}
\label{sec:Computations}

\subsection{The Gersten property}

\begin{paragr}
	In this section, we consider $S$ an excellent regular scheme. Any regular scheme is endowed with the codimension function defined by the structure sheaf as in Example~\ref{ex:dualpin}. For any regular scheme $X$, we denote by $X^{(p)}$ the points of codimension $p$ in $X$.
	
\end{paragr}

\begin{paragr}
	Let $\hM$ be a (homological) Milnor-Witt cycle module over $S$. 
	As in Definition~\ref{DefinitionComplex}, for any $S$-scheme essentially of finite type $X$, we have a complex $C^*(X,\hM,*)$ (with cohomological conventions for degrees)
	\begin{align}
		\label{eq:cohcompl}
		\cdots\to C^{p}(X,\hM,*)\to C^{p+1}(X,\mathcal{M},*-1)\to\cdots
	\end{align}
	such that $C^p(X,\hM,*)=\oplus_{x\in X^{(p)}}\hM(x,*)$. As in Definition~\ref{def:AgpM}, the cohomology of this complex is denoted by $A^p(X,\hM,*)$. In particular, the group $A^0(X,\hM,*)$ is the kernel
	\begin{align}
		A^0(X,\hM,*)
		=
		\operatorname{ker}(
		\oplus_{x\in X^{(0)}}\hM(x,*)
		\xrightarrow{\eqref{eq:reshom}}
		\oplus_{x\in X^{(1)}}\hM(x,*-1))
	\end{align}
	and we have an augmented complex
	\begin{align}
		\label{eq:augcomp}
		0\to
		A^0(X,\hM,*)
		\to
		C^{0}(X,\hM,*)\to C^{1}(X,\hM,*-1)\to\cdots.
	\end{align}
	
\end{paragr}

\begin{df}

	We say that 
	$\hM$ 
	is \textbf{Gersten} if for every $X\in \Sm_S$ local, the group $A^p(X,M,*)$ 
	vanishes for every $p\neq0$. 
\end{df}

If $S$ is an excellent discrete valuation ring, this property is equivalent to \cite[Def. 2.13]{BHP22}.

\begin{ex}
	The main theorem of \cite{BHP22} shows that Milnor-Witt K-theory $\KMW$ is Gersten.
\end{ex}

\begin{paragr}
	The presheaf $X\mapsto A^0(X,\hM,*)$ for $X\in\Sm_S$ is a Nisnevich sheaf denoted by $\underline{A}^0(S,\hM,*)$, which is an \emph{unramified sheaf} (see \cite[\S2]{MorelLNM}). For $X\in\Sm_S$, the restriction of the sheaf $\underline{A}^0(S,\hM,*)$ to $X_{Nis}$ is denoted $\underline{A}^0(X,\hM,*)$.
	
	The map $U\in X_{Zar}\mapsto C^*(U,\hM,*)$ defines a Zariski sheaf of complexes on $X$, which we denote by $\underline{C}^*(X,\hM,*)$. Concretely, $\underline{C}^p(X,\hM,*)=\oplus_{x\in X^{(p)}}i_{x*}\hM(x,*)$, and therefore $\underline{C}^*(X,\hM,*)$ is a complex of flasque sheaves, together with a canonical map
	\begin{align}
		\label{eq:augcompl}
		\underline{A}^0(X,\hM,*)\to \underline{C}^*(X,\hM,*).
	\end{align}
\end{paragr}

\begin{lm}
	\label{lm:gerstenex}
	The Milnor-Witt cycle module $\hM$ is Gersten if and only if for every $X\in\Sm_S$, the augmented complex~\eqref{eq:augcompl} is exact.
	In other words, the complex $\underline{C}^*(X,\hM,*)$ is a flasque resolution of the sheaf $\underline{A}^0(X,\hM,*)$.
\end{lm}
\proof
To say that $\underline{C}^*(X,\hM,*)$ is a flasque resolution of $\underline{A}^0(X,\hM,*)$ amounts to say that the complex~\eqref{eq:augcomp} is exact for every $U\in\Sm_X$ local, which is precisely the Gersten property.
\endproof

\begin{paragr}
	Recall that a \textbf{homotopy module} on $S$ is a strictly $\mathbb{A}^1$-invariant Nisnevich sheaf of $\mathbb{Z}$-graded abelian groups $M_*$ on $\Sm_S$, together with isomorphisms 
	\begin{align}
		\label{eq:-1iso}
		(M_n)_{-1}\simeq M_{n-1}.
	\end{align}
	Here for $\mathcal{F}$ a sheaf of abelian groups, we denote by $\mathcal{F}_{-1}$ the sheaf $X\mapsto\operatorname{ker}(\mathcal{F}(X\times\mathbb{G}_m)\xrightarrow{1_X^*}\mathcal{F}(X))$, where $1_X:X\to X\times\mathbb{G}_m$ is the unit section (see \cite[Ch. 2]{MorelLNM}).
	
\end{paragr}

\begin{prop}
	Let $S$ be an excellent regular scheme.

	If $\hM$ is Gersten, then the unramified sheaf $\underline{A}^0(S,\hM,*)$ is a homotopy module on $S$.
\end{prop}
\proof

The strict $\mathbb{A}^1$-invariance of the unramified sheaf $\underline{A}^0(S,\hM,*)$ follows from Lemma~\ref{lm:gerstenex} and homotopy invariance (Theorem~\ref{HomotopyInvariance}) (see also \cite[Thm. 9.4]{Feld1} or \cite[\S4]{BHP22}).
The isomorphism~\eqref{eq:-1iso} for $\underline{A}^0(X,\mathcal{M},*)$ is induced by the boundary map $A^0(X\times\mathbb{G}_m,\hM,*)\xrightarrow{\partial}A^0(X,\hM,*-1)$ associated to the boundary triple $X\xrightarrow{0_X} \mathbb{A}^1_X\leftarrow X\times\mathbb{G}_m$, see the proof of \cite[Th. 4.1.7]{Feld2}.
\endproof

\subsection{Chow-Witt group of a number ring}

According to \Cref{ex:KMW_hlg_MW-module}, we can define
Chow-Witt groups of any excellent scheme, whose grading is obtained
after the choice of a dimension function. When $S$ is regular,
we take the cohomological version, and choose the opposite
of the codimension as a dimension function.

Let $K$ be a number field and $\cO_K$ its ring of integers. Let $X=\Spec{\cO_K}$.
We compute the first Chow-Witt group $\CHW^1(X)$.
\begin{num}\textit{GW-group of finite fields}. \label{GW_group_finite_fields}
	
	Given a field $K$, recall that $I(K)$ is the kernel of the rank map $\GW(K) \rightarrow \ZZ$.
	We let $Q(K)=K^\times/(K^\times)^2$ the quadratic residue class of units of $K$.
	Recall $I(K)/I(K)^2=Q(K)$.
	Consider a finite field $F=\FF_q$ with $q=p^r$ elements. Then from \cite[Lemma 1.5]{MH}, we get:
	$$
	I(F)^2=0 \rightarrow I(F)=I(F)/I(F)^2=Q(F)
	$$
	Observe finally that, as $F^\times$ is cyclic of order $q-1$, we get:
	$$
	Q(\FF_q)=\begin{cases}
		0 & q \text{ even,} \\
		\ZZ/2 & q \text{ odd.}
	\end{cases}
	$$
	Therefore, from the split exact sequence:
	\begin{equation}\label{eq:GW_finite}
		0 \rightarrow Q(\FF_q) \rightarrow \GW(\FF_q) \rightarrow \ZZ \rightarrow 0
	\end{equation}
	we get:
	$$
	\GW(\FF_q)=\begin{cases}
		\ZZ & q \text{ even,} \\
		\ZZ/2 \oplus \ZZ & q \text{ odd.}
	\end{cases}
	$$
	The projection to the second factor is given by the rank map,
	whereas for $q$ odd, $\GW(\FF_q) \rightarrow \ZZ/2$ maps $<a>-1$ to the quadratic residue class of $a \in \FF_q^\times$
	\emph{i.e.} $0$ if $a$ is a square in $\FF_q$, $1$ otherwise.
\end{num}

\begin{num}
	Let $x \in X_{(0)}$, in other words a prime of $\cO_K$. We choose a uniformizing parameter $\pi_x$ of $\cO_{X,x}$.
	We get a residue map
	$$
	\partial_x:\KMW_1(K) \rightarrow \GW(\kappa(x),\nu_x^\vee) \simeq \GW(\kappa(x)), [u] \mapsto n_\epsilon.<\overline{u.\pi_x^{-n}}>
	$$
	where $n$ is the order of $u$ at $x$ \emph{i.e.} $n=v_x(u)$ where $v_x$ is the valuation corresponding to the point $x$.
	
	Then $\CHW^1(X)$ is the cokernel of the map:
	$$
	\KMW_1(K) \xrightarrow{\tdiv=\sum_x \partial_x} \oplus_{x \in X_{(0)}} \GW(\kappa(x)).
	$$
	Now we consider the commutative diagram:
	$$
	\xymatrix@=20pt{
		& 
		K^\times 
		\ar[r]
		\ar[d] ^{\zdiv}
		& 
		\KMW_1(K)
		\ar[r]
		\ar_{\tdiv}[d] 
		& 
		I(K)
		\ar^{}[d]
		\ar[r] 
		&
		0
		\\
		0
		\ar[r]
		&
		Z_0(X)
		\ar[r]_{\times h}
		& 
		\oplus_x \GW(\kappa(x))
		\ar[r]_{\mod h} 
		& 
		\oplus_x I(\kappa(x))
		\ar[r]
		& 0.
	}
	$$
	where each lines are induced by multiplication by the hyperbolic map $h$ followed by projection modulo $h$ (hence are exact), and where $\zdiv$ is the classical divisor map.. The bottom left map denoted by $\times h$ is injective because the residue fields of a number ring are finite (see also §\ref{GW_group_finite_fields}).
	
	Recall we have an exact sequence:
	$$
	0 \rightarrow \cO_K^\times \rightarrow K^\times\xrightarrow{\zdiv} Z_0(X) \rightarrow \CH^1(X) \rightarrow 0.
	$$
	From \cite[IV.3.4]{MH}, we also have an exact sequence:
	$$
	0 \rightarrow I(\cO_K) 
	\xrightarrow{} 
	I(K)
	\xrightarrow{} 
	\oplus_x I(\kappa(x))
	\rightarrow \CH^1(X)/2 \rightarrow 0.
	$$
	Therefore the snake lemma gives a long exact sequence:
	$$
	\KMW_1(\cO_K) 
	\xrightarrow{{(*)}}
	I(\cO_K)
	\to
	\CH^1(X) 
	\rightarrow \CHW^1(\cO_K) 
	\rightarrow 
	\CH^1(\cO_K)/2 
	\rightarrow 0.
	$$
	Thus, we have obtained:
\end{num}

\begin{thm}
	\label{thm_Chow_Witt_number_ring}
	Let $\cO_K$ be a number ring. Then, one has the following exact sequence
	$$
	\CH^1(\cO_K) 
	\rightarrow 
	\CHW^1(\cO_K) 
	\rightarrow 
	\CH^1(\cO_K)/2 
	\rightarrow 0
	$$
	where the first map is induced by the hyperbolic map $h$ and the second one is induced by modding out by $h$.
	In particular, $\CHW^1(X)$ is finite.
\end{thm}

\begin{ex}
	\begin{enumerate}
		\item As soon as the number ring $\cO_K$ is a principal ideal domain, we have
		$$
		\CHW^1(\cO_K)=0.
		$$
		This is the case for quadratic\footnote{Recall
			it is not known if there exists infinitely many real quadratic fields with class number $1$ (a question attributed
			to Gauss).} extensions 
		$$K=\QQ(\sqrt d), d=2,3,5,13,-1,-2,-3,-7, -8.
		$$
		\item As soon as the class number of $K$ is odd, the canonical map: $\CH^1(\cO_K) \xrightarrow{\times h} \CHW^1(\cO_K)$ is
		an isomorphism. Indeed, one can construct a retract by taking the modulo $\eta$ map multiplied by $\frac 1 2$. Examples: $K=\QQ(\sqrt{-23}), \QQ(\sqrt[3]{7})$.
	\end{enumerate}

\end{ex}

\subsection{Derived homotopy $\infty$-category} \label{SubsectionDerivedHomotopyCat}
\begin{num}
	Recall that presentable $\infty$-categories can be localized, a theory that parallels the notion of Bousfield localization of model categories (see \cite[5.5.4]{HTT}). We will need further a condition that ensure that the localization of a presentable $\infty$-category equipped with a monoidal structure admits a canonical monoidal structure. This is provided by \cite[Prop. 2.2.1.9]{HA}) as explained after the following definition.
	
\end{num}

\begin{df} \label{DefCompatible}
	Let $\CC$ be a presentable monoidal $\infty$-category, and $\mathfrak{S}$ a set of maps in $\CC$. We say that $\mathfrak{S}$ is $\otimes$-compatible if, for every $S$-equivalence $f$ and every object $M$, the map $f\otimes M$ is an $\mathfrak{S}$-equivalence.
\end{df}
Proposition 2.2.1.9 of \cite{HA} shows that in that case, there exists a canonical monoidal structure on the localized $\infty$-category $\CC[\mathfrak{S}^{-1}]$ such that the localization functor
\begin{center}
	$L_S:\CC\to \CC[\mathfrak{S}^{-1}]$
\end{center}
is compatible with the monoidal structure (i.e. is part of a morphism of operads over the commutative operad).
\par The $\AA^1$-derived $\infty$-category satisfies a universal property which was previously stated in terms of model categories before (see \cite[§5]{CD19}).
\begin{prop} Let $S$ be a scheme and $R$ a coefficient ring. There exists an initial $R$-linear $\infty$-category $\DA(S,R)$ equipped with a monoidal $\infty$-functor
	\begin{center}
		$M:\Sm_S \to \DA(S,R)$
	\end{center}
	such that:
	\begin{enumerate}
		\item the map $p_*:M(\AA^1_S)\to M(S)$ induced by the projection is an equivalence;
		\item the cokernel $\un_S(1)[2]$ of the split map $M(\{\infty\})\to M(\PP^1_S)$ is $\otimes$-invertible;
		\item for any Nisnevich hyper-cover $p:W_{\bullet}\to X$ of a smooth $S$-scheme $X$, the induced map
		\begin{center}
			$p_*:M(W_{\bullet})\to M(X)$
		\end{center}
		is an equivalence, where we have used the canonical left Kan extension of the functor $M$ to the $\infty$-category of simplicial schemes using the fact that $\DA(S,R)$ admits colimits.
	\end{enumerate}
\end{prop}
\begin{proof} The existence of the category follow from the classical model theoretic construction after Morel and Voevodsky (see \cite[§5]{CD19}) and the associated (monoidal) $\infty$-category. However, it is now easier to use the $\infty$-categorical construction, and necessary to get the universal property. The following material is folklore (see also \cite{Drew18}). We include it for completeness. We follow the approach of \cite{Rob15}. We start with the $\infty$-category of $\infty$-functors
	\begin{center}
		$\PSh(S,R):=\PSh(\Sm_S,D(R))$
	\end{center}
	equipped with its natural monoidal structure (see \cite[4.8.1.12, 4.8.1.13]{HA}). The Yoneda embedding gives us the functor:
	\begin{center}		$M:\Sm_S\to \PSh(S,R)$,
	\end{center}
	that we extend to simplicial smooth schemes
	\begin{center}
		$M:\Delta^{op}\Sm_S\to \PSh(S,R)$
	\end{center}
	using the existence of colimit in $\PSh(S,R)$. Then we localize the (presentable) monoidal $\infty$-category $PSh(S,R)$ with respect to the set of maps $\mathfrak{S}$ of the form:
	\begin{enumerate}
		\item $M(\AA^1_X)\to M(X)$;
		\item $M(W_{\bullet}) \to M(X)$ for any Nisnevich hyper-cover $W_{\bullet} \to X$ of a smooth $S$-scheme $X$.
	\end{enumerate}
	The set $\mathfrak{S}$ is $\otimes$-compatible (see \ref{DefCompatible}, as follows from the classical theory), so the resulting localized $\infty$-category admits a monoidal structure. We denote the latter by $\DA^{eff}(S,R)$ and call it the {\em effective localized $\AA^1$-derived Nisnevich $\infty$-category} with $R$-coefficients.
	\par Finally, we define $\DA(S,R)$ as the formal inversion of the Tate twist $\un_S(1)$ in the presentable monoidal $\infty$-category $\DA^{eff}(S,R)$ as defined in \cite[Cor. 4.25]{Rob15}. The universal property of the composite monoidal functor
	\begin{center}
		$M:\Sm_S\to \DA^{eff}(S,R) \to \DA(S,R)$
	\end{center}
	follows from the construction as in \cite[Cor. 5.11]{Rob13}.
\end{proof}

\begin{rem}
	By construction, the $\infty$-category $\DA^{eff}(S,R)$ is equivalent to the $\infty$-category associated with the $\AA^1$-localization of the Nisnevich-local model category structure on the category of complexes $C(\Sm_S,R)$ of (pre)sheaves of $R$-modules on $\Sm_S$ as constructed in \cite[§5]{CD3}.
	\par Then, taking into account \cite[Th. 4.29]{Rob13}, this implies that the $\infty$-category $\DA(S,R)$ is equivalent to the $\infty$-category associated with the $\PP^1$-stable, Nisnevich-local and $\AA^1$-local model structure on the category $C(\Sm_S,R)$.
\end{rem}

\begin{rem}
	In practice, we have $S=\Spec k$ and $R=\ZZ$, and denote by $\DA(k)=\DA(S,R)$.
\end{rem}

\begin{rem}
	\label{def_derived_cat_with_specializations}
	Replacing the category $\Sm_k$ of smooth schemes by the category $\Cortilde_k$ in the previous construction, we obtain an $\infty$-category which corresponds to the Milnor-Witt derived homotopy category $\widetilde{\mathbf{DM}}_k$ defined in \cite[Ch. 3]{BCDFO}. The same holds for the category of smooth schemes with specializations (see \ref{notation_derived_cat_with_specializations}).
\end{rem}

\bibliographystyle{amsalpha}
\bibliography{MW}

\end{document}